\documentclass[3p,times]{elsarticle}
\usepackage[dvipsnames]{xcolor}
\usepackage{graphicx}
\usepackage{amsmath,amsthm,amsbsy,amssymb,enumerate,latexsym}
\usepackage{calligra,calrsfs,mathrsfs}
\usepackage{multirow}
\usepackage{ulem,cancel}
\usepackage{cases}
\usepackage{pgfplots}
\pgfplotsset{compat=1.18}
\usepackage{subcaption}
\usepackage{tikz}
\usepackage{tkz-euclide}
\usetikzlibrary[patterns]
\usepackage{hyperref}
\usepackage{algorithm}
\usepackage{algpseudocode}
\usepackage{dsfont}
\usepackage{caption}
\usepackage{epstopdf}
\usepackage{epsfig}
\usepackage{array}
\newcolumntype{P}[1]{>{\centering\arraybackslash}p{#1}}
\usepackage{enumerate}
\usepackage{float}
\usepackage[T1]{fontenc}
\usepackage[latin9]{inputenc}
\usepackage{geometry}
\usepgfplotslibrary{fillbetween}
\usetikzlibrary{patterns}
\usepackage{tikz-dimline}
\usepackage{color}
\usepackage{stmaryrd}
\usepackage{setspace}
\usepackage{amsthm}
\usepackage{textcomp}
\usepackage{hyperref}%,showkeys}
\usepackage{soul,xcolor}
\usepackage{multirow}
\usepackage{ulem}
\usepackage[inline]{enumitem}
\usepackage{lmodern}
\biboptions{sort&compress}
\allowdisplaybreaks

\newcommand{\nobracket}{}

\newcommand{\bs}[1]{\ensuremath{\boldsymbol{#1}}}
\newcommand{\tmop}[1]{\ensuremath{\operatorname{#1}}}

\renewenvironment{proof}{\noindent\textbf{Proof\ }}{\hspace*{\fill}$\Box$\medskip}

\newcommand{\ifcomment}{\iffalse}

\newcommand{\ssf}[1]{{\mbox{\sffamily #1}}}

\newcommand{\jumpb}[1]{[\![#1]\!]}

\newcommand{\avg  }[1]{\langle \, #1 \, \rangle}

\newcommand{\avgb}[1]{\{\!\!\{ #1 \}\!\!\}}    %\newcommand{\avgb}[1]{\left\{\!\left\{ #1 \right\}\!\right\}}
\newcommand{\ti}[1]{\tilde{#1}}
\newcommand{\G}{\Gamma}

\newcommand{\Om}{\Omega}
\newcommand{\om}{\omega}

\newcommand{\cT}{{\mathcal T}}
\newcommand{\tO}{\tilde{\Omega}_h}

\newcommand{\tG}{\tilde{\Gamma}_h}

\newcommand{\tGD}{\tilde{\Gamma}_{h;D}}
\newcommand{\tGN}{\tilde{\Gamma}_{h;N}}
\newcommand{\GD}{\Gamma_{D;h}}
\newcommand{\GN}{\Gamma_{N;h}}

\newcommand{\tx}{{\tilde{\bs{x}}}}

\newcommand{\oS}{\ssf{S}_{\bs{d}}}
\newcommand{\oSh}{\ssf{S}_{1/2\bs{d}}}
\newcommand{\oE}{\ssf{E}}
\newtheorem{thm}{Theorem}

\newtheorem{lemma}{Lemma}
\newdefinition{rem}{Remark}
\newtheorem{assumption}{Assumption}

\newcolumntype{M}[1]{>{\centering\arraybackslash}m{#1}}

\newcommand{\ReverseLogLogSlopeTriangle}[6]
{
	\pgfplotsextra
	{
		\pgfkeysgetvalue{/pgfplots/xmin}{\xmin}
		\pgfkeysgetvalue{/pgfplots/xmax}{\xmax}
		\pgfkeysgetvalue{/pgfplots/ymin}{\ymin}
		\pgfkeysgetvalue{/pgfplots/ymax}{\ymax}
		
		\pgfmathsetmacro{\xArel}{#1}
		\pgfmathsetmacro{\yArel}{#3}
		\pgfmathsetmacro{\xBrel}{#1-#2}
		\pgfmathsetmacro{\yBrel}{\yArel}
		\pgfmathsetmacro{\xCrel}{\xBrel}
		\pgfmathsetmacro{\lnxB}{\xmin*(1-(#1-#2))+\xmax*(#1-#2)} % in [xmin,xmax].
		\pgfmathsetmacro{\lnxA}{\xmin*(1-#1)+\xmax*#1} % in [xmin,xmax].
		\pgfmathsetmacro{\lnyA}{\ymin*(1-#3)+\ymax*#3} % in [ymin,ymax].
		\pgfmathsetmacro{\lnyC}{\lnyA-#4*(\lnxA-\lnxB)}
		\pgfmathsetmacro{\yCrel}{\lnyC-\ymin)/(\ymax-\ymin)} % THE IMPROVED EXPRESSION WITHOUT 'DIMENSION TOO LARGE' ERROR.
		
		\coordinate (A) at (rel axis cs:\xArel,\yArel);
		\coordinate (B) at (rel axis cs:\xBrel,\yBrel);
		\coordinate (C) at (rel axis cs:\xCrel,\yCrel);
		
		\draw[#5,line width = #6 mm]     (A)-- node[pos=0.5,anchor=south] {1}
		(B)-- node[pos=0.5, anchor = east] {#4}	
		(C)-- 
		cycle;
	}
}

\newcommand{\ReverseLogLogSlopeTriangleN}[6]
{
	\pgfplotsextra
	{
		\pgfkeysgetvalue{/pgfplots/xmin}{\xmin}
		\pgfkeysgetvalue{/pgfplots/xmax}{\xmax}
		\pgfkeysgetvalue{/pgfplots/ymin}{\ymin}
		\pgfkeysgetvalue{/pgfplots/ymax}{\ymax}
		
		\pgfmathsetmacro{\xArel}{#1}
		\pgfmathsetmacro{\yArel}{#3}
		\pgfmathsetmacro{\xBrel}{#1-#2}
		\pgfmathsetmacro{\yBrel}{\yArel}
		\pgfmathsetmacro{\xCrel}{\xBrel}
		\pgfmathsetmacro{\lnxB}{\xmin*(1-(#1-#2))+\xmax*(#1-#2)} % in [xmin,xmax].
		\pgfmathsetmacro{\lnxA}{\xmin*(1-#1)+\xmax*#1} % in [xmin,xmax].
		\pgfmathsetmacro{\lnyA}{\ymin*(1-#3)+\ymax*#3} % in [ymin,ymax].
		\pgfmathsetmacro{\lnyC}{\lnyA-#4*(\lnxA-\lnxB)}
		\pgfmathsetmacro{\yCrel}{\lnyC-\ymin)/(\ymax-\ymin)} % THE IMPROVED EXPRESSION WITHOUT 'DIMENSION TOO LARGE' ERROR.
		
		\coordinate (A) at (rel axis cs:\xArel,\yArel);
		\coordinate (B) at (rel axis cs:\xBrel,\yBrel);
		\coordinate (C) at (rel axis cs:\xCrel,\yCrel);
		
		\draw[#5,line width = #6 mm]     (A)-- node[pos=0.5,anchor=north] {1}
		(B)-- node[pos=0.5, anchor = east] {#4}	
		(C)-- 
		cycle;
	}
}

\makeatletter
\def\ps@pprintTitle{%
   \let\@oddhead\@empty
   \let\@evenhead\@empty
   \let\@oddfoot\@empty
   \let\@evenfoot\@oddfoot
}
\makeatother

\begin{document}
	
\begin{frontmatter}

  \title{Gap-SBM: A New Conceptualization of the Shifted Boundary Method \\
  with Optimal Convergence for the Neumann and Dirichlet Problems
  % \title{Optimal Neumann and Dirichlet Boundary Conditions \\
  % with a New Conceptualization of the Shifted Boundary Method
  %A Simple Neumann Boundary Condition for The Shifted Boundary Method: Conceptualization, Implementation, and Numerical Analysis.
  }

	\author[duke]{J. Haydel Collins}
	\ead{jhc63@duke.edu}
      	\author[PSU]{Kangan Li}
	\ead{kbl5610@psu.edu}
	\author[franchecomte]{Alexei Lozinski}
	\ead{alexei.lozinski@univ-fcomte.fr}
	\author[duke]{Guglielmo Scovazzi}
	\ead{guglielmo.scovazzi@duke.edu}
	\address[duke]{Department of Civil and Environmental Engineering, Duke University, Durham, North Carolina 27708, USA}
	\address[franchecomte]{Universit\' e Marie et Louis Pasteur, CNRS, LmB (UMR 6623), F-25000 Besan\c con, France}
      	\address[PSU]{Energy and Mineral Engineering Department, The Pennsylvania State University, University Park, PA, 16802}

	%\cortext[ca]{\color{red} Alexei Lozinski}

	\begin{abstract} % abstract
	We propose and mathematically analyze a new Shifted Boundary Method for the treatment of Dirichlet and Neumann boundary conditions, with provable optimal accuracy in the $L^2$- and $H^1$-norms of the error. 
	The proposed method is built on three stages. 
	First, the distance map between the SBM surrogate boundary and the true boundary is used to construct an approximation to the geometry of the gap between the two. 
	Then, the representations of the numerical solution and test functions are extended from the surrogate domain to the such gap. Finally, approximate quadrature formulas and specific shift operators are applied to integrate a variational formulation that also involves the fields extended in the gap.
	An extensive set of two-dimensional tests demonstrates the theoretical findings and the overall optimal performance of the proposed method.
	\end{abstract}

	\end{frontmatter}

\section{Introduction \label{sec:intro} }

This article explores a reimagined approach for the treatment of Dirichlet and Neumann boundary conditions with the Shifted Boundary Method. Although this approach applies to both, the major impetus for this topic stemmed from the pursuit of optimal convergence rates for Neumann conditions in a primal formulation. What follows is a general overview of motivations, historical background, and key ideas that will be explored in depth throughout the present article.

Proper enforcement of Neumann boundary conditions is critical across a wide range of engineering and applied science applications, particularly in computational solid mechanics and heat transfer, where traction and flux boundary conditions are ubiquitous. The analyses performed in these areas rely on computational representations of increasingly complex geometry to the extent that traditional conforming Finite Element Methods may no longer be sufficiently practical. These limitations are further exacerbated in situations where meshing must be performed iteratively, as in geometrically demanding scenarios like shape optimization, digital twins, and additive manufacturing.

Addressing the demand for less labor-intensive design algorithms, recent advancements have been made in immersed (or embedded, or unfitted) computational methods. In the last two decades, much attention has been given on reducing the design cost for problems involving complex geometrical features, described in standard formats (i.e., CAD) and non-standard formats (i.e., STL, level sets, etc.). Indeed, immersed/embedded/unfitted methods have shown the potential to drastically reduce the pre-processing time involved in the acquisition of the geometry and the generation of the computational grid. An (incomplete) list of these developments in the context of finite element methods include the Immersed Boundary Finite Element Method (IB-FEM)~\cite{boffi2003finite,zhang2004immersed}, the cutFEM~\cite{badia2018aggregated,hansbo2002unfitted,hollig2003finite,hollig2001weighted,ruberg2012subdivision,ruberg2014fixed,schott2015face,burman2018cut,burman2019dirichlet,burman2017cut,burman2010fictitious,burman2012fictitious,burman2014unfitted,burman2018shape,massing2015nitsche,burman2015cutfem,kamensky2017immersogeometric,xu2016tetrahedral,lozinski2019nocut}, the Finite Cell Method~\cite{parvizian2007finite,duster2008finite,hollig2003finite,hollig2001weighted}, Immerso-Geometric Analysis~\cite{kamensky2017immersogeometric}, B-spline immersed methods~\cite{ruberg2012subdivision,ruberg2014fixed} and similar earlier methods. Many of these approaches require the geometric construction of the partial elements cut by the embedded boundary (cut-cells) to form the solution space, and typically employ Nitsche's method for consistent weak boundary enforcement.

CutFEM relies on data structures that are considerably more tedious to implement with respect to corresponding fitted finite element methods. Furthermore, integrating the variational forms on the characteristically irregular cut cells may also be difficult and advanced quadrature formulas might need to be employed~\cite{parvizian2007finite,duster2008finite}. Additionally, small-cut cells can induce poor matrix conditioning and even numerical instabilities, which need to be addressed with appropriate stabilization operators~\cite{burman2010ghost,burman2014fictitious} or element aggregation techniques~\cite{badia2018aggregated}.

The Shifted Boundary Method was proposed as an alternative unfitted method that remedies the so called ``small cut cell problem'' by removing cut cells entirely. Instead, the location where boundary conditions are applied is \textit{shifted} from the true to an approximate (surrogate) boundary composed of facets belonging to fully intact cells (the surrogate domain). This shift can be thought of as a modification/correction of the boundary conditions by way of Taylor expansions in the direction of the closest distance between corresponding points on the true and surrogate boundaries. These shifted conditions are enforced weakly, using Nitsche's method, leading to a relatively simple, robust, accurate, and efficient algorithm. Indeed, the computational infrastructure based on distances instead of cut cells can be less algorithmically burdensome in regards to implementation.

The Shifted Boundary Method (SBM) was introduced in~\cite{main2018shifted0} and belongs to the more specific class of approximate domain methods~\cite{bramble1972projection,bramble1996finite,bramble1994robust,cockburn2012solving,cockburn2014priori,cockburn2014solving,bertoluzza2005fat,bertoluzza2011analysis,glowinski1994fictitious,lozinski2016new}, along with $\phi$-FEM~\cite{duprez2020phi,duprez2022immersed,cotin2022varphi,duprez2023phi,duprez2023phi2,duprez2023new} albeit with some key differences. The work in Main et al.~\cite{main2018shifted0} demonstrated the viability of the SBM for Poisson and Stokes flow problems. Soon after, the method was generalized in~\cite{main2018shifted} to the advection-diffusion and Navier-Stokes equations, and later to hyperbolic conservation laws in~\cite{song2018shifted}. An analysis of the stability and accuracy of the SBM for the Poisson, advection-diffusion, and Stokes operators was also included in~\cite{main2018shifted0,main2018shifted,atallah2020analysis},respectively.  A high-order version of the SBM was proposed in~\cite{atallah2022high}, applications to solid and fracture mechanics problems were presented in~\cite{liu2020shift,atallah2021shifted,li2021shifted,li2023blended,li2021shiftedsimple} and simulations of static and moving interfaces were developed  in~\cite{li2020shifted,colomes2021weighted}. Most recently, the SBM was extended to contact problems in solid mechanics in Li et al.~\cite{li2025contact}

Until now, Dirichlet and Neumann conditions were handled differently due to the certain challenges posed by the limitations of the Taylor expansions. That is to say, the solution and its gradient are available within piecewise-linear interpolation spaces (P1), however, the higher-order terms in the Taylor expansion are not. This means that Neumann boundary conditions, if shifted naively, will result in a loss of one order of convergence in $L^2$. Earlier work in Atallah et al.~\cite{atallah2021solid} addressed this challenge via a mixed shifted formulation strategy (e.g. solving for strains and displacements in solid mechanics). Although mixed formulations increase computational cost, it was shown that it is only necessary to solve the mixed variation along the strip of elements adjacent to the Neumann boundary while enforcing continuity of the normal stress component between the primal and mixed domains. This approach successfully recovered second order accuracy for Neumann Boundary conditions with only a meager increase in computational burden. 

Naturally, the potential for a SBM that accurately enforces Neumann boundary conditions without a mixed formulation remained an open research question. This current work demonstrates the viability of a newly conceptualized SBM that optimally enforces both Neumann and Dirichlet boundary conditions, which requires no mixed formulation strategy, and retains the classic Taylor expansion shift operator paradigm. 

The key idea in the proposed SBM variant, named here Gap-SBM, is to construct a geometric approximation to the gap between the surrogate and true boundaries, and then devise approximate quadrature formulas to integrate a modified variational formulation, which includes the extension of the solution and test functions from the surrogate domain to the approximate true boundary.

Although the integration of the variational form in the gap between surrogate and true boundaries is a prerogative of cutFEM approaches, the proposed method remains conceptually an SBM, because the construction of the approximation to the gap geometry is done via distance vectors/maps, and the numerical integration is performed on the surrogate boundary using special quadratures that do not involve any cut cells. Because these geometric constructions and integration formulas do not involve cut cells, the method is inherently of SBM type.
Specifically, the proposed method needs only the finite element infrastructure that already exists on the original background mesh: no additional degrees of freedom, cut-cells, or ghost penalization are needed.

In the following derivations, analysis, and numerical experiments, the proposed method is shown to be stable and have optimal error convergence rates (in the $L^2$- and $H^1$-norm of the error). We emphasize that, to the best of our knowledge, it is for the first time in the literature that optimal convergence in the $L^2$ norm can be proven theoretically for a method acting on unfitted grids without using advanced quadrature formulas on cut cells. Indeed, previous analyses of different SBM and $\phi$-FEM variants had to contend with a half-order sub-optimality, which was never observed in practice, but was persistent theoretically. Here, the fully optimal convergence is proven, at least for the symmetric variant of the Gap-SBM.

The rest of this article is organized as follows: Section~\ref{sec:sbm_intro} introduces the SBM notation, Section~\ref{sec:newSBM} derives the new proposed SBM for the Poisson problem with Dirichlet and Neumann conditions, Section~\ref{sec:theo} derives stability results and error estimates in the $H^1$- and $L^2$-norms, Section~\ref{sec:sbm_linela} extends the proposed method to the equations of compressible isotropic linear elasticity, and Section~\ref{sec:numerical_results} demonstrates the optimal convergence of the error in a series of numerical tests.

\section{Preliminaries on the Shifted Boundary Method \label{sec:sbm_intro}}
This section introduces the notation and general strategy of the Shifted Boundary Method (SBM). In Section~\ref{sec:newSBM}, we will describe the specific details of the new SBM pursued in this work. 
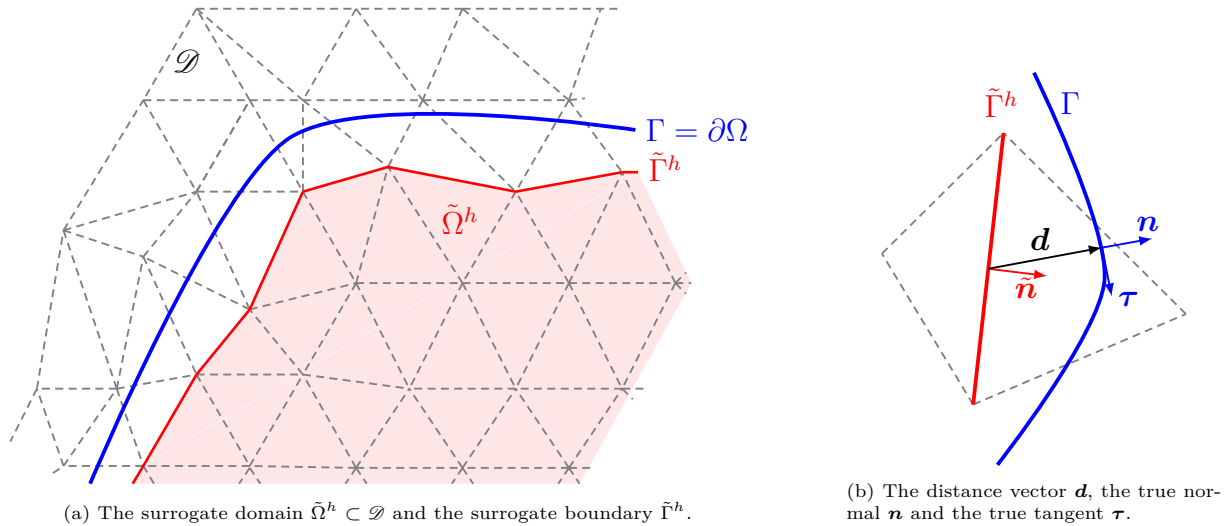
\begin{figure}
\centering
\begin{subfigure}[b]{.6\textwidth}\centering
\begin{tikzpicture}[scale=0.7]

%%% fill 
\draw [black, draw=none,name path=surr] plot coordinates {  (7.3,2.1) (5,1.73205) (2.6,2.2) (1,1.73205) (0,-0.5) (-1,-1.73205) (-2,-3.4641) };
% \draw [black, draw=none,name path=layer] plot coordinates {(7.3,2.1)  (6.3,3.5) (6,3.4641) (0,3.4641) (-1,1.73205) (-2,0.5) (-2.5,-2) (-3.5,-3.4641) (-3.7, -3.6) (-2.2, -3.8) (-2,-3.4641) };
\draw [black, draw=none,name path=extra] plot coordinates {(7.3,2.1) (8.3,0.0) (6.2,-3.8) (-2.2, -3.8) (-2,-3.4641)};
\tikzfillbetween[of = surr and extra,split=false]{red!10!};

%%% first line of elements
\draw[line width = 0.25mm,densely dashed,gray] (4,5.19615) -- (7,5.19615);
\draw[line width = 0.25mm,densely dashed,gray] (4,5.19615) -- (6,3.4641);
\draw[line width = 0.25mm,densely dashed,gray] (4,5.19615) -- (3.25,3.4641);
\draw[line width = 0.25mm,densely dashed,gray] (4,5.19615) -- (2,5.19615);
\draw[line width = 0.25mm,densely dashed,gray] (2,5.19615) -- (3.25,3.4641);
\draw[line width = 0.25mm,densely dashed,gray] (2,5.19615) -- (1,3.4641);
\draw[line width = 0.25mm,densely dashed,gray] (-1,5.19615) -- (2,5.19615) ;
\draw[line width = 0.25mm,densely dashed,gray] (-1,5.19615) -- (1,3.4641) ;
\draw[line width = 0.25mm,densely dashed,gray] (-1,5.19615) -- (-2,3.4641) ;
\draw[line width = 0.25mm,densely dashed,gray] (0,3.4641) -- (-1,5.19615);
\draw[line width = 0.25mm,densely dashed,gray] (-1,1.73205) -- (0,3.4641);
\draw[line width = 0.25mm,densely dashed,gray] (1,1.73205) -- (2,0);
\draw[line width = 0.25mm,densely dashed,gray] (6,3.4641) -- (6.6,5.2);

%%% second line of elements 
\draw[line width = 0.25mm,densely dashed,gray] (0,-0.5) -- (2,0);
\draw[line width = 0.25mm,densely dashed,gray] (2,0) -- (2.6,2.2);
\draw[line width = 0.25mm,densely dashed,gray] (2,0) -- (4,0);
\draw[line width = 0.25mm,densely dashed,gray] (4,0) -- (2.6,2.2);
\draw[line width = 0.25mm,densely dashed,gray] (5,1.73205) -- (4,0);
\draw[line width = 0.25mm,densely dashed,gray] (4,0) -- (6,0);
\draw[line width = 0.25mm,densely dashed,gray] (6,0) -- (5,1.73205);
\draw[line width = 0.25mm,densely dashed,gray] (6,0) -- (7,2.1);
\draw[line width = 0.25mm,densely dashed,gray] (6,0) -- (8,0);
\draw[line width = 0.25mm,densely dashed,gray] (8,0) -- (7,2.1);
\draw[line width = 0.25mm,densely dashed,gray]  (-3.5,1) --  (-2,3.4641);
\draw[line width = 0.25mm,densely dashed,gray]  (-2,3.4641) --  (0,3.4641);
\draw[line width = 0.25mm,densely dashed,gray]  (-2,3.4641) --  (-1,1.73205);

%%% third line of elements 
\draw[line width = 0.25mm,densely dashed,gray] (-1,-1.73205) -- (1,-1.73205);
\draw[line width = 0.25mm,densely dashed,gray] (2,0) -- (1,-1.73205);
\draw[line width = 0.25mm,densely dashed,gray] (1,-1.73205) -- (0,-0.5);
\draw[line width = 0.25mm,densely dashed,gray]  (-4,-2) -- (-3.5,1);
\draw[line width = 0.25mm,densely dashed,gray]  (-4,-2) -- (-4.5,-3);
\draw[line width = 0.25mm,densely dashed,gray]  (-3.5,1) --  (-2.5,-2);
\draw[line width = 0.25mm,densely dashed,gray]  (-3.5,1) --  (-2,0.5);
\draw[line width = 0.25mm,densely dashed,gray]  (-3.5,1) --  (-1,1.73205);

%%% fourth line of elements 
\draw[line width = 0.25mm,densely dashed,gray] (-2.5,-2) -- (-4,-2);
\draw[line width = 0.25mm,densely dashed,gray] (-3.5,-3.4641) -- (-4,-2);
\draw[line width = 0.25mm,densely dashed,gray] (0,-3.4641) -- (-2,-3.4641);
\draw[line width = 0.25mm,densely dashed,gray]  (-1,-1.73205) -- (0,-3.4641);
\draw[line width = 0.25mm,densely dashed,gray] (0,-3.4641) -- (1,-1.73205);

%%% fifth (patch) line of elements 
\draw[line width = 0.25mm,densely dashed,gray]  (8,0) -- (7,-2);
\draw[line width = 0.25mm,densely dashed,gray]  (6,0) -- (7,-2);
\draw[line width = 0.25mm,densely dashed,gray]  (6,0) -- (5,-2);
\draw[line width = 0.25mm,densely dashed,gray]  (4,0) -- (5,-2);
\draw[line width = 0.25mm,densely dashed,gray]  (4,0) -- (3,-2);
\draw[line width = 0.25mm,densely dashed,gray]  (2,0) -- (3,-2);
\draw[line width = 0.25mm,densely dashed,gray] (1,-1.73205) -- (3,-2);
\draw[line width = 0.25mm,densely dashed,gray] (7,-2) -- (5,-2);
\draw[line width = 0.25mm,densely dashed,gray] (5,-2) -- (3,-2);
\draw[line width = 0.25mm,densely dashed,gray]  (6,-3.5) -- (7,-2);
\draw[line width = 0.25mm,densely dashed,gray]  (6,-3.5) -- (5,-2);
\draw[line width = 0.25mm,densely dashed,gray]  (4,-3.5) -- (5,-2);
\draw[line width = 0.25mm,densely dashed,gray]  (4,-3.5) -- (3,-2);
\draw[line width = 0.25mm,densely dashed,gray]  (2,-3.5) -- (3,-2);
\draw[line width = 0.25mm,densely dashed,gray]  (2,-3.5) -- (1,-1.73205);
\draw[line width = 0.25mm,densely dashed,gray]  (2,-3.5) -- (0,-3.4641);
\draw[line width = 0.25mm,densely dashed,gray]  (2,-3.5) -- (4,-3.5);
\draw[line width = 0.25mm,densely dashed,gray]  (6,-3.5) -- (4,-3.5);

\draw[line width = 0.25mm,densely dashed,gray]  (0,-3.5) -- (0.2,-3.8);
\draw[line width = 0.25mm,densely dashed,gray]  (0,-3.5) -- (-0.2,-3.8);
\draw[line width = 0.25mm,densely dashed,gray]  (2,-3.5) -- (2.2,-3.8);
\draw[line width = 0.25mm,densely dashed,gray]  (2,-3.5) -- (1.8,-3.8);
\draw[line width = 0.25mm,densely dashed,gray]  (4,-3.5) -- (4.2,-3.8);
\draw[line width = 0.25mm,densely dashed,gray]  (4,-3.5) -- (3.8,-3.8);
\draw[line width = 0.25mm,densely dashed,gray]  (6,-3.5) -- (6.2,-3.8);
\draw[line width = 0.25mm,densely dashed,gray]  (6,-3.5) -- (5.8,-3.8);
\draw[line width = 0.25mm,densely dashed,gray]  (6,-3.5) -- (6.4,-3.5);

\draw[line width = 0.25mm,densely dashed,gray] (7,-2) -- (7.4,-2.0);
\draw[line width = 0.25mm,densely dashed,gray] (7,-2) -- (7.2,-2.2);
\draw[line width = 0.25mm,densely dashed,gray] (8,0) -- (8.3,0.0);
\draw[line width = 0.25mm,densely dashed,gray] (8,0) -- (8.2,0.2);
\draw[line width = 0.25mm,densely dashed,gray] (8,0) -- (8.2,-0.2);
\draw[line width = 0.25mm,densely dashed,gray] (-1.8, -3.8) -- (-2,-3.4641);

\draw[line width = 0.25mm,densely dashed,gray] (-3.7, -3.6) -- (-3.5,-3.4641);
\draw[line width = 0.25mm,densely dashed,gray] (-3.5,-3.4641) -- (-2.5,-2);
\draw[line width = 0.25mm,densely dashed,gray ] (-2.5,-2) -- (-2,0.5);
\draw[line width = 0.25mm,densely dashed,gray] (-2,0.5) -- (-1,1.73205);
\draw[line width = 0.25mm,densely dashed,gray] (-1,1.73205) -- (0,3.4641);
\draw[line width = 0.25mm,densely dashed,gray] (0,3.4641) -- (6,3.4641);
\draw[line width = 0.25mm,densely dashed,gray] (6,3.4641) -- (6.3,3.5);

\draw[line width = 0.25mm,densely dashed,gray] (0,3.4641) -- (1,1.73205);
\draw[line width = 0.25mm,densely dashed,gray] (1,1.73205) -- (1,3.4641);
\draw[line width = 0.25mm,densely dashed,gray] (1,3.4641) -- (2.6,2.2);
\draw[line width = 0.25mm,densely dashed,gray] (3.25,3.4641) -- (2.6,2.2);
\draw[line width = 0.25mm,densely dashed,gray] (3.25,3.4641) -- (5,1.73205);
\draw[line width = 0.25mm,densely dashed,gray] (6,3.4641) -- (5,1.73205);
\draw[line width = 0.25mm,densely dashed,gray] (6,3.4641) -- (7,2.1);
\draw[line width = 0.25mm,densely dashed,gray] (0,-0.5) -- (-2,0.5);
\draw[line width = 0.25mm,densely dashed,gray] (-2,0.5) -- (-1,-1.73205);
\draw[line width = 0.25mm,densely dashed,gray] (-2.5,-2) -- (-1,-1.73205);
\draw[line width = 0.25mm,densely dashed,gray] (-2.5,-2) -- (-2,-3.4641);
\draw[line width = 0.25mm,densely dashed,gray] (0,-0.5) -- (-1,1.73205);
\draw[line width = 0.25mm,densely dashed,gray] (-1,1.73205) -- (1,1.73205);
\draw[line width = 0.25mm,densely dashed,gray] (-2,-3.4641) -- (-3.5,-3.4641);

%%%% True boundary
\draw [line width = 0.5mm,blue, name path=true] plot[smooth] coordinates {(-3.00,-3.8) (0.75,2.75) (7.25,2.9)};

%%%% Surrogate boundary (tilde)
\draw[line width = 0.35mm,red] (7,2.1) -- (7.3,2.1);
\draw[line width = 0.35mm,red] (5,1.73205) --  (7,2.1);
\draw[line width = 0.35mm,red] (1,1.73205) -- (2.6,2.2);
\draw[line width = 0.35mm,red] (2.6,2.2) -- (5,1.73205);
\draw[line width = 0.35mm,red] (1,1.73205) -- (0,-0.5);
\draw[line width = 0.35mm,red] (0,-0.5) -- (-1,-1.73205);
\draw[line width = 0.35mm,red] (-1,-1.73205) -- (-2,-3.4641);
\draw[line width = 0.35mm,red] (-2.2, -3.8) -- (-2,-3.4641);

%%%% labels
\node[text width=0.5cm] at (7.85,2.20) {\large${\color{red}\ti{\G}^{h}}$};
\node[text width=1.2cm] at (8.3,2.9) {\large${\color{blue}\G=\partial \Om}$};
\node[text width=3cm] at (5.75,1.20) {\large${\color{red}\ti{\Om}^h}$};
\node[text width=3cm] at (0.7,4.2) {\large$\mathcal{D}$};
\end{tikzpicture}
\caption{The surrogate domain $\ti{\Om}^h \subset {\cal D}$ and the surrogate boundary $\ti{\G}^h$.}
\label{fig:SBM}
\end{subfigure}
\hspace{1cm}
\begin{subfigure}[b]{.3\textwidth}\centering
	\begin{tikzpicture}[scale=0.8]
%%% second line of elements 
\draw[line width = 0.25mm,densely dashed,gray] (0,0.5) -- (-1.5,3);
\draw[line width = 0.25mm,densely dashed,gray] (-1.5,3) -- (0.5,5);
\draw[line width = 0.25mm,densely dashed,gray] (0,0.5) -- (3.5,2);
\draw[line width = 0.25mm,densely dashed,gray] (3.5,2) -- (0.5,5);

%%%% True boundary
\draw [line width = 0.5mm,blue, name path=true] plot[smooth] coordinates {(0.4,-0.5) (2.16,2.55) (1.0,6)};
%%%% Surrogate boundary
\draw[line width = 0.5mm,red] (0,0.5) -- (0.5,5);
%% labels
\node[text width=0.5cm] at (0.5,5.5) {\large${\color{red}\ti{\G}^h}$};
\node[text width=0.5cm] at (1.75,5.5) {\large${\color{blue}\G}$};
\node[text width=0.5cm] at (1.25,3.25) {\large$\bs{d}$};
\node[text width=0.5cm] at (3,3.5) {\color{blue} \large$\bs{n}$};
\node[text width=0.5cm] at (1.0,2.43) {\color{red} \large$\ti{\bs{n}}$};
\node[text width=0.5cm] at (2.7,2.25) {\color{blue} \large$\bs{\tau}$};
%% arrows
\draw[->,red, line width = 0.25mm,-latex] (0.25,2.75) -- (1.22,2.63);
\draw[->,line width = 0.25mm,-latex] (0.25,2.75) -- (2.12,3.1);
\draw[->, blue, line width = 0.25mm,-latex] (2.12,3.1) -- (2.28,2.29);
\draw[->, blue, line width = 0.25mm,-latex] (2.12,3.1) -- (2.95,3.25);
\end{tikzpicture}
    \caption{The distance vector $\bs{d}$, the true normal $\bs{n}$ and the true tangent $\bs{\tau}$.}
    \label{fig:ntd}
\end{subfigure}
\caption{The surrogate domain, its boundary, and the distance vector $\bs{d}$.}
\label{fig:surrogates}
\end{figure}

\subsection{Surrogate domains and boundaries}
\label{sec:surr_dom_bndry}

Let $\Om$ be a connected open set in $\mathbb{R}^{2}$ with Lipschitz boundary $\G = \partial \Om$ and let $\bs{n}$ be the outer-pointing normal to $\G$. We consider a closed domain ${\cal D}$ such that $\text{clos}(\Om) \subseteq {\cal D}$ and we introduce a family $\cT_h$ of admissible and shape-regular, quasi-uniform tessellations (i.e., grids, or meshes) of ${\cal D}$.
We will indicate by $h_T$ the size of element $T \in \cT_h$ and by $h$ the piecewise constant function such that $h_{|T}=h_T$.
In the numerical experiments, we will consider tessellations that are either triangular or Cartesian.
For triangular grids, shape-regularity is intended in the sense of Ciarlet.
For Cartesian grids, ${\cal D}$ is uniformly discretized with square elements of side $h$.
The numerical analysis will be restricted to triangular grids to avoid the complex notation and proofs for the general setting. In this context, the quasi-uniformity hypothesis is reframed, with a slight abuse of notation, setting the function $h$ to be globally constant.
\begin{rem}
The assumption of quasi-uniformity is not essential for the numerical analysis of the proposed methods, but it greatly simplifies the notation in the mathematical proofs.
\end{rem}
\begin{rem}
In this work we limit the discussion to two dimensions, but analogous strategies can be applied in the three-dimensional case. We leave this extension to future, more applied work.
\end{rem}

As shown in Figure~\ref{fig:surrogates}, the SBM is based on restricting the tessellation where the discrete variational formulation is applied to those elements that are {\sl strictly} contained in $\text{clos}(\Om)$, i.e., we form
$$
\ti{\cT}_h := \{ T \in \cT_h : T \subset \text{clos}(\Om) \} \, ,
$$ 
which identifies the {\sl surrogate domain}
$$
\tO := \text{int} \left(\bigcup_{T \in \ti{\cT}_h}  T \right) \subseteq \Om \,,
$$
with {\sl surrogate boundary} $\tG:=\partial \tO$ and outward-oriented unit normal vector $\ti{\bs{n}}$ to $\tG$. 
Obviously, $\ti{\cT}_h$ is an admissible and shape-regular tessellation of $\tO$ (see Figure~\ref{fig:SBM}).
In other words, all cut elements are removed from the active computational domain, which is now $\tO$ instead of $\Om$.
We now introduce a mapping
\begin{subequations}\label{eq:defMmap}
	\begin{align}
	\bs{M}_{h}:&\; \tG \to \G \; ,  \\
	&\; \ti{\bs{x}} \mapsto \bs{x}   \; ,
	\end{align}
\end{subequations}
which associates to any point $\ti{\bs{x}} \in \tG$ on the surrogate boundary a point $\bs{x} = \bs{M}_{h}(\ti{\bs{x}})$ on the physical boundary $\Gamma$.  Whenever uniquely defined, the closest-point projection of $\ti{\bs{x}}$ upon $\Gamma$ is a natural choice for $\bs{x}$, as shown e.g. in Figure~\ref{fig:ntd}. 
Through $\bs{M}_{h}$, a distance vector function $\bs{d}_{\bs{M}_{h}}$ can be defined as
\begin{align}
\label{eq:Mmap}
\bs{d}_{\bs{M}_{h}} (\ti{\bs{x}})
\, = \, 
\bs{x}-\ti{\bs{x}}
\, = \, 
[ \, \bs{M}_{h}-\bs{I} \, ] (\ti{\bs{x}})
\; .
\end{align}
For the sake of simplicity, we set $\bs{d} = \bs{d}_{\bs{M}_{h}} $ where $\bs{d} = \|\bs{d}\| \bs{\nu}$  and $\bs{\nu}$ is a unit vector. 
\begin{rem}
	If $\bs{x} = \bs{M}_{h}(\ti{\bs{x}})$ does not belong to corners or edges, then the closest-point projection implies $\bs{\nu}=\bs{n}$, where $\bs{n}$ has been defined as the outward pointing normal to $\G$.
\end{rem}
\begin{rem}
There are strategies for the definition of the map $\bs{M}_{h}$ and distance $\bs{d}$ other than the closest-point projection, such as level sets, for which $\bs{d}$ is defined by means of a distance function.
Other more sophisticated choices of $\bs{M}_{h}$ may be locally preferable and we refer to \cite{atallah2021analysis} for more details.
\end{rem} 
In case the boundary $\G$ is partitioned into a Dirichlet boundary $\G_{D}$ and a Neumann boundary $\G_{N}$ with $\G = \overline{\G_{D} \cup \G_{N}}$ and $\G_{D} \cap \G_{N} = \emptyset$, we need to identify whether a surrogate edge $\ti{e} \subset \tG$ is associated with $\G_{D}$ or $\G_{N}$. To that end, we partition $\tG$ as $\overline{\tGD \cup \tGN}$ with $\tGD \cap \tGN = \emptyset$ using again a map $\bs{M}_{h}$, such that
\begin{align}\label{def:tGD}
\tGD = \{ \ti{e} \subseteq \tG : \bs{M}_{h}(\ti{e}) \, \subseteq \, \G_{D} \}
\end{align}
and 
\begin{align}\label{def:tGN}
\tGN = \{ \ti{e} \subseteq \tG : \bs{M}_{h}(\ti{e}) \, \subseteq \, \G_{N} \} \; .
\end{align}
We will also assume that $\tGN = \tG \setminus \tGD$, that is, that either a surrogate edge $\ti{e}$ entirely belongs to the surrogate Dirichlet boundary or to the surrogate Neumann boundary. We will then prevent the case of mixed Dirichlet/Neumann surrogate edges. This hypothesis is realized in practice, by renouncing to the map $\bs{M}_{h}$ be the closest-point projection, as discussed in more detail in~\cite{atallah2021analysis}.

\subsection{General notation for inner products, norms, and seminorms}
Throughout this article, we denote by $L^{2}(\om)$, for $\om \subset \Om$, the space of Lebesgue square-integrable functions on $\Om$.
 %and by $L^{2}_{0}(\Om)$ the space of square integrable functions with zero mean on $\Om$ (i.e., $q \in L^{2}_{0}(\Om)$ implies $\int_{\Om} q = 0$). 
 We will use the Sobolev spaces $H^m(\om)=W^{m,2}(\om)$ of index of regularity $m \geq 0$ and index of summability 2, equipped with the (scaled) norm
\begin{equation}
\|v \|_{H^{m}(\om)} 
= \left( \| \, v \, \|^2_{L^2(\om)} + \sum_{k = 1}^{m} \| \, l(\om)^k \ssf{D}^k v \, \|^2_{L^2(\om)} \right)^{1/2} \; ,
\end{equation}
where $\ssf{D}^{k}$ is the $k$th-order spatial derivative operator and $l(A)=\mathrm{meas}_{2}(A)^{1/2}$ is a characteristic length of the domain $A$.
Note that $H^0(\om)=L^{2}(\om)$ and, as usual, we use a simplified notation for norms and semi-norms, i.e., we set $\| \, v  \, \|_{m,\Om}=\|\, v \, \|_{H^m(\om)}$ and $| \, v \, |_{k,\om}= 
\| \,\ssf{D}^k v \,\|_{0,\Om}= \| \,\ssf{D}^k v \, \|_{L^2(\om)}$.

We also introduce the definition of the $L^{2}$-inner product over $\om$, namely $( \, u \, , \, v  \, )_{\om} = \int_{\om} u \, v$, and an analogous inner product on the subset $\gamma \subset \partial \Om$, namely $\avg{ u \, , \, w }_{\gamma} = \int_{\gamma} u \, w$.
We can also restrict to $\om$ and $\gamma$ the norms and seminorms initially defined on $\Omega$ and $\Gamma$, that is $\| \cdot \|_{\om,k}$, $| \cdot |_{\om,k}$ and $\| \cdot \|_{\gamma,0}$, for example.

%-------------------------------------------%
\subsection{General strategy of the standard Shifted Boundary Method} 
Consider now the Poisson problem with Dirichlet and Neumann boundary conditions:
\begin{subequations}
\label{eq:strong}
\begin{align}
- \Delta u 
&=\; 
f \, , \qquad \mbox{in } \Om \; , \\
 u 
&=\; 
u_D \, , \ \quad \mbox{on } \G_D \; , \\ 
\nabla u \cdot \bs{n} 
&=\; 
h_N \, , \ \quad \mbox{on } \G_N \; ,
\end{align}
\end{subequations}
where $\G = \overline{\G_D \cup \G_N}$ and $\G_D \cap \G_N=\emptyset$.
Note also that vectors and tensors are marked in bold, while scalars are marked with regular fonts.

As already mentioned, the SBM discretizes the governing equations in $\tO$ rather than in $\Om$, with the challenge of accurately imposing boundary conditions on $\tG$.
To this end, boundary conditions are {\it shifted} from $\G$ to $\tG$, by performing an $m$th-order Taylor expansion of the variable of interest at the surrogate boundary, under the assumption that a solution variable $u$ is sufficiently smooth in the strip between $\tG$  and $\G$.
Let $\ssf{D}^{i}_{\bs{d}}$ denote the $i$th-order directional derivative along $\bs{d}$:
$$
\ssf{D}^{i}_{\bs{d}} u = \displaystyle{\sum_{\bs{\alpha} \in \mathbb{N}^n, |\bs{\alpha}|=i} \frac{i!}{\bs{\alpha}!}   \frac{\partial^i u}{\partial \bs{x}^{\bs{\alpha}}} \bs{d}^{\bs{\alpha}}  } \; . 
$$
Then, for $\ti{\bs{x}} \in \tG$ and $\bs{x} = \bs{M}_{h}(\ti{\bs{x}})$ and we can write
\begin{align}
u(\bs{x}) = u(\ti{\bs{x}}+\bs{d}(\ti{\bs{x}}))=u(\tx) +  \sum_{i = 1}^{m}  \frac{\ssf{D}^{i}_{\bs{d}} \, u(\ti{\bs{x}})}{i!} + (\ssf{R}^{m}(u,\bs{d}))(\tx)\,, %\qquad m \geq 1
\end{align}
where the remainder $\ssf{R}^{m}(u,\bs{d})$ satisfies  $|\ssf{R}^{m}(u,\bs{d})| = o(\Vert \bs{d}\Vert^m)$ as $\Vert \bs{d}\Vert \to 0$.
Assume that the Dirichlet condition $u(\bs{x})=u_D(\bs{x})$ needs to be imposed on the true boundary $\G_D$.
Using the map $\bs{M}_h$, one can extend $u_D$ from $\G_D$ to $\tGD$ as $\oE{u}_D(\tx)=u_D(\bs{M}_h (\tx))$. Then, the Taylor expansion can be used to enforce the Dirichlet condition on $\tGD$ rather than $\G_D$, as 
\begin{equation}\label{eq:trace-u}
\oS^{m} u  - \oE{u}_D + \ssf{R}^{m}(u,\bs{d}) = 0 
\; , \qquad \mbox{on } \tGD \; ,
\end{equation}
where we have introduced the boundary {\it shift} operator for every $\ti{\bs{x}} \in \tGD$, namely:
\begin{equation}\label{eq:def-bndS}
\oS^{m} u(\ti{\bs{x}}) := u(\ti{\bs{x}}) + \sum_{i = 1}^{m}  \frac{\ssf{D}^{i}_{\bs{d}} \, u (\ti{\bs{x}})}{i!}  \; .
\end{equation}
Neglecting the remainder $\ssf{R}^{m}(u,\bs{d})$, we obtain the final expression of the {\it shifted} approximation of order $m$ of the boundary condition
\begin{equation}
\oS^{m} u \approx \oE{u}_D  \, , \quad  \mbox{on } \tGD \; .
\end{equation}
This shifted boundary condition will be enforced weakly in what follows, and whenever there is no source of confusion, the symbol $\oE$ will be removed from the extended quantities, and we would write $u_D$ in place of $\oE{u}_D$.
\begin{equation}\label{eq:Finalu-g}
\oS u \approx u_D  \, , \quad  \mbox{on } \tGD \; ,
\end{equation}
with 
\begin{equation}
\oS := \oS^{1} u(\ti{\bs{x}}) = u(\ti{\bs{x}}) + \nabla u(\ti{\bs{x}}) \cdot \bs{d}(\ti{\bs{x}}) \; .
\label{eq:oS_taylor}
\end{equation}
To develop a simple SBM variational formulation, assume a triangular grid and consider the space
\begin{align}
V_h(\tO)  = \; \left\{ v_h \in C^0(\tO)  \ | \ {v_h}_{|T} \in \mathcal{P}^1(T)  \, , \, \forall T \in \ti{\mathcal{T}}_h \right\}  \, ,
\end{align}
where $\mathcal{P}^1(T)$ is the the space of linear polynomials over the triangle $T \in \ti{\mathcal{T}}_h$.
The penalty-free SBM formulation inspired by the work in~\cite{collins2023penalty} reads:
\begin{quote}
Find $u_h \in V_h(\tO)$ such that, $\forall w_h \in V^k_h(\tO)$
\begin{multline}
\label{eq:SB_Poisson_uns}
( \, \nabla u_h \, , \, \nabla w_h  \, )_{\tO}
- \avg{ \nabla u_h  \cdot \ti{\bs{n}} \, , \, w_h }_{\tG}
- \avg{ \bs{n} \cdot \ti{\bs{n}} \; (h_N - \nabla u_h  \cdot \bs{n}) \, , \, w_h }_{\tGN}
\\
+  \avg{ \oS u_h -u_D\, , \, \nabla w_h \cdot \ti{\bs{n}} }_{\tGD}
-( \, f \, , \, w_h \, )_{\tO}
\;=\;0
\; .
\end{multline}
\end{quote}
It relies on Nitsche's method to enforce Dirichlet boundary conditions, in which the shift operator $\oS$ (i.e., a Taylor expansion in our case) is used to extend the discrete solution $u_h \in V_h(\tO)$ from $\tGD$ to $\G_D$. Neumann conditions are imposed weakly assuming a constant extrapolation of the gradient $\nabla u_h$.
 This choice is due to the fact that, in the case of piecewise-linear approximation, it is not possible to construct a Taylor expansion of $\nabla u_h$, since the Hessian of $u_h$ would vanish. It was found in numerical computations that this limitation causes the SBM to have suboptimal convergence (i.e., first- instead of second-order convergence) in the $L^2$-norm of the error, while the $H^1$-seminorm of the error converges optimally (first-order).
Adapting the discrete approximation spaces, similar derivations and conclusions can be obtained in the case of bi-linear elements over Cartesian grids.

Atallah et al.~\cite{atallah2021solid} addressed the sub-optimality of the SBM with Neumann boundary conditions via a mixed shifted formulation (e.g. solving for strains and displacements in solid mechanics).
In the next section, we will develop an alternative approach based on a primal formulation.

%-------------------------------------------%
%
%
\section{A new conceptualization of the Shifted Boundary Method
\label{sec:newSBM}}

To cure the loss of optimality in the convergence of the $L^2$-norm, we propose to account for the effect on the solution $u_h$ of the gap $\Om \setminus \tO$ between $\tG$ and $\G$ (the gray-shaded region in Figure~\ref{fig:SBM}). Of course, we will {\sl avoid explicitly integrating} the discrete equations over the gap region but, rather, we will derive approximate quadrature formulas that do not require integration over cut elements.
\begin{figure}[tb]
	\centering
      		\begin{tikzpicture}[scale=1.4]
		%%% draw 
        \draw [black, draw=none,name path=surr] plot coordinates {
                    (0,-0.5)
                    (1,1.73205)
                    (2.6,2.2)
                  };
		\draw [blue, name path=true] plot[smooth] coordinates {
                    (1,-1)
                    (2.0,0.4)
                    (2.7,0.6)};

		%%%% fill
            \fill[fill=Green!15!]  (1,1.73205) -- (2,0.4) -- (2.7,0.6) -- (2.6,2.2)  -- (1,3.4641) -- (1,1.73205);
            \fill[fill=orange!15!] (1,1.73205) -- (2,0.4) -- (1,-1)    -- (0,-0.5)   -- (-1,1.73205) -- (1,1.73205);
            \draw[line width = 0.25mm,Green] (2,0.4) -- (2.7,0.6);
            \draw[line width = 0.25mm,orange] (2,0.4) -- (1,-1);

%		\tikzfillbetween[of=true and surr,split=false]{green!15!};
		%%% first line of elements 
		\draw[line width = 0.25mm,densely dashed,gray] (-1,1.73205) -- (0,3.4641);
		\draw[line width = 0.25mm,densely dashed,gray] (0,3.4641) -- (1,1.73205);%(2,0);
		\draw[line width = 0.25mm,densely dashed,gray] (1,1.73205) -- (1,3.4641);
		\draw[line width = 0.25mm,densely dashed,gray] (1,3.4641) -- (0,3.4641);
		\draw[line width = 0.25mm,densely dashed,gray] (1,3.4641) -- (2.6,2.2);
		\draw[line width = 0.25mm,densely dashed,gray] (1,3.4641) -- (3.25,3.4641);
		\draw[line width = 0.25mm,densely dashed,gray] (3.25,3.4641) -- (2.6,2.2);
		%%% second line of elements 
		\draw[line width = 0.25mm,densely dashed,gray] (0,-0.5) -- (-2,0.5);
		\draw[line width = 0.25mm,densely dashed,gray] (-2,0.5) -- (-1,1.73205);
		\draw[line width = 0.25mm,densely dashed,gray] (-1,1.73205) -- (1,1.73205);
		\draw[line width = 0.25mm,densely dashed,gray] (0,-0.5) -- (-1,1.73205);
		%%%% True boundary
		\draw [line width = 0.5mm,blue, name path=true] plot[smooth] coordinates {(0.7,-1.5) (2.0,0.4) (3.75,0.65)};

        %% surrogate boundary
        \draw[line width = 0.5mm,red] (2.6,2.2) -- (3.5,2.6);
		\draw[line width = 0.5mm,red] (1,1.73205) -- (2.6,2.2);
		\draw[line width = 0.5mm,red] (1,1.73205) -- (0,-0.5);
		\draw[line width = 0.5mm,red] (0,-0.5) -- (-1,-0.7);

        %%arrows
	    \draw[->,line width = 0.5mm,-latex] (1,1.73205) -- (2,0.4);
	 	\draw[->,line width = 0.5mm,-latex] (2.6,2.2) -- (2.7,0.6);
	 	\draw[->,line width = 0.5mm,-latex] (0,-0.5) -- (1,-1);

		%% labels
		\node[text width=0.5cm] at (3.3,2.15) {${\color{red}\tG}$};
		\node[text width=3cm] at (1.25,2.25) {${\color{red}\tO}$};
		\node[text width=0.5cm] at (4.1,0.65) {${\color{blue}\G}$};
            %\node[text width=0.5cm] at (1.8,1.2) {\large$\bs{d}$};
            \node[text width=0.5cm] at (0.5,0.7) {\color{red} $\ti{e}$};
            \node[text width=0.5cm] at (0.1,-0.7) {$\ti{\bs{a}}_{1}$};
            \node[text width=0.5cm] at (1.35,1.635) {$\ti{\bs{a}}_2$};
            \node[text width=0.5cm] at (1.2,-1.05) {$\bs{a}_{1}^{\tmop{ext}}$};
            \node[text width=0.5cm] at (2.2,0.30) {$\bs{a}_2^{\tmop{ext}}$};
            \node[text width=0.5cm] at (0.6,-0.6) {${e_{\ti{\bs{a}}_1}}$};
            \node[text width=0.5cm] at (1.3,1.0) {${e_{\ti{\bs{a}}_2}}$};
            \node[text width=0.05cm] at (0.8,0.2) {${T^{\tmop{ext}}_1}$};
            \node[text width=0.05cm] at (1.5,-0.4) {\color{red} ${{e^{\tmop{ext}}}}$};
            \node[text width=1.5cm] at (0.1,1.4) {\color{red} $\ti{T}_1 \subset \tO$};
            \node[text width=0.05cm] at (1.8,1.2) {$T^{\tmop{ext}}_2$};
%            \node[text width=0.05cm] at (2.3,0.25) {\color{Green} ${e_2^{\tmop{ext}}}$};
            \node[text width=1.5cm] at (1.8,2.35) {\color{red} $\ti{T}_2 \subset \tO$};
%            \node[text width=0.75cm,fill=green!15!] at (1.75,2) {\large$\bs{T^{\tmop{ext}}_E}$};
            \node[text width=0.5cm] at (3.0,1.5) {$\bs{d}$};
		\end{tikzpicture}
		\caption{The extension of the shape functions from the surrogate domain (e.g., elements $\ti{T}_1$ and $\ti{T}_2$) to the gap $\Om \setminus \tO$ (e.g., the element extensions $T^{\tmop{ext}}_1$ and $T^{\tmop{ext}}_2$). $T^{\tmop{ext}}_1$ has the edge ${e^{\tmop{ext}}}$ that interpolates $\G$, and similarly for $T^{\tmop{ext}}_2$.  
		}
			\label{fig:SBM2}
\end{figure}
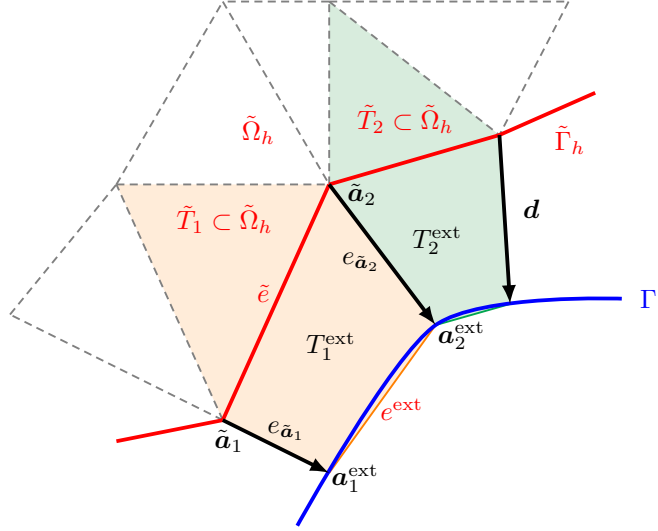

To this end, we define the extension $V^{\tmop{ext}}_h(\Om)$ of $V_h(\tO)$, where a function $v_h^{\tmop{ext}} \in V^{\tmop{ext}}_h(\Om)$ is obtained as the linear combination of the extensions to the gap $\Om \setminus \tO$ of the piecewise-linear basis functions used to represent $v_h$ in $\tO$. In other words, referring to the sketch of Figure~\ref{fig:SBM2} for the a two-dimensional triangular grid, each shape function that is non-zero over the elements attached to edges in $\tG$ is evaluated over a point in the gap, then the linear combination forming $v_h$ is taken and renamed $v_h^{\tmop{ext}}$. For all other elements in the discretization we have instead that $v_h^{\tmop{ext}}=v_h$, and no explicit extension is needed.

The gap $\Om \setminus \tO$ is discretized as follows: for every element with an edge $\ti{e}$ on $\tG$ (e.g., the element $\ti{T}_1$ in Figure~\ref{fig:SBM2}), we consider the two end nodes of that edge (e.g., $\ti{\bs{a}}_{1}$ and $\ti{\bs{a}}_{2}$), and project them via $\bs{M}_{h}$ onto $\G$ (to $\bs{a}_{1}^{\tmop{ext}}=\bs{M}_{h}(\ti{\bs{a}}_{1})$ and $\bs{a}_{2}^{\tmop{ext}}=\bs{M}_{h}(\ti{\bs{a}}_{2})$), to obtain then a projection of $\ti{e}$, called $e^{\tmop{ext}}$, which interpolates $\G$ between $\bs{a}_{1}^{\tmop{ext}}$ and $\bs{a}_{2}^{\tmop{ext}}$.
We define the union of the edges $e^{\tmop{ext}}$ as $\G_h$, the interpolant of $\G$, which can be further decomposed into a Dirichlet part $\GD$ and Neumann part $\GN$.
The quadrilateral that connects $\ti{\bs{a}}_{1}$, $\ti{\bs{a}}_{2}$, $\bs{a}_{1}^{\tmop{ext}}$, and $\bs{a}_{2}^{\tmop{ext}}$ is the element extension (${T^{\tmop{ext}}_1}$ in Figure~\ref{fig:SBM2}) to $\Om \setminus \tO$ of the original element ($\ti{T}_1$).
Let us denote by $\ti{\mathcal{T}}_h^{\tmop{ext}}$ the set of quadrilaterals constructed to discretize the gap $\Om \setminus \tO$ by this procedure.

\begin{rem}
\label{rem:edge_ext}
In the geometric construction presented here, the boundary $\Om$ is approximated as polygonal, that is $\Om \approx \Om^h$, where $\Om^h$ is a polygonal domain with its vertices lying on the boundary of $\Om$.
$\Om^h$ introduces a geometric error, which is however quadratic in nature, and this approximation can be made safely in the context of piecewise-linear finite element approximation spaces. Also observe that one can avoid this approximation by using the distance $\bs{d}$ along the entire edge $\ti{e}$, so that the edge $e^{\tmop{ext}}$ would be curved. We prefer to avoid these complications for the sake of simplicity.
In the case of higher-order discretizations, the previous argument can be adjusted by computing distances at each node along the edges in $\tG$ (including nodes internal to the edges), and constructing with such distances a higher-order approximation $\Om^h$ of $\Om$.
\end{rem}

Based on the discussion in Remark~\ref{rem:edge_ext}, we will always use interchangeably $\Om$ and $\Om^h$ in what follows.
The previous geometric construction allows us to extend the shape functions defined over $\ti{T}_1$ to ${T^{\tmop{ext}}_1}$. This is in a nutshell the construction of a function $v_h^{\tmop{ext}} \in V^{\tmop{ext}}_h(\Om)$.
In what follows, we will only consider the two-dimensional setting, but analogous derivations can be extended to the three-dimensional case.

Observe that the dimension of the function spaces $V^{\tmop{ext}}_h(\Om)$ and $V_h(\tO)$ are the same, since no additional degrees of freedom are added in the extension process.

We want now to derive a weak form that implements a SBM discretization of the strong form~\eqref{eq:strong}. For a sufficiently regular solution $u$ of the infinite dimensional problem, the strong form~\eqref{eq:strong} is multiplied by $w_h^{\tmop{ext}} \in V^{\tmop{ext}}_h(\Om)$:
\begin{align}
\label{eq:weak1}
- ( w_h^{\tmop{ext}} \, , \, \Delta u )_{\Om}
= 
( w_h^{\tmop{ext}} \, , \, f )_{\Om}
\; .
\end{align}

Let us introduce the following notation, for the sake of brevity: 
\begin{subequations}
\begin{align}
( \cdot \, , \, \cdot )_{\ti{\mathcal{T}}_h} 
&=\; 
\sum_{\ti{T} \in \ti{\mathcal{T}}_h} ( \cdot \, , \, \cdot )_{\ti{T}}
\; ,
\\
( \cdot \, , \, \cdot )_{\ti{\mathcal{T}}_h^{\tmop{ext}}} 
&=\;
\sum_{T^{\tmop{ext}} \in \ti{\mathcal{T}}_h^{\tmop{ext}}} ( \cdot \, , \, \cdot )_{T^{\tmop{ext}}}
\; ,
\\
\avg{ \cdot \, , \, \cdot }_{\partial \ti{\mathcal{T}}_h } 
&=\;
\sum_{\ti{T} \in \ti{\mathcal{T}}_h} \avg{ \cdot \, , \, \cdot }_{\partial \ti{T}} 
\; ,
\\
\avg{ \cdot \, , \, \cdot }_{\partial \ti{\mathcal{T}}_h^{\tmop{ext}} } 
&=\;
\sum_{T^{\tmop{ext}} \in \ti{\mathcal{T}}_h^{\tmop{ext}}} \avg{ \cdot \, , \, \cdot }_{\partial T^{\tmop{ext}}} 
\; .
\end{align}
\end{subequations}
Then integrating by parts equation~\eqref{eq:weak1} over all the triangles in $\ti{\mathcal{T}}_h$ and the quadrilaterals in $\ti{\mathcal{T}}_h^{\tmop{ext}}$,
 we have
%%
%\begin{align}
%- ( v_h^{\tmop{ext}} \, , \, \Delta \bar{u}_h )_{\ti{\mathcal{T}}_h}
%\, - \, ( v_h^{\tmop{ext}} \, , \, \Delta \bar{u}_h)_{\ti{\mathcal{T}}_h^{\tmop{ext}}} 
%= 
%( v_h^{\tmop{ext}} \, , \, f )_{\ti{\mathcal{T}}_h}
%\, + \, ( v_h^{\tmop{ext}} \, , \, f )_{\ti{\mathcal{T}}_h^{\tmop{ext}}} 
%\; ,
%\end{align}
%%
%
%
\begin{multline}
\label{eq:SBM_e3}
( \nabla w_h  \, , \,  \nabla u )_{\ti{\mathcal{T}}_h}
\, - \,
\avg{ w_h \, , \, \nabla u  \cdot \bs{n}_T}_{\partial \ti{\mathcal{T}}_h} 
%\; \phantom{=} &
%\nonumber 
\, + \, 
( \nabla w_h^{\tmop{ext}}  \, , \,  \nabla u )_{\ti{\mathcal{T}}_h^{\tmop{ext}}} 
\, - \,
\avg{ w_h^{\tmop{ext}} \, , \, \nabla u  \cdot \bs{n}_{T^{\tmop{ext}}} }_{\partial \ti{\mathcal{T}}_h^{\tmop{ext}}} 
%\; = &
\\
\; = \;
( w_h \, , \, f )_{\ti{\mathcal{T}}_h}
\, + \, 
( w_h^{\tmop{ext}} \, , \, f )_{\ti{\mathcal{T}}_h^{\tmop{ext}}} 
\, ,
\end{multline}
where we have used the fact that for a given $w_h \in V_h(\tO)$, the extension $w_h^{\tmop{ext}} \in V^{\tmop{ext}}_h(\Om)$ coincides with $w_h$ over $\tO$.
\begin{rem}
\label{rem:jump_wh}
While $w_h \in V_h(\tO)$ is continuous over $\tO$, the extension $w_h^{\tmop{ext}} \in V^{\tmop{ext}}_h(\Om)$ can be discontinuous over the gap $\Om \setminus \tO$. Looking at Figure~\ref{fig:SBM2}, and in particular to the edge ${e_{\ti{\bs{a}}_2}}$ emanating from $\ti{\bs{a}}_2$ along the distance vector $\bs{d}(\ti{\bs{a}}_2)$, it is clear that the extended shape functions from the elements $\ti{T}_1$ and $\ti{T}_2$ may not match on ${e_{\ti{\bs{a}}_2}}$. This happens because the gradient over ${T^{\tmop{ext}}_1}$ is the same as the one over $\ti{T}_1$ and the gradient over ${T^{\tmop{ext}}_2}$ is the same as the one over $\ti{T}_2$, but the gradients over $\ti{T}_1$ and $\ti{T}_2$ in general do not match.
\end{rem}

Because of the presence of potential discontinuities, the framework of Discontinuous Galerkin Methods seems the most appropriate to proceed.
The contribution from the internal element boundaries in equation~\eqref{eq:SBM_e3} can be expanded by making use of the following definitions and identities of the jumps and averages of edge quantities:
\begin{align}
\jumpb{w} &=\; w^{+}\bs{n}^{+}+w^{-}\bs{n}^{-} \; , \label{eq:jump} \\
\jumpb{\bs{v}} &=\; \bs{v}^{+}\cdot \bs{n}^{+}+ \bs{v}^{-}\cdot \bs{n}^{-} \; , \label{eq:avg} \\
\avgb{w} &=\; 
\frac{1}{2}(w^{+}+w^{-}) \; , \\
\jumpb{v \, \bs{w}} &=\; 
\avgb{v} \jumpb{\bs{w}}+\jumpb{v} \cdot \avgb{\bs{w}} \; .
\end{align}
In what follows, to simplify the notation, we will write $w_h$ in place of $w_h^{\tmop{ext}} \in V^{\tmop{ext}}_h(\Om)$, since the regions of integration uniquely define whether we are considering a test/trial function on $\tO$ or its extension over $\Om \setminus \tO$.
Hence, denoting $\ti{\mathcal{E}}^o$ the set of interior faces in $\tO$ and $\mathcal{E}^{\tmop{ext};o}$ the set of faces of elements in $\ti{\mathcal{T}}_h^{\tmop{ext}}$ that do not lie on $\G$ nor $\tG$, we have
\begin{align}
\avg{ w_h \, , \, \nabla u  \cdot \bs{n}}_{\partial \ti{\mathcal{T}}_h} 
\, + \, 
\avg{ w_h \, , \, \nabla u  \cdot \bs{n}}_{\partial \ti{\mathcal{T}}_h^{\tmop{ext}}} 
=&\; 
\avg{ 1 \, , \, \jumpb{w_h \nabla u} }_{\ti{\mathcal{E}}^o \cup \tG \cup \mathcal{E}^{\tmop{ext};o} }
\nonumber \\
\phantom{=}&\; 
\; + \;
\avg{ w_h \, , \,  \underbrace{\nabla u \cdot \bs{n}}_{=h_N}  }_{\GN}
\; + \;
\avg{ w_h \, , \,  \nabla u \cdot \bs{n}  }_{\GD}
\nonumber \\
=&\; 
\avg{ 1 \, , \, 
\avgb{w_h} \jumpb{\nabla u} + \jumpb{w_h}\cdot\avgb{\nabla u} }_{\ti{\mathcal{E}}^o \cup \tG \cup \mathcal{E}^{\tmop{ext};o} }
\nonumber \\
\phantom{=} &\; 
\; + \;
\avg{ w_h \, , \, h_N  }_{\GN}
\; + \;
\avg{ w_h \, , \,  \nabla u \cdot \bs{n}  }_{\GD}
\nonumber \\[.2cm]
=&\; 
\avg{ \jumpb{w_h} \, , \, \avgb{\nabla u} }_{\mathcal{E}^{\tmop{ext};o} }
\; + \;
\avg{ w_h \, , \, h_N  }_{\GN}
\nonumber \\
\phantom{=} &\; 
\; + \;
\avg{ w_h \, , \,  \nabla u \cdot \bs{n}  }_{\GD}
\; , 
\label{WrongEquality}
\end{align}
where we have used the fact that $\jumpb{\nabla u}=0$ on edges in $\ti{\mathcal{E}}^o \cup \tG \cup \mathcal{E}^{\tmop{ext};o}$ and $\jumpb{w_h}=0$ on edges in $\ti{\mathcal{E}}^o \cup \tG$.
Observing that $( \nabla w_h  \, , \,  \nabla u )_{\ti{\mathcal{T}}_h} = ( \nabla w_h  \, , \,  \nabla u )_{\tO}$ and $( \nabla w_h  \, , \,  f )_{\ti{\mathcal{T}}_h} = ( \nabla w_h  \, , \,  f )_{\tO}$, we have
\begin{multline}
( \nabla w_h  \, , \,  \nabla u )_{\tO}
\, + \, 
( \nabla w_h  \, , \,  \nabla u )_{\ti{\mathcal{T}}_h^{\tmop{ext}}} 
\, - \,
\avg{ \jumpb{w_h} \, , \, \avgb{\nabla u} }_{\mathcal{E}^{\tmop{ext};o}}
\; - \;
\avg{ w_h \, , \, h_N  }_{\GN}
\\
\; - \;
\avg{ w_h \, , \,  \nabla u \cdot \bs{n}  }_{\GD}
\; = \;
( w_h \, , \, f )_{\tO}
\, + \, 
( w_h \, , \, f )_{\ti{\mathcal{T}}_h^{\tmop{ext}}} 
\; .
\label{eq:SBM_var_gen00}
\end{multline}
Let us now replace the infinite dimensional exact solution $u$ with its approximation $u_h^{\tmop{ext}} \in V^{\tmop{ext}}_h(\Om)$ and simplify the notation with $u_h$ in place of $u_h^{\tmop{ext}} \in V^{\tmop{ext}}_h(\Om)$, as we did before for $w_h^{\tmop{ext}}$.
Complementing the previous equation with terms that weakly enforce the Dirichlet condition and the continuity of the solution across edges in $\mathcal{E}^{\tmop{ext};o}$, we finally obtain the discrete weak formulation:
\begin{multline}
( \nabla w_h  \, , \,  \nabla u_h )_{\tO}
\, + \, 
( \nabla w_h  \, , \,  \nabla u_h )_{\ti{\mathcal{T}}_h^{\tmop{ext}}} 
\; - \;
\avg{ w_h \, , \, h_N  }_{\GN}
\\
\, - \,
\avg{ \jumpb{w_h} \, , \, \avgb{\nabla u_h} }_{\mathcal{E}^{\tmop{ext};o} }
\; - \;
\avg{ w_h \, , \,  \nabla u_h \cdot \bs{n}  }_{\GD}
\, - \; \theta \;
\avg{ \avgb{\nabla w_h} \, , \, \jumpb{u_h} }_{\mathcal{E}^{\tmop{ext};o} }
\; - \; \theta \;
\avg{ \nabla  w_h \cdot \bs{n} \, , \,  u_h-u_D  }_{\GD}
\\
\; + \;
\langle \gamma \, h^{-1} \jumpb{w_h} \,, \, \jumpb{u_h} \rangle_{\mathcal{E}^{\tmop{ext};o}}
\; + \;
\langle \gamma \, h^{-1} w_h \,, \, u_h-u_D \rangle_{\GD}
\; = \;
( w_h \, , \, f )_{\tO}
\, + \, 
( w_h \, , \, f )_{\ti{\mathcal{T}}_h^{\tmop{ext}}} 
\; ,
\label{eq:SBM_var_gen}
\end{multline}
where Remark~\ref{rem:jump_wh} is in order also for $u_h$. The terms
$$-\theta \avg{ \avgb{\nabla w_h} \, , \, \jumpb{u_h} }_{\mathcal{E}^{\tmop{ext};o} } - \theta \avg{ \nabla  w_h \cdot \bs{n} \, , \,  u_h-u_D  }_{\GD} + \langle \gamma \, h^{-1} \jumpb{w_h} \,, \, \jumpb{u_h} \rangle_{\GD} + \langle \gamma \, h^{-1} w_h \,, \, u_h \rangle_{\GD}$$
are typical of an interior penalty discontinuous Galerkin discretization. 
In particular, for $\theta=1$ and $\gamma>0$ we obtain a symmetric interior penalty Galerkin discretization, while for $\theta=-1$ and $\gamma=0$ we obtain the skew-symmetric (or non-symmetric) interior penalty Galerkin discretization, which has the advantage of being penalty-free. 
An SBM version of the latter had been recently explored in~\cite{collins2023penalty}. 

Now, rather than computing integrals on the gap between $\tG$ and $\G_h$, or on $\GD$ and $\GN$, we propose an approximation to these integrals using information on the $\tG$ and appropriate rescaling of integrals.
Let us start approximating the term:
\begin{align}
( \nabla w_h  \, , \,  \nabla u_h )_{\ti{\mathcal{T}}_h^{\tmop{ext}}} 
&\approx\;
\sum_{\ti{e} \subset \tG}
\avg{ \nabla w_h  \, , \,  \nabla u_h \, \frac{|T_{\ti{e}}^{\tmop{ext}}|}{|\ti{e}|}  }_{\ti{e}} 
\; ,
\end{align}
where, $T_{\ti{e}}$ is the quadrilateral emanating from edge $\ti{e}$.
In practice, we have reduced the integrals over extended elements $T^{\tmop{ext}} \in \ti{\mathcal{T}}_h^{\tmop{ext}}$ to integrals over edges $\ti{e} \in \tG$, by introducing the rescaling factor $|T^{\tmop{ext}}|/|\ti{e}|$, which is the ratio between the measure (area) of $T^{\tmop{ext}}$ and the measure (length) of the edge $\ti{e}$. See Figure~\ref{fig:SBM3} for a sketch of the geometric construction.
This simplified approach is reminiscent of using a biased left-node quadrature over an interval in one space dimension, instead of the common mid-point or Gauss quadratures. Similarly, defining 
\begin{equation}
\label{eq:Hdef}
H_{\ti{e}}
\; := \;
\frac{|T_{\ti{e}}^{\tmop{ext}}|}{|\ti{e}|}
\; ,
\end{equation}
we obtain
\begin{align}
( w_h \, , \, f )_{\ti{\mathcal{T}}_h^{\tmop{ext}}} 
&\approx\;
\sum_{\ti{e} \subset \tG}
\avg{ w_h \, , \, f \, H_{\ti{e}}  }_{\ti{e}} 
\; .
\end{align}
Consider now the term
\begin{align}
\avg{ w_h \, , \, h_N  }_{\GN}
&\approx\;
\sum_{\ti{e} \subset \tGN}
\avg{ \oS(w_h) \, , \, h_N(M_h(\ti{\bs{x}})) \, j_{\ti{e}}  }_{\ti{e}} 
\; ,
\end{align}
where we approximated the Jacobian of the transformation mapping any boundary edge $\ti{e} \subset \tG$ to its extension $e^{\tmop{ext}} \subset \G$ as 
\begin{equation}
\label{eq:incr_approx}
j_{\ti{e}}
\; := \;
\frac{ |e^{\tmop{ext}}| }{ |\ti{e}| } 
\approx 
\frac{| \mathrm{d} \mathbf{M}_h(\tilde{\bs{x}}) |}{|  \mathrm{d}  \tilde{\bs{x}} |} 
\; ,
\end{equation}
thus neglecting a higher-order error contributions. Note also that the map $M_h$ is used to evaluate $h_N$ on $\bs{x} = M_h(\ti{\bs{x}})$, and that the shift $\oS(w_h)$ is used to evaluate $w_h$ on $e^{\tmop{ext}}$.
Hence, we use the reference edge $\ti{e}$ to perform the integration over $e^{\tmop{ext}}$, introducing an approximate formula for the change of variables under integration.
Similarly:
\begin{align}
\label{eq:simple1}
\avg{ w_h \, , \,  \nabla u_h \cdot \bs{n}  }_{\GD}
& \approx \;
\sum_{\ti{e} \subset \tGD}
\avg{ \oS(w_h) \, , \,  (\nabla u_h \cdot \bs{n}) \, j_{\ti{e}} }_{\ti{e}}
\; , \\
\label{eq:simple2}
\avg{ \nabla  w_h \cdot \bs{n} \, , \,  u-u_D  }_{\GD}
& \approx \;
\sum_{\ti{e} \subset \tGD}
\avg{ \nabla w_h \cdot \bs{n} \, , \, (\oS(u_h) -u_D) \, j_{\ti{e}}  }_{\ti{e}} 
\; , \\
\label{eq:simple3}
\langle \gamma \, h^{-1} w_h \,, \, u_h-u_D \rangle_{\GD}
& \approx \;
\sum_{\ti{e} \subset \tGD}
\avg{ \gamma \, h^{-1} \, \oS(w_h) \, , \, (\oS(u_h) -u_D) \, j_{\ti{e}}  }_{\ti{e}} 
\; .
\end{align}
Note that in the approximations (\ref{eq:simple1})--(\ref{eq:simple2}) we introduce a slight abuse of notation by writing the normal $\bs{n}$ inside the integrals over $\ti{e}$. In fact, at any point $\ti{\bs{x}}\in\ti{e}$, the normal vector $\bs{n}$ to the actual boundary $\Gamma$ should be evaluated as $\bs{n}(M_h(\ti{\bs{x}}))$, and the same applies $u_D$ in (\ref{eq:simple2})--(\ref{eq:simple3}).
The last three remaining terms to be approximated are:
\begin{align}
\label{eq:ext_jump1}
\avg{ \jumpb{w_h} \, , \, \avgb{\nabla u_h} }_{\mathcal{E}^{\tmop{ext};o} }
& \approx \;
\sum_{\ti{\bs{a}} \in \mathcal{N}(\tG)}  
|\bs{d}_{\ti{\bs{a}}}| \, 
\jumpb{ \oSh w_h}_{\ti{\bs{a}}}
\cdot 
\avgb{ \nabla u_h }_{\ti{\bs{a}}} 
\nonumber \\
&\qquad =
\sum_{\ti{\bs{a}} \in \mathcal{N}(\tG)}  
\frac{|\bs{d}_{\ti{\bs{a}}}|}{2} \, 
\jumpb{ \nabla w_h \cdot \bs{d}_{\ti{\bs{a}}} }_{\ti{\bs{a}}}
\cdot 
\avgb{ \nabla u_h }_{\ti{\bs{a}}} 
\; , 
\\
\label{eq:ext_jump2}
\avg{ \avgb{\nabla w_h} \, , \, \jumpb{u_h} }_{\mathcal{E}^{\tmop{ext};o} }
& \approx \;
\sum_{\ti{\bs{a}} \in \mathcal{N}(\tG)}  
|\bs{d}_{\ti{\bs{a}}}| \, 
\avgb{ \nabla w_h }_{\ti{\bs{a}}} 
\cdot 
\jumpb{ \oSh u_h}_{\ti{\bs{a}}}
\nonumber \\
&\qquad =
\sum_{\ti{\bs{a}} \in \mathcal{N}(\tG)}  
\frac{|\bs{d}_{\ti{\bs{a}}}|}{2} \, 
\avgb{ \nabla w_h }_{\ti{\bs{a}}} 
\cdot 
\jumpb{ \nabla u_h \cdot \bs{d}_{\ti{\bs{a}}} }_{\ti{\bs{a}}}
\; ,
\\
\label{eq:ext_jump3}
\langle \gamma \, h^{-1} \jumpb{w_h} \,, \, \jumpb{u_h} \rangle_{\mathcal{E}^{\tmop{ext};o}}
& \approx \;
\sum_{\ti{\bs{a}} \in \mathcal{N}(\tG)}  
\frac{\gamma}{h} \, |\bs{d}_{\ti{\bs{a}}}| \, 
\jumpb{ \oSh w_h}_{\ti{\bs{a}}}
\cdot 
\jumpb{ \oSh u_h}_{\ti{\bs{a}}}
\nonumber \\
&\qquad =
\sum_{\ti{\bs{a}} \in \mathcal{N}(\tG)}  
\frac{\gamma}{4h} \, |\bs{d}_{\ti{\bs{a}}}| \, 
\jumpb{ \nabla w_h \cdot \bs{d}_{\ti{\bs{a}}} }_{\ti{\bs{a}}}
\cdot 
\jumpb{ \nabla u_h \cdot \bs{d}_{\ti{\bs{a}}} }_{\ti{\bs{a}}}
\; ,
\end{align}
where $\oSh(w_h(\ti{\bs{x}}))=w_h(\ti{\bs{x}})+1/2 \nabla w_h(\ti{\bs{x}}) \cdot \bs{d}$ and $\mathcal{N}(\tG)$ is the set of grid nodes along the surrogate boundary $\tG$, at which both $u_h$ and $w_h$ are continuous.
In particular, the average $\avgb{ \cdot }_{\ti{\bs{a}}}$ and jump $\jumpb{ \cdot }_{\ti{\bs{a}}}$ across the node $\ti{\bs{a}} \subset \tG$ are computed by evaluating the fields on the two edges of $\tG$ emanating from $\ti{\bs{a}}$, with formulas analogous to \eqref{eq:jump} and \eqref{eq:avg}, assuming that the unit vectors $\bs{n}^+,\bs{n}^-$ are normal to the edge in $\mathcal{E}^{\tmop{ext};o}$ stemming from $\ti{\bs{a}}$, i.e. to the vector $\bs{d}_{\ti{\bs{a}}}$. 
In practice, formulas \eqref{eq:ext_jump1}--\eqref{eq:ext_jump3} approximate the integrals on the edges in $\mathcal{E}^{\tmop{ext};o}$ with a mid-point type formula, using information available at each node $\ti{\bs{a}} \in \mathcal{N}(\tG)$. The magnitude $|\bs{d}_{\ti{\bs{a}}}|$ represents the length of each edge.
Note that the formulas \eqref{eq:ext_jump1}--\eqref{eq:ext_jump2} are actually exact in the case of triangular meshes and linear finite elements since the gradients of involved functions are then constant on each cell.
Approximations \eqref{eq:ext_jump1}--\eqref{eq:ext_jump3} can also be applied to the case of Cartesian grids with bi-linear elements, and should include the higher-order terms in the Taylor expansion in the case of higher-order finite element spaces.
\begin{figure}[tb]
	\centering
	\begin{tikzpicture}[scale=1.4]
		%%% draw 
        \draw [black, draw=none,name path=surr] plot coordinates {
                    (0,-0.5)
                    (1,1.73205)
                    (2.6,2.2)
                  };
		\draw [blue, name path=true] plot[smooth] coordinates {
                    (0.9,-1.75) (1,-1) (2.0,0.4) (3.75,0.45)};

%		%%%% fill
%            \fill[fill=Green!15!]  (1,1.73205) -- (2,0.4) -- (2.7,0.6) -- (2.6,2.2)  -- (1,3.4641) -- (1,1.73205);
%            \fill[fill=orange!15!] (1,1.73205) -- (2,0.4) -- (1,-1)    -- (0,-0.5)   -- (-1,1.73205) -- (1,1.73205);
%            \draw[line width = 0.25mm,Green] (2,0.4) -- (2.7,0.6);
%            \draw[line width = 0.25mm,orange] (2,0.4) -- (1,-1);

%		\tikzfillbetween[of=true and surr,split=false]{green!15!};
		%%% first line of elements 
		\draw[line width = 0.25mm,densely dashed,gray] (-1,1.73205) -- (0,3.4641);
		\draw[line width = 0.25mm,densely dashed,gray] (0,3.4641) -- (1,1.73205);%(2,0);
		\draw[line width = 0.25mm,densely dashed,gray] (1,1.73205) -- (1,3.4641);
		\draw[line width = 0.25mm,densely dashed,gray] (1,3.4641) -- (0,3.4641);
		\draw[line width = 0.25mm,densely dashed,gray] (1,3.4641) -- (2.6,2.2);
		\draw[line width = 0.25mm,densely dashed,gray] (1,3.4641) -- (3.25,3.4641);
		\draw[line width = 0.25mm,densely dashed,gray] (3.25,3.4641) -- (2.6,2.2);
		%%% second line of elements 
		\draw[line width = 0.25mm,densely dashed,gray] (0,-0.5) -- (-2,0.5);
		\draw[line width = 0.25mm,densely dashed,gray] (-2,0.5) -- (-1,1.73205);
		\draw[line width = 0.25mm,densely dashed,gray] (-1,1.73205) -- (1,1.73205);
		\draw[line width = 0.25mm,densely dashed,gray] (0,-0.5) -- (-1,1.73205);
		%%%% True boundary
		\draw [line width = 0.5mm,blue, name path=true] plot[smooth] coordinates {(0.9,-1.75) (1,-1) (2.0,0.4) (3.75,0.45)};
		%%%% Approximate boundary (edge)
		\draw [line width = 0.3mm,red, name path=true] plot[smooth] coordinates {(1,-1) (2.0,0.4)};
        %% surrogate boundary
        \draw[line width = 0.5mm,red] (2.6,2.2) -- (3.5,2.6);
		\draw[line width = 0.5mm,red] (1,1.73205) -- (2.6,2.2);
		\draw[line width = 0.5mm,red] (1,1.73205) -- (0,-0.5);
		\draw[line width = 0.5mm,red] (0,-0.5) -- (-1,-0.7);

        %%coordinates
		\coordinate (Ta1) at (0,-0.5);
		\coordinate (a1e) at (1,-1);
		\coordinate (m1) at ($0.5*(Ta1) + 0.5*(a1e)$);
		\coordinate (m1t) at ($0.5*(Ta1) + 0.5*(a1e) +(0.1,0.2)$);

		\coordinate (Ta2) at (1,1.73205);
		\coordinate (a2e) at (2,0.4);
		\coordinate (m2) at ($0.5*(Ta2) + 0.5*(a2e)$);
		\coordinate (m2t) at ($0.5*(Ta2) + 0.5*(a2e) +(0.25,0.1)$);
		
		\coordinate (mte) at ($0.5*(Ta1) + 0.5*(Ta2)$);
		\coordinate (g1e) at ($0.577*(Ta1) + 0.423*(mte)$);
		\coordinate (g1et) at ($0.577*(Ta1) + 0.423*(mte) - (0.22,-0.1)$);
		\coordinate (g2e) at ($0.577*(Ta2) + 0.423*(mte)$);
		\coordinate (g2et) at ($0.577*(Ta2) + 0.423*(mte) - (0.22,-0.1)$);

        %%arrows
	    \draw[->,line width = 0.5mm,-latex] (1,1.73205) -- (2,0.4);
	 	\draw[->,line width = 0.5mm,-latex] (0,-0.5) -- (1,-1);
	 	\draw[->,lightgray, line width = 0.25mm,-latex] (g1e) -- (1.21,-0.70);
	    \draw[->,lightgray, line width = 0.25mm,-latex] (g2e) -- (1.77,0.1);

		%% dots
		\filldraw[cyan] (m1) circle (1.5pt);
		\draw[yellow] (m1) circle (1.5pt);
		\filldraw[cyan] (m2) circle (1.5pt);
		\draw[yellow] (m2) circle (1.5pt);
		\filldraw[black] (g1e) circle (1.5pt);
		\draw[yellow] (g1e) circle (1.5pt);
		\filldraw[black] (g2e) circle (1.5pt);
		\draw[yellow] (g2e) circle (1.5pt);

		%% labels
		\node[text width=0.5cm] at (m1t) {${\color{cyan} \bs{q}_1^{\tmop{ext}}}$};
		\node[text width=0.5cm] at (m2t) {${\color{cyan} \bs{q}_2^{\tmop{ext}}}$};
		\node[text width=0.5cm] at (g1et) {$\bs{q}_{\ti{e}_1}$};
		\node[text width=0.5cm] at (g2et) {$\bs{q}_{\ti{e}_2}$};
		\node[text width=0.5cm] at (3.3,2.15) {${\color{red}\tG}$};
		\node[text width=3cm] at (1.25,2.25) {${\color{red}\tO}$};
		\node[text width=0.5cm] at (4.1,0.65) {${\color{blue}\G}$};
        %\node[text width=0.5cm] at (1.8,1.2) {\large$\bs{d}$};
        \node[text width=0.5cm] at (0.5,0.7) {\color{red} $\ti{e}$};
        \node[text width=0.5cm] at (0.1,-0.7) {$\ti{\bs{a}}_{1}$};
        \node[text width=0.5cm] at (1.35,1.635) {$\ti{\bs{a}}_2$};
        \node[text width=0.5cm] at (1.2,-1.05) {$\bs{a}_{1}^{\tmop{ext}}$};
        \node[text width=0.5cm] at (2.2,0.30) {$\bs{a}_2^{\tmop{ext}}$};
%        \node[text width=0.5cm] at (0.6,-0.6) {${e_{\ti{\bs{a}}_1}}$};
%        \node[text width=0.5cm] at (1.3,1.0) {${e_{\ti{\bs{a}}_2}}$};
        \node[text width=0.05cm] at (0.8,0.2) {${T_{\ti{e}}^{\tmop{ext}}}$};
        \node[text width=0.05cm] at (1.5,-0.4) {\color{red} ${{e^{\tmop{ext}}}}$};
        \node[text width=1.5cm] at (0.0,1.2) {\color{red} $\ti{T}_{\ti{e}} \subset \tO$};
%        \node[text width=0.05cm] at (1.8,1.2) {$T^{\tmop{ext}}_2$};
%        \node[text width=0.05cm] at (2.3,0.25) {\color{Green} ${e_2^{\tmop{ext}}}$};
%        \node[text width=1.5cm] at (1.8,2.35) {\color{red} $\ti{T}_2 \subset \tO$};
%        \node[text width=0.75cm,fill=green!15!] at (1.75,2) {\large$\bs{T^{\tmop{ext}}_E}$};
        \node[text width=0.5cm] at (0.5,-1.0) {$\bs{d}_{\ti{\bs{a}}_1}$};
        \node[text width=0.5cm] at (2.0,0.8) {$\bs{d}_{\ti{\bs{a}}_2}$};

		\end{tikzpicture}
		\caption{An edge $\ti{e} \in \tG$, the attached element $\ti{T}_{\ti{e}} \in \ti{\mathcal{T}}_h$, and the attached extended element ${T_{\ti{e}}^{\tmop{ext}}} \in \ti{\mathcal{T}}_h^{\tmop{ext}}$. The above sketch also features the geometric construction for the integration of the variational forms on the extended element $T^{\tmop{ext}}$. Observe that the positions of the quadrature points on the ``lateral'' edges of ${T_{\ti{e}}^{\tmop{ext}}}$ are $\bs{q}_1^{\tmop{ext}}=\ti{\bs{a}}_{1}+1/2\bs{d}_{\ti{\bs{a}}_1}$ and $\bs{q}_2^{\tmop{ext}}=\ti{\bs{a}}_{2}+1/2\bs{d}_{\ti{\bs{a}}_2}$, respectively. Instead, $\bs{q}_{\ti{e}_1}$ and $\bs{q}_{\ti{e}_2}$ are Gauss quadrature points along the edge $\ti{e}$.
		}
		\label{fig:SBM3}
\end{figure}
In summary, the weak form~\ref{eq:SBM_var_gen} is approximated as 
\begin{multline}
\label{eq:SBM_var_gen_simple}
( \nabla w_h  \, , \,  \nabla u_h )_{\tO}
\, + \, 
\sum_{\ti{e} \subset \tG}
\avg{ \nabla w_h  \, , \,  \nabla u_h \, H_{\ti{e}}  }_{\ti{e}} 
\, - \, 
\sum_{\ti{e} \subset \tGN}
\avg{ \oS(w_h) \, , \, h_N(M_h(\ti{\bs{x}})) \, j_{\ti{e}}  }_{\ti{e}} 
\\
\; - \; 
\sum_{\ti{\bs{a}} \in \mathcal{N}(\tG)}  
\frac{|\bs{d}_{\ti{\bs{a}}}|}{2} \, 
\bigg(
\jumpb{ \nabla w_h \cdot \bs{d}_{\ti{\bs{a}}} }_{\ti{\bs{a}}}
\cdot 
\avgb{ \nabla u_h }_{\ti{\bs{a}}} 
\; + \theta \,
\avgb{ \nabla w_h }_{\ti{\bs{a}}} 
\cdot 
\jumpb{ \nabla u_h \cdot \bs{d}_{\ti{\bs{a}}} }_{\ti{\bs{a}}}
\bigg)
\\ 
\; - \;
\sum_{\ti{e} \subset \tGD}
\bigg(
\avg{ \oS(w_h) \, , \,  (\nabla u_h \cdot \bs{n}) \, j_{\ti{e}} }_{\ti{e}}
\, + \theta \, 
\avg{ \nabla w_h \cdot \bs{n} \, , \, (\oS(u_h) -u_D) \, j_{\ti{e}}  }_{\ti{e}} 
\bigg)
\\ 
\; + \;
\sum_{\ti{\bs{a}} \in \mathcal{N}(\tG)}  
\frac{\gamma}{4 h} \, |\bs{d}_{\ti{\bs{a}}}| \, 
\jumpb{ \nabla w_h \cdot \bs{d}_{\ti{\bs{a}}} }_{\ti{\bs{a}}}
\cdot 
\jumpb{ \nabla u_h \cdot \bs{d}_{\ti{\bs{a}}} }_{\ti{\bs{a}}}
+\sum_{\ti{e} \subset \tGD}
\frac{\gamma}{h} \, 
\avg{  \oS(w_h) \, , \, (\oS(u_h) -u_D) \, j_{\ti{e}}  }_{\ti{e}} 
\\
\; = \;
( w_h \, , \, f )_{\tO}
\, + \, 
\sum_{\ti{e} \subset \tG}
\avg{ w_h \, , \, f \, H_{\ti{e}}  }_{\ti{e}} 
\; .
\end{multline}

\begin{rem}
Observe that the shape functions utilized by the proposed method are a partition of unity. Furthermore, the proposed formulation is exact if the solution is affine, because the gradient is globally constant and the Taylor expansion is exact in this case.
\end{rem}

\section{Theoretical analysis}
\label{sec:theo}

As already alluded to at the beginning of Section~\ref{sec:surr_dom_bndry}, the theoretical analysis will be performed only for two-dimensional triangular grids. With some appropriate adjustment, the analysis could be extended to uniform Cartesian grids, but we prefer to omit this step for the sake of brevity and to avoid complex notation.
For the sake of simplicity and without loss of generality, we shall only consider here problem~\eqref{eq:strong} with homogeneous boundary conditions, that is the case of $u_D=0$ and $h_N=0$.
Problem~\eqref{eq:strong} is well posed for any $f \in L^2 (\Omega)$, and we assume that the boundary is sufficiently smooth to have elliptic regularity, that is the weak solution $u$ is assumed in $H^2 (\Omega)$ with $| u |_{2, \Omega} \leqslant C \| f \|_{0,\Omega}$.
This choice is motivated by the fact that we would ultimately like to study the $L^2$-optimality of the numerical approximation $u_h$ to $u$ in the case of Dirichlet and Neumann boundary conditions, and this can be done only in the case in which $u$ is at least in $H^2 (\Omega)$.

As already mentioned in Remark~\ref{rem:edge_ext}, in principle the quadrilaterals $T_{\ti{e}}^{\tmop{ext}}$ have curved edges $e^{\tmop{ext}}$, which are approximated with straight edges between nodes.
The union of these straight edges forms the boundary $\G^h=\partial \Om^h$ of the approximate domain $\Om^h \supset \tO$. The measure of $\Om^h$ approximates the measure of $\Om$ up to an error $O(h^2)$.
Neglecting the discrepancy between $\Om^h$ and $\Om$ is normally acceptable in the case of piecewise-linear (globally continuous) finite element spaces, since the solution error in the natural and $L^2$ norms are typically of order $O(h)$ and $O(h^2)$, respectively.
Hence, for the moment, we neglect this fine level of approximation and develop our theory as if one can compute the length of the edge $|e^{\tmop{ext}}|$ and the area of the quadrilateral $|T_{\ti{e}}^{\tmop{ext}}|$ exactly. Moreover, the Jacobian of the mapping from an edge $\ti{e}$ on the surrogate boundary to the corresponding edge $e^{\tmop{ext}}$ is approximated according to~\eqref{eq:incr_approx}.
Taking these errors into account would make the theoretical analysis much more tedious.
While this is a necessity in the case of higher-order approximation spaces, these errors can be neglected for the case of piecewise-linear interpolation spaces considered here.   
In the analysis that follows we will consider the following bilinear form,  
\begin{align}
a_h (w, v) 
&=\;
(\nabla w , \, \nabla v)_{\tO} 
\, + \, (\nabla w \,, \, \nabla v)_{\ti{\mathcal{T}}_h^{\tmop{ext}}}
\nonumber \\
&\phantom{=}\;
\, - \, \langle \jumpb{w} \,, \, \avgb{\nabla v} \rangle_{\mathcal{E}^{\tmop{ext} ; o}} 
\, - \, \theta \, \langle \jumpb{v} \,, \, \avgb{\nabla w}\rangle_{\mathcal{E}^{\tmop{ext} ; o}} 
\, + \, \langle \gamma \, h^{-1} \, \jumpb{w} \,, \, \jumpb{v} \rangle_{\mathcal{E}^{\tmop{ext} ; o}}
\nonumber \\
&\phantom{=}\;
\, - \, \sum_{\ti{e} \subset \tGD} \langle \oS(w) \,, \oS (\nabla v) \cdot \bs{n} \, j_{\ti{e}} \rangle_{\ti{e}} 
\, - \, \theta \, \sum_{\ti{e} \subset \tGD} \langle \oS (\nabla w) \cdot \bs{n} \,, \, \oS (v) \, j_{\ti{e}} \rangle_{\ti{e}} 
\nonumber \\
&\phantom{=}\;
\, + \, \sum_{\ti{e} \subset \tGD} \langle \gamma \, h^{-1} \, \oS (w) \, , \, \oS (v) \, j_{\ti{e}} \rangle_{\ti{e}}
\; ,
\end{align}
which is an ``intermediate step'' between~\eqref{eq:SBM_var_gen} and~\eqref{eq:SBM_var_gen_simple}, in the sense that it relates to~\eqref{eq:SBM_var_gen_simple} but, to simplify the notation, the terms \eqref{eq:ext_jump1},~\eqref{eq:ext_jump2}, and~\eqref{eq:ext_jump3} are not approximated and~\eqref{eq:incr_approx} is applied as for~\eqref{eq:simple1},~\eqref{eq:simple2}, and~\eqref{eq:simple3}.
Here, we interpret the shift operator $\oS$ as 
\begin{equation}
\label{eq:new_shift}
\oS (v^h) \; = \;  v^h (\mathbf{M}_h (\tilde{\bs{x}})) \; , 
\end{equation}
an idea we already explored in~\cite{collins2023penalty,visbech2025spectral}.
When applied to a piecewise linear function $v_h \in V_h$, this definition of $\oS$ is consistent with (i.e., identical to) definition~\eqref{eq:oS_taylor}, which is based on Taylor expansions.
Moreover, for $v_h \in V_h$, $\oS
(\nabla v_h) = \nabla v_h$ on any edge $\ti{e} \subset \tGD$, since $\nabla
v_h$ is piecewise constant. Thus the SBM solution $u_h \in V_h$ of problem~\eqref{eq:strong} with homogeneous boundary conditions satisfies
\begin{equation}
  \label{weakSBM} a_h (v_h, u_h) = (v_h \,, \,
  f)_{\tO} \, + \, \sum_{\ti{e} \subset \tG}
  \langle v_h, f \, H_{\ti{e}}
  \rangle_{\ti{e}}, \quad \forall v_h \in V_h \; .
\end{equation}
Interpreting the shift operator $\oS$ as in~\eqref{eq:new_shift} allows us to write that the exact solution $u \in H^2 (\Omega)$ to the same problem satisfies
\begin{equation}
  \label{weakFormh} a_h (v_h, u ) = (v_h \,, \,
  f)_{\Omega} \,, \quad \forall v_h \in V_h \; ,
\end{equation}
where we also have neglected errors in the geometric approximation of the boundaries according to~\eqref{eq:incr_approx}, so that 
\begin{align}
\label{cheat}
    \sum_{\ti{e} \subset \tGD} \langle \oS (v_h) \,, \oS (\nabla u) \cdot \bs{n} \, j_{\ti{e}} \rangle_{\ti{e}} 
    &=\;
   \sum_{\ti{e} \subset \tGD} \int_{\ti{e}} \oS (v_h) \, \oS
   (\nabla u) \cdot \bs{n} \, \frac{|  \mathrm{d} \mathbf{M}_h
   (\tilde{\bs{x}}) |}{| \mathrm{d} \tilde{\bs{x}} |}  
   \nonumber \\
   &=\;
   \sum_{\ti{e} \subset \tGD} \int_{\mathbf{M}_h \left( \ti{e} \right)} v_h \, \nabla u \cdot \bs{n}
   \nonumber \\
   &=\;
   \langle v_h \, , \nabla u \cdot \bs{n} \rangle_{0,\tGD} \; .
\end{align} 
An important assumption is now in order:
\begin{assumption}
\label{assu:je_bound}
The term $j_{\ti{e}}$ is bounded above, that is $|j_{\ti{e}}|\le C_D$.
\end{assumption}
Observe that Assumption~\ref{assu:je_bound} is normally satisfied in engineering computations, even the ones with the most challenging geometries. This is because the grids utilized in practical computations produce a surrogate boundary that broadly captures the shape of the true domain, for any grid resolution.

The first results in our analysis are two lemmas on the stability of the bilinear form~\eqref{weakSBM}, followed by an error estimate in the natural norm.

\begin{lemma}[\textbf{stability of the penalty-based formulations}]
\label{LemCoer}
    The bilinear form $a_h$ is coercive on $V_h$ for any $\theta
  \in \mathbb{R}$ and $\gamma > 0$ sufficiently large. More precisely, there exist
  $\alpha > 0$ and $\gamma_0 > 0$, depending only on the regularity of the
  mesh and on $\theta$, such that for all $\gamma \geqslant \gamma_0$ and all
  $v_h \in V_h$
  \[ a_h (v_h, v_h) \geqslant \alpha \interleave v_h \interleave_h^2 \]
  with
  \begin{align}
  	  \label{triple}
      \interleave v_h \interleave_h^2 
      & := \, 
      | v_h |_{1, \Omega}^2 + \frac{1}{h}
     {\left\| \jumpb{v_h} \right\|_{0,\mathcal{E}^{\tmop{ext} ; o}}^2}  +
     \frac{1}{h} {\left\| \oS (v_h) \right\|_{0,\tGD^{[j]}}^2} 
	\\[.2cm]
  	\label{norm-jG}
	\left\| w \right\|_{0,\tGD^{[j]}}
      & := \, 
      \sum_{\ti{e} \subset \tGD } 	\left\| w \right\|_{0,j[\ti{e}]}
      \; ,
	\\
	\label{norm-je}
	\left\| w \right\|_{0,j[\ti{e}]}
      & := \, 
      \left\| w {\sqrt{j_{\ti{e}}}} \right\|_{0,\ti{e}} 
      \; .
  \end{align} 
\end{lemma}

\begin{proof}

\noindent We start by stating a trace inverse inequality (which implicitly relies on Assumption~\ref{assu:je_bound})
  \[ \left\| \avgb{\nabla v_h} \right\|_{0,\mathcal{E}^{\tmop{ext} ; o}}^2 +
     {\| \nabla v_h \|_{0,\tGD^{[j]}}^2}  \leqslant \frac{C_{\tmop{inv}}^2}{h} {| v_h
     |^2_{1, \Omega}}  \]
  and we conclude
  \begin{align*}
  a_h (v_h, v_h) 
  & \geqslant \; 
  | v_h |_{1, \Omega}^2 - (1 + \theta) \langle
     \jumpb{v_h} \,, \, \avgb{\nabla v_h}
     \rangle_{\mathcal{E}^{\tmop{ext} ; o}} 
     +\langle \gamma \, h^{-1}  \jumpb{v_h}, \jumpb{v_h} \rangle_{0,\mathcal{E}^{\tmop{ext} ; o}}
  \nonumber \\
  & \phantom{\geqslant} \; 
  \; + \;
  (1 + \theta) \sum_{\ti{e} \subset \tGD } \langle \oS (v_h) \,, \nabla v_h \cdot \bs{n} \, j_{\ti{e}} \rangle_{0,\ti{e}} 
  \; + \;
  \sum_{\ti{e} \subset \tGD } \langle \gamma \, h^{-1} \,  \, \oS (v_h) \,, \, \oS (v_h) j_{\ti{e}} \rangle_{0,\ti{e}} 
  \nonumber \\
  & \geqslant \; 
  \left( 1 - \frac{\varepsilon}{2} C_{\tmop{inv}}^2 \right) | v_h |_{1, \Omega}^2 
  \; + \;
  \left( \frac{\gamma}{h} \,  - \frac{(1 + \theta)^2}{2 \varepsilon h} \right) 
     \left( \left\| \jumpb{v_h}\right\|_{0,\mathcal{E}^{\tmop{ext} ; o}}^2  
  \; + \;
  \left\| \oS (v_h) \right\|_{0,\tGD^{[j]}}^2 \right) 
       \; .
  \end{align*}
  This gives the announced estimate taking $\varepsilon$ sufficiently small,
  e.g. $\varepsilon = C_{\tmop{inv}}^{- 2}$, and $\gamma_0= \frac{1+(1 + \theta)^2 C_{\tmop{inv}}^{2}}{2} $, so that $\alpha=\frac{1}{2}$.
  \end{proof}

The above lemma does not cover the antysymmetric penalty-free case, i.e.
$\theta = - 1$, $\gamma = 0$. Indeed, the bilinear form $a_h$ is coercive with
respect to the $H^1$ seminorm, but not the norm $\interleave \cdot
\interleave_h$. However, we can replace the notion of coercivity by the
inf-sup property, as stated in the following lemma. This can be done under
additional assumption about the Dirichlet part of the surrogate boundary:

\begin{assumption}\label{AssForPenFree} There exist positive constants $c_D$ and $C_D$ such that on any	edge $\ti{e} \subset \tGD$ we have $c_D \leqslant j_{\tilde{e}} \leqslant C_D$ and $\tilde{\bs{n}} \cdot \bs{n}	\geqslant c_D$, where, at any point $\ti{\bs{x}}\in\ti{e}$, $\tilde{\bs{n}}$ is the normal to $\ti{e}$ that points outside of $\tO$, while $\bs{n}$ is the normal to the actual boundary $\Gamma$ evaluated at $M_h(\ti{\bs{x}})$.
\end{assumption} 
\noindent Observe that condition $\tilde{\bs{n}} \cdot \bs{n} \geqslant c_D$ in Assumption \ref{AssForPenFree} can be violated, but typically only on isolated edges.
We conjecture that a finer analysis could then help to prove the upcoming result even without this assumption.
\begin{lemma}[\textbf{Stability of the penalty-free formulation}]
  \label{LemInfSup}Assume $\theta = - 1$, $\gamma = 0$. There exists $\alpha
  > 0$, depending only on the regularity of the mesh, such that
  \begin{eqnarray*}
    \inf_{v_h \in V_h} \sup_{w_h \in V_h}  \frac{a_h (w_h, v_h)}{\interleave
    w_h \interleave_h \interleave v_h \interleave_h} \geqslant \alpha
    \; .
  \end{eqnarray*}
\end{lemma}

\begin{proof}

\noindent Taking $w_h=v_h$, we immediately obtain the coercivity of $a_h$ with respect to the $H^1$ seminorm:
  \begin{eqnarray*}
    a_h (v_h, v_h) \geqslant | v_h |_{1, \Omega}^2 \; , \qquad  \forall v_h \in V_h \; .
  \end{eqnarray*}
  Now, for any $v_h \in V_h$, introduce $g_h$ as the piecewise-linear function
  on $\tilde{\mathcal{T}}_h$ taking the same values as $v_h$ at the boundary
  nodes on $\tGD$ and set to zero on at all the other nodes of the mesh.
  Reexamining the proof of~\cite[Lemma 1]{collins2023penalty}, we observe that, for any boundary edge $\ti{e} \subset \tGD$,
  \begin{equation}\label{PFArt14a}
	\langle \nabla g_h \cdot \tilde{\bs{n}} \,,
	\, \oS (v_h) \rangle_{\ti{e}} \geqslant \frac{c_1}{h} \|v_h
	\|_{0, \tilde{e}}^2 - c_2 |v_h |_{1, T_{\tilde{e}}}^2
	\; ,
\end{equation}
where $T_{\tilde{e}}$ denotes the element of $\tilde{\mathcal{T}}_h$ attached to $\tilde{e}$.  This bound corresponds to equation~(14a) in the proof of~\cite[Lemma 1]{collins2023penalty}, taken in an element-wise version (prior to summing over all the boundary edges). In order to adapt (\ref{PFArt14a}) to the needs of the current article, in particular to pass from the normal vector $\tilde{\bs{n}}$ to $\bs{n}$, we need to invoke Assumption \ref{AssForPenFree}. Introducing the unit tangent vector $\tilde{\bs{\tau}}$ on $\tilde{e}$ alongside the normal vector $\tilde{\bs{n}}$, we can derive from (\ref{PFArt14a}):
\begin{multline}\label{PFstep1}
	\langle \nabla g_h \cdot \bs{n} \,, \, \oS
	(v_h) j_{\tilde{e}} \rangle_{\ti{e}} = \langle \nabla g_h \cdot
	\tilde{\bs{n}} (\tilde{\bs{n}} \cdot \bs{n})
	\,, \, \oS (v_h) j_{\tilde{e}} \rangle_{\ti{e}}
	+ \langle \nabla g_h \cdot \tilde{\bs{\tau}}
	(\tilde{\bs{\tau}} \cdot \bs{n}) \,,
	\, \oS (v_h) j_{\tilde{e}} \rangle_{\ti{e}} 
	\\
	\geqslant 
	c_D \frac{c_1}{h} \|v_h \sqrt{j_{\tilde{e}}} \|_{0, \tilde{e}}^2 
	- 
	c_D c_2 C_D
	|v_h |_{1, T_{\tilde{e}}}^2 
	- \left|\langle \nabla v_h	\cdot \tilde{\bs{\tau}} (\tilde{\bs{\tau}} \cdot \bs{n})
	\,, \, \oS (v_h) j_{\tilde{e}} \rangle_{\ti{e}} \right| 
	\; .
\end{multline}
In the last term above we have replaced $\nabla g_h \cdot
\tilde{\bs{\tau}}$ by $\nabla v_h \cdot \tilde{\bs{\tau}}$ since
$g_h = v_h$ on $\tilde{e}$. By scaling arguments and Young inequality, we
can further bound this term as
\begin{equation}\label{PFstep2}
	\left|\langle \nabla v_h \cdot \tilde{\bs{\tau}} (\tilde{\bs{\tau}}
	\cdot \bs{n}) \,, \, \oS (v_h)j_{\tilde{e}} \rangle_{\ti{e}} \right|
	\leqslant  \frac{c}{\sqrt{h}} |v_h |_{1,T_{\tilde{e}}} \| \oS (v_h) \|_{0,j[\ti{e}]} 
	\leqslant \frac{C}{h} |v_h |_{1, T_{\tilde{e}}}^2 + c_D
	\frac{c_1}{4h} \| \oS (v_h) \|_{0,j[\ti{e}]}^2
\end{equation}
with $h$-independent constants $c$ and $C$. Similarly,
\begin{equation}\label{PFstep3}
	\|v_h \|_{0,j[\ti{e}]}^2 \geqslant \left( \| \oS (v_h) \|_{0,j[\ti{e}]} -\| \left( \oS (v_h) - v_h \right) \|_{0,j[\ti{e}]} \right)^2 
	\geqslant \frac{3}{4} \| \oS (v_h) \|_{0,j[\ti{e}]}^2 - Ch |v_h |_{1, T_{\tilde{e}}}^2
	\; .
\end{equation}
Putting (\ref{PFstep2}) and (\ref{PFstep3}) into (\ref{PFstep1}) and summing over $\ti{e} \subset \tGD$ we arrive at
\begin{equation}\label{FPstep4}
	\sum_{\ti{e} \subset \tGD} \langle \oS (\nabla g_h) \cdot \bs{n}
	\,, \, \oS (v_h) \, j_{\ti{e}}
	\rangle_{\ti{e}} 
	\geqslant \frac{\tilde{c}_1}{h} \| \oS(v_h) \|_{0,\tGD^{[j]}}^2 
	- \tilde{c}_2 |v_h |_{1,\tilde{\Omega}}^2
\end{equation}
with $\tilde{c}_1 = c_D c_1 / 2$ and $\tilde{c}_2$ regrouping the constants
in \eqref{PFstep1}, \eqref{PFstep2} and \eqref{PFstep3}.

We shall also need to bound some norms of $g_h$ by those of $v_h$. To this end, we start from the following bound, analogous to equation~(14b) in the proof of~\cite[Lemma 1]{collins2023penalty}: 
\begin{equation}\label{PFArt14b}
	|g_h |_{1, \Omega}^2 + \frac{1}{h} \sum_{\ti{e} \subset \tGD} \|\oS(g_h) \|_{0,j[\ti{e}]}^2
	\leqslant \frac{c_3}{h} \sum_{\ti{e} \subset \tGD} \|v_h \|_{0,j[\ti{e}]}^2 
	\; .
\end{equation}
The original version of this bound is slightly different and does not include the factors $j_{\tilde{e}}$, nor the shift operator $\oS$. The present version is easily proven by scaling, having in mind Assumption \ref{AssForPenFree}. 
  Note that
\begin{eqnarray}
	\label{termc4} \frac{1}{h} {\left\| \jumpb{g_h}
		\right\|_{0,\mathcal{E}^{\tmop{ext} ; o}}^2}  \leqslant c_4 | g_h |_{1,\Omega}^2
		\; .
\end{eqnarray}
Indeed, on any exterior (fictitious) edge $e_a$ adjacent to a boundary node
$a$, so that $e_a$ is shared by extensions of interior mesh triangles, say
$T_1, T_2$, we have
\begin{eqnarray*}
	| \nabla g_h |_{T_1} - \nabla g_h |_{T_2}  |
	\leqslant \frac{c_4}{h} | g_h |_{1, T_1 \cup T_2}^2
\end{eqnarray*}
with an $h$-independent constant $c_4$.
Since $\jumpb{g_h} (a) = 0$, we can estimate
\begin{eqnarray*}
	\left\| \jumpb{g_h} \right\|_{e_a}^2 \leqslant h^2  \| \nabla g_h |_{T_1}
	 - \nabla g_h |_{T_2} \|_{e_a}^2 \leqslant c_4 h | g_h |_{1, T_1 \cup T_2}^2
	 \; .
\end{eqnarray*}
Summing this over all such edges gives (\ref{termc4}). Thanks to (\ref{termc4}) and (\ref{PFstep3}) taken in the opposite sense (still valid by scaling), we can rewrite (\ref{PFArt14b}) as
\begin{eqnarray}\label{majgh}
	\interleave g_h \interleave_h^2 + h \|\avgb{\nabla g_h}\|_{\mathcal{E}^{\tmop{ext} ; o}}^2 
	\leqslant \frac{\tilde c_3}{h} \sum_{\ti{e} \subset \tGD} \|\oS(v_h) \|_{0,j[\ti{e}]}^2
	\; . 
\end{eqnarray}

  Take any $\lambda > 0$, and observe, using (\ref{FPstep4}) and (\ref{majgh}),
  \begin{align*}
    a_h (v_h + \lambda g_h, v_h)
    &\geqslant \;
     (1 - \tilde c_2 \lambda) | v_h |_{1,\Omega}^2 + \tilde c_1  \frac{\lambda}{h} \|\oS(v_h) \|_{0,\tGD^{[j]}}^2
     +\lambda (\nabla g_h, \nabla v_h)_\Omega \\
		&\phantom{\geqslant} \;    - \lambda \, \langle \jumpb{g_h} \,, \, \avgb{\nabla v_h} \rangle_{\mathcal{E}^{\tmop{ext} ; o}}
    + \lambda \, \langle \jumpb{v_h} \,, \, \avgb{\nabla g_h} \rangle_{\mathcal{E}^{\tmop{ext} ; o}}
    - \lambda \sum_{\ti{e} \subset \tGD} \langle \oS (g_h) \,, \oS (\nabla v_h) \cdot \bs{n} \, j_{\ti{e}} \rangle_{\ti{e}}
    \nonumber \\
    &\geqslant \;
    (1 - \tilde c_2 \lambda) | v_h |_{1, \Omega}^2
    + \tilde c_1 \lambda \frac{1}{h} {\| \oS(v_h) \|_{0,\tGD^{[j]}}^2} \\
    &\phantom{\geqslant} \;   - \lambda \sqrt{\frac{\tilde c_3}{h}} {\|\oS(v_h)  \|_{0,\tGD^{[j]}}}  \times 
    \nonumber \\
    & \phantom{\geqslant} \; \qquad \times
    \left(| v_h |_{1, \Omega}^2 + 
      h\| \avgb{\nabla v_h} \|_{\mathcal{E}^{\tmop{ext} ; o}}^2
    + \frac 1h \| \jumpb{v_h} \|_{\mathcal{E}^{\tmop{ext} ; o}}^2
    + h\| \oS (\nabla v_h) \cdot \bs{n} \|^2_{0,\tGD^{[j]}} \right)^{1/2}
    \; .
  \end{align*}
  All the terms in the parentheses in the last line can be bounded by $C_I^2|v_h|_{1,\Omega}^2$ with an $h$-independent constant $C_I$ thanks to trace inverse inequalities and, in particular, inequality (\ref{termc4}) applied to $v_h$ instead of $g_h$. This
  gives with the help of Young inequality for any $\varepsilon > 0$, followed again by (\ref{termc4}),  
  \begin{align*}
    a_h (v_h + \lambda g_h, v_h) &\geqslant \left( 1 - \tilde c_2 \lambda - \frac{C_I^2
    \lambda}{2 \varepsilon} \right) | v_h |_{1, \Omega}^2 + \left( \tilde c_1 -
    \frac{\tilde c_3 \varepsilon}{2} \right) \lambda \frac{1}{h} {\| \oS(v_h) \|_{0,\tGD^{[j]}}^2} \\
    &\geqslant \left( \frac 12 - \tilde c_2 \lambda - \frac{C_I^2\lambda}{2 \varepsilon} \right) | v_h |_{1, \Omega}^2 
    + \frac{1}{2 c_4 h} {\left\| \jumpb{v_h} \right\|_{0,\mathcal{E}^{\tmop{ext} ; o}}^2}
    + \left( \tilde c_1 - \frac{\tilde c_3 \varepsilon}{2} \right) \lambda \frac{1}{h} {\| \oS(v_h) \|_{0,\tGD^{[j]}}^2} \\
    & \geqslant c_5 \interleave v_h \interleave_h^2
  \end{align*}
  with
\begin{eqnarray*}
	c_5 = \min \left( \frac{1}{2} - \tilde c_2 \lambda - \frac{C_I^2 \lambda}{2
		\varepsilon} ,
	\frac{1}{2 c_4}, \left( \tilde c_1 - \frac{\tilde c_3 \varepsilon}{2} \right) \lambda
	\right)
\end{eqnarray*}
assuming that $\varepsilon$ and $\lambda$ are chosen small enough so that $c_5 > 0$.
Having fixed $\lambda$ as above, we deduce from (\ref{majgh})
  \begin{eqnarray*}
    \interleave v_h + \lambda g_h \interleave_h \leqslant c_6 \interleave v_h
    \interleave_h
    \; .
  \end{eqnarray*}
  This gives the announced inf-sup with $\alpha = c_5 / c_6$ taking $w_h = v_h
  + \lambda g_h$.
\end{proof}

\begin{thm}[\textbf{convergence in the natural norm}]
  Let the parameters $\theta, \gamma$ be chosen either as
  in Lemma~\ref{LemCoer} or in Lemma~\ref{LemInfSup}. Suppose also that $f
  \in H^1 (\Omega)$ and $u \in H^2 (\Omega)$. 
  %{\color{blue} [GS: quick question, but if we choose $f \in H^1 (\Omega)$, shouldn't $u \in H^3(\Omega)$ by the regularity property of the Laplace operator? We could still use the estimate below based on the $H^2$-seminorm, of course.]}
  % AL: I've removed this question as we've already discussed it some time ago. I think that we can leave the text as is here, since to apply the regularity theorem, one probably needs more restrictive assumptions on the regularity of $\Omega$ than what we stated.  
  Then the following $H^1$-error
  estimate holds:
  \[ | u - u_h |_{1, \Omega} \leqslant \interleave u - u_h \interleave_h \leqslant  Ch (|u|_{2, \Omega} +\|f\|_{1, \Omega}) \; .
  \]
\end{thm}

\begin{proof}

\noindent Subtracting~\eqref{weakFormh} from \eqref{weakSBM} yields the ``modified Galerkin orthogonality'' statement
  \begin{equation}
    \label{GalOrtho} a_h (v_h, u - u_h)
    \;=\;
    \sum_{\ti{e} \subset \tG} \left(
    (v_h, f)_{T_{\ti{e}}^{\tmop{ext}}} - \left\langle v_h, f \, \right\rangle_{\ti{e}} \, H_{\ti{e}}
	\right) \; ,
	\qquad \forall v_h \in V_h
  \; .
  \end{equation}
  Introducing, on any edge $\tilde{e}$, the averages $\bar{f}$ and $\bar{v}_h$
  of $f$ and $v_h$ on that edge, we have that $\int_{\ti{e}} f = \bar{f} \, \left|
  \ti{e} \right|$ and $\int_{\ti{e}} v_h = \bar{v}_h \, \left| \ti{e}
  \right|$, and we continue the above calculations as
  \begin{align}
     a_h (v_h, u - u_h)
     & = \;
     \sum_{\ti{e} \subset \tG} \left(
     (v_h - \bar{v}_h, f)_{T_{\ti{e}}^{\tmop{ext}}} + (\bar{v}_h, f - \bar{f})_{T_{\ti{e}}^{\tmop{ext}}} +
     (\bar{v}_h, \bar{f})_{T_{\ti{e}}^{\tmop{ext}}} - \left\langle v_h - \bar{v}_h, f
     \, \right\rangle_{\ti{e}} \, H_{\ti{e}} -
     \bar{v}_h \bar{f} \, |T_{\ti{e}}^{\tmop{ext}} |
     \right)
     \nonumber \\
     & = \;
     \sum_{\ti{e} \subset \tG} \left( (v_h - \bar{v}_h, f)_{T_{\ti{e}}^{\tmop{ext}}} +
     (\bar{v}_h, f - \bar{f})_{T_{\ti{e}}^{\tmop{ext}}} - \left\langle v_h - \bar{v}_h,
     f \, \right\rangle_{\ti{e}} \, H_{\ti{e}} \right)
     \nonumber \\
     & \leqslant \;
     Ch |v_h |_{1, \Omega} \|f\|_{0,\Omega} + Ch \|v_h \|_{0,\Omega}
     |f|_{1, \Omega} + Ch^2  |v_h |_{1, \tG}
     \|f\|_{\tG}
     \; .
  \end{align}
  To derive the last bound, we have used the following inequalities involving any boundary edge $\ti{e} \in \tG$, shared by the element $\tilde{T}_{\ti{e}} \in \ti{\cT}_h$ and the (curved) quadrilateral $T_{\ti{e}}^{\tmop{ext}}$ (see Fig.~\ref{fig:SBM3} for an illustration):
  \begin{equation}
    \|v_h - \bar{v}_h \|_{T_{\ti{e}}^{\tmop{ext}}} \leqslant Ch |v_h |_{1, \tilde{T}_{\ti{e}}} \; ,
    \qquad
    \| \bar{v}_h \|_{T_{\ti{e}}^{\tmop{ext}}} \leqslant C \|v_h \|_{\tilde{T}_{\ti{e}}} \; ,
    \label{LocEstFatBan}
  \end{equation}
  \begin{equation}
  \|f - \bar{f} \|_{T_{\ti{e}}^{\tmop{ext}}} \leqslant Ch |f|_{1, \tilde{T}_{\ti{e}} \cup T_{\ti{e}}^{\tmop{ext}}}
  \; , \qquad
  \|v_h - \bar{v}_h \|_{\tilde{e}} \leqslant Ch |v_h |_{1, \tilde{e}} \; .
  \end{equation}
  All of these can be easily proven by scaling arguments and Poincar{\'e}-type
  inequalities. We have also used the fact that $H_{\ti{e}}
  \leqslant Ch$. Now, adding to these the trace inverse inequality
  $|v_h |_{1,  \tG} \leqslant \frac{C}{\sqrt{h}} |v_h |_{1,\tilde{\Omega}_h}$
  and the trace inequality
  $\|f\|_{\tG} \leqslant \|f\|_{1, \tilde{\Omega}_h}$
  yields
  \begin{equation}\label{ahNonConf} a_h (v_h, u - u_h) \leqslant Ch \| v_h \|_{1, \Omega} \|f\|_{1, \Omega}
     \leqslant Ch \interleave v_h \interleave_h \|f\|_{1, \Omega} \; .\end{equation}
  The passage from the full $H^1$ norm of $v_h$ to its triple norm in the last line, is justified by the Poincar\'e-type inequality (valid since $\Gamma_D$ is assumed of positive measure): for any $v\in H^1(\Omega))$
  $$
    \| \, v \, \|_{0, \Om} \leq C_P \, \left( \| \, \nabla v \, \|_{0,\Om} + |\Gamma_D|^{-1/2} \| \, v \, \|_{0, \Gamma_D} \right) \; .
  $$
 From this and the fact that the boundary term on $\tGD$ in the triple norm (\ref{triple})  can be interpreted as the norm on $\Gamma_D$ when applied to $v_h\in\tilde V_h^{\tmop{ext}}(\Omega)$, we deduce that
  \begin{equation}
\label{eq:Poincare1bis}
 \| \, v_h \, \|_{0, \Om} \leq C_P \left( \| \, \nabla v_h \, \|_{0,\Om} + h^{-1/2} \| \,  \, v_h \, \|_{0, \tGD} \right) 
  \leq C_P \interleave v_h \interleave_h\; .
\end{equation}
  Introducing the interpolant $I_h u \in V_h$ (constructed by nodal interpolation over
  the mesh $\ti{\cT}_h$ inside $\ti{\Omega}_h$), we derive from (\ref{ahNonConf}) using either the coercivity of the form $a_h$ from Lemma \ref{LemCoer}, or the inf-sup property from Lemma \ref{LemInfSup}, depending on the choice of parameters $\theta$, $\gamma$:
  \begin{equation}
    \label{ahCea} 
    \alpha \interleave I_h u - u_h  \interleave_h
    \leq \sup_{w_h \in V_h}  \frac{a_h (w_h, I_h u - u_h)}{\interleave 	w_h \interleave_h}
    \leq \sup_{w_h \in V_h}  \frac{a_h (w_h, u- I_h u)}{\interleave 	w_h \interleave_h}
    + Ch \|f\|_{1, \Omega} 
    \; .
  \end{equation}
  Using the continuity of the form $a_h$ in the triple norm (evident by Cauchy-Schwarz and scaling), we arrive at:
  \[ \interleave I_h u - u_h \interleave_h \leqslant Ch\|f\|_{1,
     \Omega} + C \interleave u - I_h u \interleave_h \; . \]
  By the Bramble--Hilbert lemma~\cite{bramble1970estimation,bramble1971bounds} and scaling arguments, we have
  \[ \interleave u - I_h u \interleave_h \leqslant C | u - I_h u |_{1, \Omega}
     \leqslant Ch | u |_{2, \Omega} \; , \]
  This allows us to obtain the announced error estimate, thanks to the triangle inequality
  \begin{equation}
    \label{AprioTrip} \interleave u - u_h \interleave_h \leqslant \interleave
    u - I_h u \interleave_h + \interleave I_h u - u_h \interleave_h 
    \leqslant Ch\|f\|_{1,\Omega} + C \interleave u - I_h u \interleave_h
    \leqslant Ch (| u |_{2, \Omega} +\|f\|_{1, \Omega}) \; .
  \end{equation}
\end{proof}

We now turn to the $L^2$ error estimate, and we initially focus on the symmetric interior penalty method (i.e. for $\theta = 1$ and $\gamma$ sufficiently large).
This variant possesses the adjoint consistency property required for optimal convergence.
To treat the errors introduced by the approximations in the gap $\Omega \setminus \tO$, we enlarge it to the set $\Omega_h^{\Gamma, \tmop{fat}}$, which contains $\Omega \setminus \tO$ together with all elements $\ti{T} \in \ti{\mathcal{T}}_h$ adjacent to $\tG$. 
Note that $\Omega_h^{\Gamma, \tmop{fat}}$ is a thin layer of width of order $h$, a fact that will be used via the following lemmas:

\begin{lemma}\label{lemma:3}
  \label{Lemma2}For any $u \in H^1 (\Omega)$
  \begin{eqnarray*}
    \|u\|_{0, \Omega_h^{\Gamma, \tmop{fat}}} \leqslant C \sqrt{h} \|u\|_{1,\Omega} \; .
  \end{eqnarray*}
\end{lemma}

This lemma is proved in \cite[Lemma 4.10]{elliott2013}. The idea of the proof is to cover the band $\Omega_h^{\Gamma, \tmop{fat}}$ by curves $\mathcal{S}_{\eta} = \{x \in \Omega : \tmop{dist} (x, \partial \Omega) = \eta\}$ for $\eta\in(0,\delta)$ with $\delta$ of order $h$, apply the trace inequality on each $\mathcal{S}_{\eta}$, and then integrate on $\eta$.

% \begin{proof}

% \noindent Let, for $\delta > 0$ and $\eta \in (0, \delta)$,
%   \begin{eqnarray*}
%     \mathcal{O}_{\delta} = \{x \in \Omega : \tmop{dist} (x, \partial \Omega) <  \delta\} 
%     \qquad \mbox{and} \qquad
%     \mathcal{S}_{\eta} = \{x \in \Omega : \tmop{dist} (x, \partial \Omega) = \eta\}
%     \; .
%   \end{eqnarray*}
%   By the trace theorem~\ref{thm:TraceTheorem}, we have
%   \begin{eqnarray*}
%     \|u\|_{0, \mathcal{S}_{\eta}} \leqslant C \|u\|_{1, \Omega} \; .
%   \end{eqnarray*}
%   Squaring this inequality and integrating this $\eta \in (0, \delta)$ gives
%   \begin{eqnarray*}
%     \|u\|_{0, \mathcal{O}_{\delta}} \leqslant C \sqrt{\delta} \|u\|_{1,
%     \Omega}
%   \end{eqnarray*}
%   and the result follows because $\Omega_h^{\Gamma, \tmop{fat}} \subset
%   \mathcal{O}_{\delta}$ with $\delta$ of order $h$.
% \end{proof}

\begin{lemma}
  \label{Lemma3}For any $u_h \in V_h$
  \begin{eqnarray*}
    \| \nabla u_h \|_{0, \Omega_h^{\Gamma, \tmop{fat}}} \leqslant C \sqrt{h} 
    (\| \nabla u_h \|_{0, \Omega} + |u_h |_{2, h, \Omega})
  \end{eqnarray*}
  with
  \begin{eqnarray*}
    |u_h |_{2, h, \Omega}^2 = \sum_{T \in \ti{\mathcal{T}}_h \cup \mathcal{T}_h^{\tmop{ext}}  } |u_h |_{2, T}^2 +
    \sum_{e \in \ti{\mathcal{E}}^{o} \cup \mathcal{E}^{\tmop{ext};o} } \frac{1}{h} \|[\nabla u_h]\|_{0,e}^2
    \; .
  \end{eqnarray*}
\end{lemma}

\begin{proof}

\noindent The idea is to apply the preceding lemma to $\nabla u_h$, but the difficulty
  is that $\nabla u_h$ is not in $H^1(\Om)$, because of the discontinuities that arise in the gap $\Omega \setminus \tO$.
  To circumvent this issue, we can construct a continuous counterpart $U_h$ of $\nabla u_h$. For example, we can set
  $U_h =\mathcal{I}_h \nabla u_h$ where $\mathcal{I}_h$ is a Cl{\'e}ment
  interpolation operator to $\mathcal{P}^1$ functions.  Then $\| U_h \|_{0,\Omega} \leqslant C \| \nabla u_h \|_{0,\Omega}$,
      $\| \nabla u_h - U_h \|_{0,\Omega} \leqslant Ch |u_h |_{2, h, \Omega}$ and $\| \nabla U_h
  \|_{0, \Omega} \leqslant C |u_h |_{2, h, \Omega}$ by scaling arguments. 
  We conclude using Lemma~\ref{lemma:3}:
  \begin{eqnarray*}
    \| \nabla u_h \|_{0,\Omega_h^{\Gamma, \tmop{fat}}} \leqslant \| \nabla u_h -
    U_h \|_{0,\Omega_h^{\Gamma, \tmop{fat}}} + \|U_h \|_{0,\Omega_h^{\Gamma,
    \tmop{fat}}} \leqslant \| \nabla u_h - U_h \|_{0,\Omega} + C \sqrt{h} \|U_h
    \|_{1, \Omega} \; .
  \end{eqnarray*}
  With the interpolation estimates from above, this gives
  \begin{eqnarray*}
    \| \nabla u_h \|_{0,\Omega_h^{\Gamma, \tmop{fat}}} \leqslant C \left( h +
    \sqrt{h} \right) (\| \nabla u_h \|_{0,\Omega} + |u_h |_{2, h, \Omega}) \; .
  \end{eqnarray*}
\end{proof}

\begin{thm}[\bf optimal $L^2$-error estimate for the symmetric interior penalty method] \label{thm:opt_L2}
  Suppose $\theta = 1$, $\gamma$ sufficiently large, $f \in H^2 (\Omega)$ and
  $u \in H^2 (\Omega)$. 
  %{\color{blue} [GS: quick question, but if we choose $f \in H^2 (\Omega)$, shouldn't $u \in H^4(\Omega)$ by the regularity property of the Laplace operator? We could still use the estimate below based on the $H^2$-seminorm, of course.]}
  The following $L^2$ error estimate holds then
  \begin{eqnarray*}
    \|u - u_h \|_{0, \Omega} \leqslant Ch^2  (|u|_{2, \Omega} + \nobracket \| f \|_{2, \Omega})
     \; .
  \end{eqnarray*}
\end{thm}

\begin{proof}

\noindent We proceed by the Aubin-Nitsche trick. Let $w \in H^2 (\Omega)$ be the solution to
\begin{subequations}
\label{eq:strong_adj}
\begin{align}
- \Delta w 
&=\; 
u - u_h \, , \quad \mbox{in } \Om \; , \\
 w 
&=\; 
0 \, , \ \qquad \quad \mbox{on } \G_D \; , \\ 
\nabla w \cdot \bs{n} 
&=\; 
0 \, , \ \qquad \quad \mbox{on } \G_N \; ,
\end{align}
\end{subequations}
Then
  \begin{eqnarray*}
    \|u - u_h \|_{0, \Omega}^2 = a_h (u - u_h,w) = a_h (w, u - u_h) \; ,
  \end{eqnarray*}
  by the symmetry of $a_h$ for the symmetric variant of the method with $\theta = 1$. 
  Take $w_h = I_h w \in V_h$ as the nodal interpolant of $w$.
  Then, using the ``modified Galerkin orthogonality'' (\ref{GalOrtho}) we rewrite
  \begin{eqnarray*}
    \|u - u_h \|_{0, \Omega}^2 = a_h (w - w_h, u - u_h) 
    + \sum_{\ti{e} \subset \tG} \left(
    (w_h, f)_{T_{\ti{e}}^{\tmop{ext}}} - \left\langle w_h, f \,
    \right\rangle_{\ti{e}} \, H_{\ti{e}}
    \right)
    \; .
  \end{eqnarray*}
  Proceeding as in the proof of the $H^1$ error by introducing the averages
  $\bar{f}$ and $\bar{w}_h$ of $f$ and $w_h$ on every edge $\ti{e}$, we
  continue the above calculations as
  \begin{align*}
    \|u - u_h \|_{0, \Omega}^2 
    &\leqslant\;
     \interleave w - w_h \interleave_h  \interleave u - u_h \interleave_h
    + \sum_{\ti{e} \subset \tG} \left( (w_h - \bar{w}_h, f)_{T_{\ti{e}}^{\tmop{ext}}} +
    (\bar{w}_h, f - \bar{f})_{T_{\ti{e}}^{\tmop{ext}}} \right)
     \nonumber \\
    &\phantom{\leqslant} \;
    - \sum_{\ti{e} \subset \tG} \left( \left\langle w_h - \bar{w}_h, f \, \right\rangle_{\ti{e}} \, H_{\ti{e}} \right)
    \nonumber \\
    &\leqslant\;
    Ch^2  |w|_{2, \Omega} \, (| u |_{2, \Omega} +\|f\|_{1, \Omega})
	+ Ch (|w_h |_{1, \Omega_h^{\Gamma, \tmop{fat}}} \, \| f \|_{0,\Omega_h^{\Gamma, \tmop{fat}}} +
    \| w_h \|_{0,\Omega_h^{\Gamma, \tmop{fat}}} \, |f|_{1, \Omega_h^{\Gamma, \tmop{fat}}} )
    \nonumber \\
    &\phantom{\leqslant} \;
    + Ch^2  |w|_{1, \tG} \|f\|_{0,\tG}
    \; .
  \end{align*}
  Here, we have used the already proven error estimate~\eqref{AprioTrip}, and bounds (\ref{LocEstFatBan}). To gain another $h$ in
  the terms multiplied by the first power of $h$, we use Lemmas \ref{Lemma2} and \ref{Lemma3}$:$
  \begin{align*}
    \|u - u_h \|_{0, \Omega}^2 
    &\leqslant\;
     Ch^2  |w|_{2, \Omega} (| u |_{2, \Omega} +\|f\|_{1, \Omega})
     + Ch^2  ( \nobracket (|w_h |_{1, \Omega } + |w_h |_{2, h,
    \Omega }) \| f \|_{1, \Omega} + |w_h |_{1, \Omega } \| \nabla f \|_{1, \Omega} )
     \nonumber \\
    &\phantom{\leqslant} \;
     + Ch^2  |w_h |_{1, \tG} \|f\|_{0,\tG}
     \; .
  \end{align*}
  By interpolotation estimates, $|w_h |_{1, \tG} + |w_h |_{1,
  \Omega} + |w_h |_{2, h, \Omega} \leqslant C \| w \|_{2, \Omega}$. By the trace
  inequality, $\|f\|_{0,\tG} \leqslant C \|f\|_{1, \Omega}$. Thus,
  the estimate above leads to
  \begin{eqnarray*}
    \|u - u_h \|_{0, \Omega}^2 \leqslant Ch^2  (|u|_{2, \Omega} + \nobracket
    \| f \|_{2, \Omega}) \| w \|_{2, \Omega}
  \end{eqnarray*}
  and we conclude by recalling $\| w \|_{2, \Omega} \leqslant C \|u - u_h
  \|_{0, \Omega}$.
\end{proof}

\begin{rem}
  For other choices of parameters, i.e. $\theta \neq 1$ with $\gamma$
  sufficiently large, or $\theta = - 1, \gamma = 0$, we can prove the
  sub-optimal $L^2$ error estimate
  \begin{eqnarray*}
    \|u - u_h \|_{0, \Omega} \leqslant Ch^{3 / 2}  (|u|_{2, \Omega} +
    \nobracket \| f \|_{2, \Omega}) \; .
  \end{eqnarray*}
  This result should not be surprising, since it also sholds for the unsymmetric Nitsche's method on body-fitted grids.
  To this end, we proceed again by the Aubin-Nitsche trick, starting by
  introducing $w \in H^2 (\Omega)$ as in the proof of Theorem~\ref{thm:opt_L2}. 
  This time, however, the bilinear form is not symmetric, and some adjustments need to be made. Then
  \begin{eqnarray}
  	\label{eq:work_l2subopt}
    \|u - u_h \|_{0, \Omega}^2 
    &=& 
    a_h (u - u_h,w)
    \nonumber \\
    &\phantom{=}& 
    a_h (w, u - u_h)
    - (1 - \theta) \, \langle \jumpb{u - u_h} \,,
    \, \avgb{\nabla w} \rangle_{\mathcal{E}^{\tmop{ext} ; o}} 
    \nonumber \\
    &\phantom{=}& 
    -(1 - \theta) \, \sum_{\ti{e} \subset \tGD} \langle \oS
    (\nabla w) \cdot \bs{n} \,, \, \oS (u - u_h)
    \, j_{\ti{e}} \rangle_{\ti{e}}
    \nonumber \\
    &=& 
    a_h (w - w_h, u - u_h) + \sum_{\ti{e} \subset \tG} (w_h,
    f)_{T_{\ti{e}}^{\tmop{ext}}} - \left\langle w_h, f \,
    \right\rangle_{\ti{e}} \, H_{\ti{e}} + c_{\theta}
    \; ,
  \end{eqnarray}
  where
  $$c_{\theta} = - (1 - \theta) \, \langle \jumpb{u - u_h}
  \,, \, \avgb{\nabla w}
  \rangle_{\mathcal{E}^{\tmop{ext} ; o}} - (1 - \theta) \,
  \sum_{\ti{e} \subset \tGD} \langle \oS (\nabla w) \cdot \bs{n}
  \,, \, \oS (u - u_h) \,
  \, j_{\ti{e}} \rangle_{\ti{e}} \; . $$
  Other than $c_{\theta}$, all the terms in the right-hand side of~\eqref{eq:work_l2subopt},
  can be treated as in the proof of Theorem~\ref{thm:opt_L2}. For $c_{\theta}$,
  we proceed as follows
  \begin{eqnarray*}
    c_{\theta} \leqslant C \sqrt{h} \interleave u - u_h \interleave_h \left(
    \| \nabla w \cdot \bs{n} \|_{0,\tGD} + \left\| \nabla w \cdot
    \bs{n} \right\|_{0,\mathcal{E}^{\tmop{ext} ; o}}  \right) \; .
  \end{eqnarray*}
  To treat the contributions on the edges in $\mathcal{E}^{\tmop{ext};o}$, we apply the following inverse inequality
  \begin{eqnarray*}
    {\| \nabla w \cdot \bs{n} \|_{0,\mathcal{E}^{\tmop{ext} ; o}}} 
    \leqslant C \left( \frac{1}{\sqrt{h}} \left\| \nabla w
    \right\|_{0,\Omega_h^{\Gamma, \tmop{fat}}}  + \sqrt{h} \left| \nabla w
    \right|_{1, \Omega_h^{\Gamma, \tmop{fat}}}  \right) \leqslant C {\|
    \nabla w \|_{1, \Omega }} 
    \; ,
  \end{eqnarray*}
  derived applying Lemma~\ref{Lemma2}, and thinking about the edges in $\mathcal{E}^{\tmop{ext};o}$ as part of quadrilaterals in $\ti{\mathcal{T}}_h^{\tmop{ext}}$.
  Collecting all these contributions and using the estimate for \
  $\interleave u - u_h \interleave_h$ gives
  \begin{eqnarray*}
    \|u - u_h \|_{0, \Omega}^2 \leqslant Ch^2  (|u|_{2, \Omega} + \nobracket
    \| f \|_{2, \Omega}) \| w \|_{2, \Omega} + Ch^{3 / 2}  (|u|_{2, \Omega} +
    \nobracket \| f \|_{1, \Omega}) \left( \| \nabla w \cdot \bs{n}
    \|_{0,\tGD} + \left\| \nabla w \right\|_{1, \Omega}  \right)
  \end{eqnarray*}
  and we conclude by recalling $\| w \|_{2, \Omega} \leqslant C \|u - u_h
  \|_{0, \Omega}$.
\end{rem}

%\subsection{Other theoretical results}
%
%{\color{red} [GS: we need to say that this SBM has the partition of unity \ldots the fact that we pass the patch test is a consequence of this \ldots]}
%
%\subsection{On conditioning numbers}
%
%\newpage

\section{Neumann Boundary Conditions for Linear Elasticity \label{sec:sbm_linela}}

In the numerical tests that follow, we also consider the equations of (compressible) isotropic linear elasticity. Their strong form is given as
\begin{subequations}
	\label{eq:SteadyLinEla}
\begin{align}
- \nabla \cdot \left( \bs{\sigma(\bs{u})} \right) &=\; \bs{b}  \qquad \text{\ \ in \ } \Om \; ,
 \\
\bs{u} &=\;  \bs{u}_D \qquad \! \text{on \ }  \GD \; ,
 \\
\bs{\sigma} \bs{n} &=\;  \bs{t}_N \qquad \text{on \ }  \GN \; ,
\end{align}
\end{subequations}
where $\bs{u}$ is the displacement field, $\bs{u}_D$ its value on the Dirichlet boundary $\GD$, $\bs{t}_N$ the normal traction along the Neumann boundary $\GN$, and $\bs{b}$ a body force.
We of course assume that $\partial \Om = \overline{\GD \cup \GN}$ and $\GD \cap \GN = \emptyset$.
The stress $\bs{\sigma}$ is a linear function of $\bs{u}$, according to the constitutive model
$$\bs{\sigma}(\bs{u}) = 2 \mu \, \bs{\varepsilon}(\bs{u}) + \lambda (\nabla \cdot \bs{u}) \bs{I}\; . $$
The proposed SBM variational form of~\eqref{eq:SteadyLinEla} can be derived in a similar way to the case of the Poisson equation. Using the identities: 
\begin{align}
\jumpb{\bs{\omega}_h} &=\; \bs{\omega}_h^{+} \bs{n}^{+}+\bs{\omega}_h^{-} \bs{n}^{-} \; , \\
\jumpb{\bs{v}_h}_{\otimes} &=\; \bs{v}_h^{+} \otimes \bs{n}^{+}+\bs{v}_h^{-} \otimes \bs{n}^{-} \; , \\
\jumpb{ \bs{\omega}_h \, \bs{v}_h} &=\; \avgb{\bs{\omega}_h} : \jumpb{\bs{v}_h}_{\otimes} + \jumpb{\bs{\omega}_h}\cdot\avgb{\bs{v}_h} \; ,
\end{align}
where $\bs{\omega}_h$ is a second-order tensor function and $\bs{v}_h$ is a vector function. Hence we have:
\begin{multline}
( \bs{\varepsilon}(\bs{w}_h)  \, , \,  2\mu \, \bs{\varepsilon}(\bs{u}_h) )_{\tO}
\, + \, 
( \bs{\varepsilon}(\bs{w}_h)  \, , \, 2\mu \,  \bs{\varepsilon}(\bs{u}_h) )_{\ti{\mathcal{T}}_h^{\tmop{ext}}} 
\, + \, 
( \nabla \cdot \bs{w}_h  \, , \,  \lambda \, \nabla \cdot \bs{u}_h )_{\tO}
\, + \, 
( \nabla \cdot \bs{w}_h  \, , \,  \lambda \, \nabla \cdot \bs{u}_h )_{\ti{\mathcal{T}}_h^{\tmop{ext}}} 
\\
\, - \,
\avg{ \jumpb{\bs{w}_h}_{\otimes} \, , \, \avgb{ \lambda \, \nabla \cdot \bs{u}_h  \bs{I} + 2\mu \,\bs{\varepsilon}(\bs{u}_h) } }_{\mathcal{E}^{\tmop{ext};o} }
\, - \theta \,
\avg{ \avgb{ \lambda \, \nabla \cdot \bs{w}_h  \bs{I} + 2\mu \, \bs{\varepsilon}(\bs{w}_h)} \, , \, \jumpb{\bs{u}_h}_{\otimes} }_{\mathcal{E}^{\tmop{ext};o} }
\\
\; - \;
\avg{ \bs{w}_h \, , \,  (\lambda \, \nabla \cdot \bs{u}_h  \bs{I} + 2\mu \,\bs{\varepsilon}(\bs{u}_h)) \bs{n}  }_{\GD}
\; - \theta \;
\avg{ (\lambda \, \nabla \cdot \bs{w}_h  \bs{I} + 2\mu \,\bs{\varepsilon}(\bs{w}_h)) \bs{n} \, , \,  \bs{u}-\bs{u}_D  }_{\GD}
\\
\; + \;
\langle \gamma \, h^{-1} \jumpb{\bs{w}_h} \,, \, \jumpb{\bs{u}_h} \rangle_{\mathcal{E}^{\tmop{ext};o}}
\; + \;
\langle \gamma \, h^{-1} \bs{w}_h \,, \, \bs{u}_h-\bs{u}_D \rangle_{\GD}
\\
\; = \;
( \bs{w}_h \, , \, \bs{b} )_{\tO}
\, + \, 
( \bs{w}_h \, , \, \bs{b} )_{\ti{\mathcal{T}}_h^{\tmop{ext}}} 
\; + \;
\avg{ \bs{w}_h \, , \, \bs{t}_N  }_{\GN}
\; .
\label{eq:SBM_var_gen2}
\end{multline}
As in the case of the Poisson equation, the variational form ~\ref{eq:SBM_var_gen2} can be approximated as 
\begin{multline}
( \bs{\varepsilon}(\bs{w}_h)  \, , \, 2\mu \,  \bs{\varepsilon}(\bs{u}_h) )_{\tO}
+
( \nabla \cdot \bs{w}_h  \, , \, \lambda \,  \nabla \cdot \bs{u}_h )_{\tO}
%\\
\, + \, 
\sum_{\ti{e} \subset \tG}
\bigg(
\avg{ \bs{\varepsilon}(\bs{w}_h)  \, , \, 2\mu \,  \bs{\varepsilon}(\bs{u}_h) \, H_{\ti{e}}  }_{\ti{e}} 
\, + \, 
\avg{ \nabla \cdot \bs{w}_h  \, , \, \lambda \,  \nabla \cdot \bs{u}_h \, H_{\ti{e}}  }_{\ti{e}} 
\bigg)
\\
\; - \; 
\sum_{\ti{\bs{a}} \in \mathcal{N}(\tG)}  
\frac{|\bs{d}_{\ti{\bs{a}}}|}{2} \, 
\bigg(
\jumpb{ \nabla \bs{w}_h \bs{d}_{\ti{\bs{a}}} }_{\otimes; \ti{\bs{a}}}
: 
\avgb{ \lambda \nabla \cdot \bs{u}_h \bs{I} + 2\mu \, \bs{\varepsilon}(\bs{u}_h) }_{\ti{\bs{a}}} 
\; + \theta \; 
\avgb{ \lambda \nabla \cdot \bs{w}_h \bs{I} + 2\mu \, \bs{\varepsilon}(\bs{w}_h) }_{\ti{\bs{a}}} 
: 
\jumpb{ \nabla \bs{u}_h \bs{d}_{\ti{\bs{a}}} }_{\otimes;\ti{\bs{a}}}
\bigg)
\\ 
\; - \;
\sum_{\ti{e} \subset \tGD}
\bigg(
\avg{ \oS(\bs{w}_h) \, , \,  ( (\lambda \nabla \cdot \bs{u}_h \bs{I} + 2\mu \, \bs{\varepsilon}(\bs{w}_h)) \bs{n} ) \, j_{\ti{e}} }_{\ti{e}}
\, + \theta \, 
\avg{ (\lambda \nabla \cdot \bs{u}_h \bs{I} + 2\mu \, \bs{\varepsilon}(\bs{w}_h)) \bs{n} \, , \, (\oS(\bs{u}_h) - \bs{u}_D) \, j_{\ti{e}}  }_{\ti{e}} 
\bigg)
\\ 
\; + \;
\sum_{\ti{\bs{a}} \in \mathcal{N}(\tG)}  
\frac{\gamma}{4 h} \, |\bs{d}_{\ti{\bs{a}}}| \, 
\jumpb{ \nabla \bs{w}_h \bs{d}_{\ti{\bs{a}}} }_{\ti{\bs{a}}}
\cdot 
\jumpb{ \nabla \bs{u}_h \bs{d}_{\ti{\bs{a}}} }_{\ti{\bs{a}}}
+\sum_{\ti{e} \subset \tGD}
\frac{\gamma}{h} \, 
\avg{  \oS(\bs{w}_h) \, , \, (\oS(\bs{u}_h) - \bs{u}_D) \, j_{\ti{e}}  }_{\ti{e}} 
\\ 
  = 
( \bs{w}_h \, , \, \bs{b} )_{\tO}
\, + \, 
\sum_{\ti{e} \subset \tG}
\avg{ \bs{w}_h \, , \, \bs{b} \, H_{\ti{e}}  }_{\ti{e}} 
\, + \, 
\sum_{\ti{e} \subset \tGN}
\avg{ \oS(\bs{w}_h) \, , \, \bs{h}_N(M_h(\ti{\bs{x}})) \, j_{\ti{e}}  }_{\ti{e}} 
\; .
\end{multline}

\begin{figure}[tbh!]
  \centering
    \subcaptionbox{The domain $\Om$ (grey).}
    [.4\textwidth]{\begin{subfigure}[tbh!]{.3\textwidth}
    \centering
    \begin{tikzpicture}
      \filldraw[color=white, fill=white!10, thin](0.0,0.0) rectangle (5,0.5);
      \filldraw[color=black, fill=black!10, thin](0.0,0.5) rectangle (5,5.5);
      \filldraw[color=black, fill=white, thick](3,3) circle (1.25);
      %\filldraw [black] (0,0) square (1pt);
      \draw (3,3) -- (4.25,3);
      \node[text width=0.1cm] at (3,3-0.2){$r$};
    \end{tikzpicture}
      \end{subfigure}
    }
    \subcaptionbox{$\tO$ (grey), $\G$ (blue) and $\tG$ (red).}
    [.4\textwidth]{\begin{subfigure}[tbh!]{0.39\textwidth}\centering
    \includegraphics[width=\linewidth]{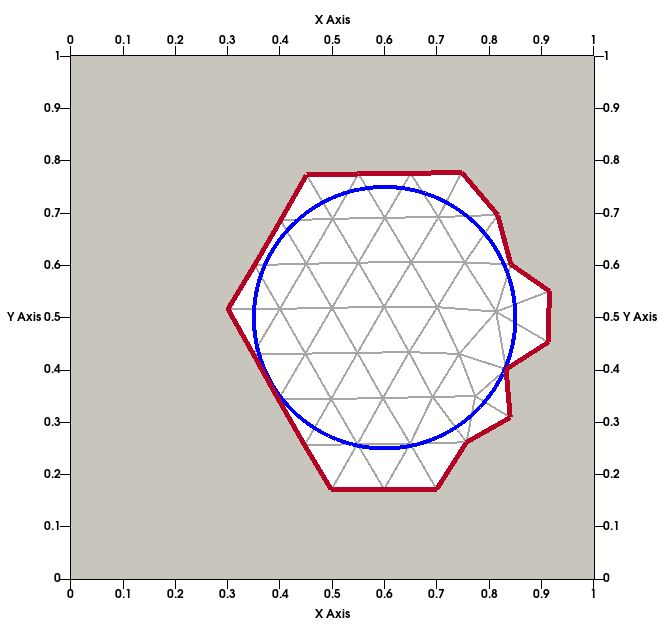}
  \end{subfigure}}
\caption{General problem schematic for the active domain (grey), true boundary (blue), surrogate boundary (red)}
    \label{fig:square}
\end{figure}

\section{Numerical results \label{sec:numerical_results}}

We present the results from a series of two-dimensional numerical experiments that demonstrate the theoretical findings of Section~\ref{sec:theo}. 
Our general approach relies on the method of manufactured solutions and involves embedding geometries (both analytic and polygonal) on a series of meshes of increased refinement.
A schematic representation of the experimental setup is provided in Figure \ref{fig:square}, depicting an embedded circular shape on a unit square domain. 

More specifically, since embedded geometries may arbitrarily intersect the grids, we also apply specific rotations to the latter (with respect to the embedded shapes) to test the effect on the numerical results of such perturbations.
In fact, the newly conceptualized method requires computation of additional geometric quantities on the surrogate boundary (with respect to a traditional SBM).
Grid rotations provide a means to examine a high number of extension arrangements, i.e. robustness.
The convergence of the $L^2$-norm and the $H^1$-seminorm of the error are assessed, along with the condition number, for both the Poisson and linear elasticity equations.

The primary motivator for the proposed method was, by and large, the development of Neumann boundary conditions that guarantee optimal convergence rates.
Although a majority of the experiments involve only embedded Neumann boundaries, the array of tests was expanded to also include Dirichlet boundary conditions. 
The numerical experiments on the Poisson equation encompass both the symmetric and anti-symmetric Nitsche formulations for the weak enforcement of Dirichlet conditions, along with the inclusion of a ``patch test.'' Likewise, a simple bending beam test was performed for linear elasticity, which included both homogeneous displacement and homogeneous traction boundary conditions. Furthermore, optimal convergence was also achieved with quadrilateral elements, demonstrating the flexibility of the method beyond standard triangular finite elements.

\subsection{Patch test for the Poisson problem}
A patch test experiment was performed to assess the ability of the method to match an affine exact solution. In general, passing a patch test does not guarantee convergence nor stability of a numerical method.
Yet, it is an important sanity check for the proposed conceptualization of the SBM, which involves solution extensions and approximate integration over the gap region.
Affine solutions are relevant in engineering applications, since they imply a constant flux scenario in the case of the Poisson problem or a constant strain scenario in the case of linear elasticity.

The proposed SBM variant possesses the partition of unity property and passes the patch test, since affine solutions can be exactly represented in the gap region.
We considered the Poisson problem with three simple geometries (circle, rotated square, and star) embedded on a unit square domain with a regular background mesh, as seen in Figure \ref{fig:patch1}. The solutions are $u = x+y$, $u = x$, and $u = y$ for the circle, rotated square, and star respectively. A Dirichlet boundary condition is strongly enforced on the outer perimeter of the unit square, while the shifted Neumann boundary condition is applied to the surrogate boundary in red. A visualization of the surrogate extensions are provided in Figure \ref{fig:patch2}, along with contours of the nodal error between the approximate and exact solutions $u_h - u_{exact}$. It is easily seen that the numerical error is within machine precision.

\begin{figure}[tbh!]
  \centering
  \begin{subfigure}[t]{0.31\textwidth}
    \includegraphics[width=\linewidth]{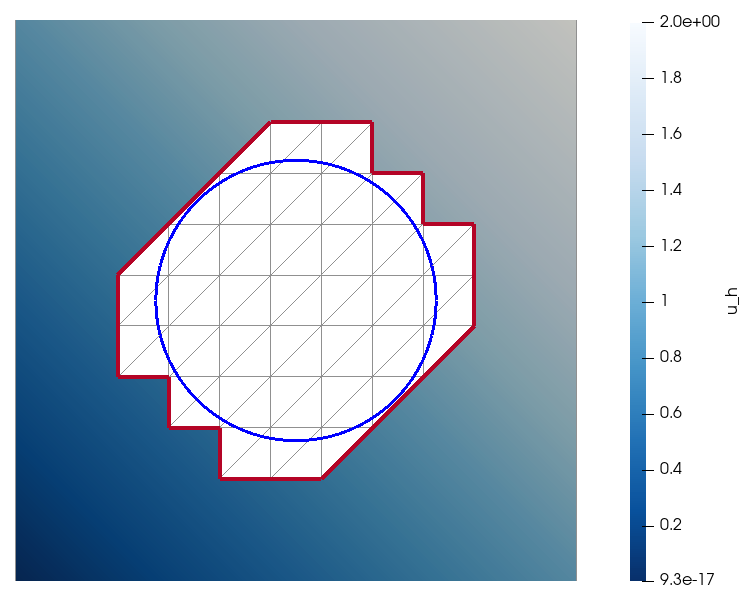}%\vspace{3mm}
    \caption{$u = x+y$, }
%    \label{fig:mesh}
  \end{subfigure}
\begin{subfigure}[t]{0.31\textwidth}\centering
    \includegraphics[width=\linewidth]{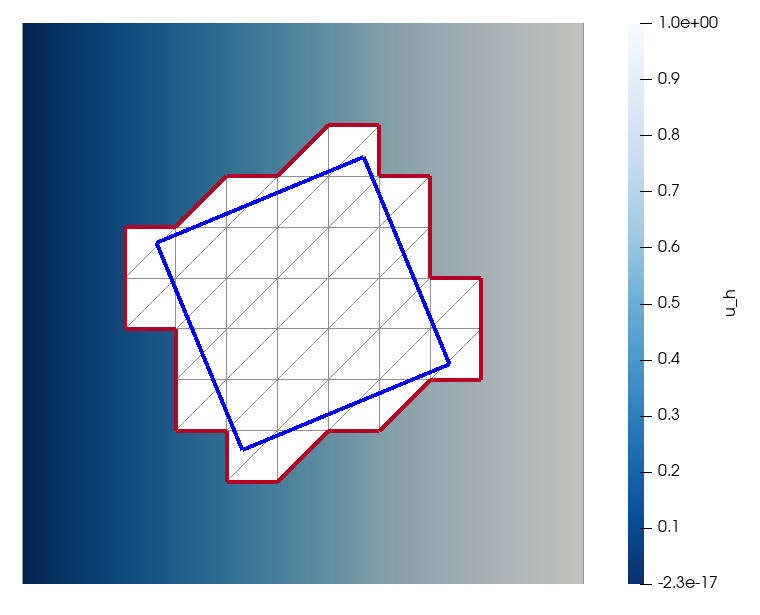}%\vspace{3mm}
    \caption{$u = x$}
%    \label{fig:mesh}
  \end{subfigure}
  \begin{subfigure}[t]{0.31\textwidth}\centering
    \includegraphics[width=\linewidth]{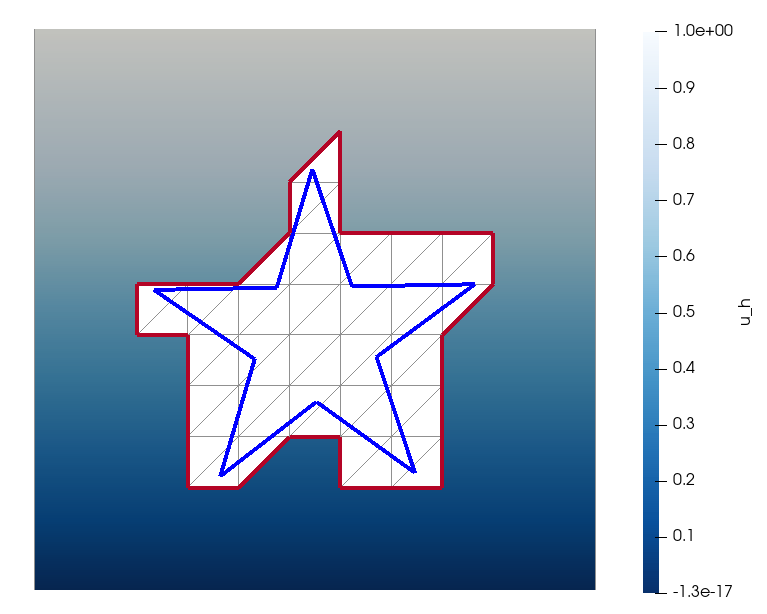}%\vspace{3mm}
    \caption{$u = y$}
%    \label{fig:mesh}
  \end{subfigure}
    \caption{Patch test solutions $u_h$ on surrogate domains $\tO$. $\G$ (blue) and $\tG$ (red) }
        \label{fig:patch1}
\end{figure}

\begin{figure}[tbh!]
  \centering
  \begin{subfigure}[t]{0.32\textwidth}\centering
    \includegraphics[width=\linewidth]{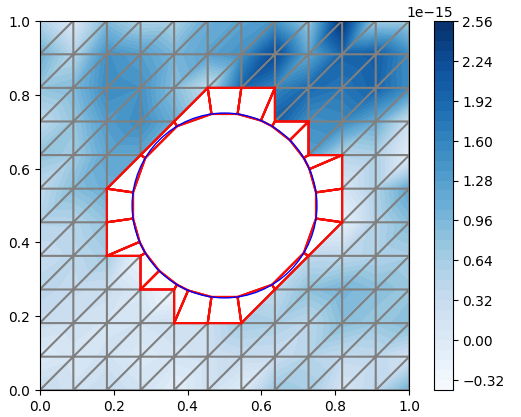}%\vspace{3mm}
    \caption{$u_h - u_{exact}$}
%    \label{fig:mesh}
  \end{subfigure}
  \begin{subfigure}[t]{0.3\textwidth}\centering
    \includegraphics[width=\linewidth]{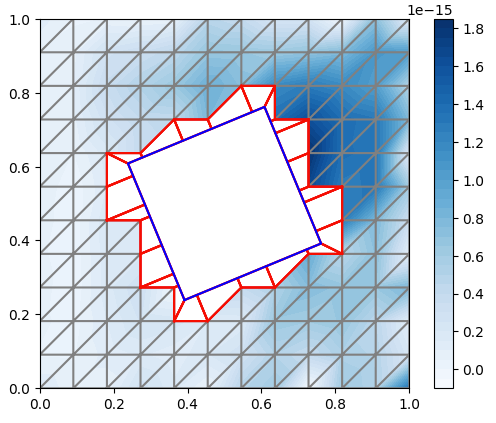}%\vspace{3mm}
        \caption{$u_h - u_{exact}$}
%    \label{fig:mesh}
  \end{subfigure}
  \begin{subfigure}[t]{0.3\textwidth}\centering
    \includegraphics[width=\linewidth]{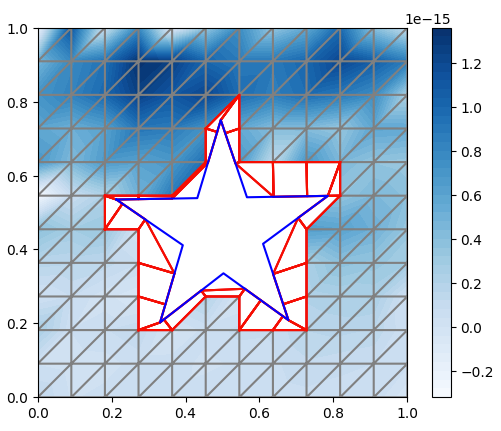}%\vspace{3mm}
        \caption{$u_h - u_{exact}$}
%    \label{fig:mesh}
  \end{subfigure}
    \caption{Patch test results with $u_h - u_{exact}$ contours. $\G$ (blue) and $T^{ext}$ extensions (red)}
            \label{fig:patch2}
\end{figure}

\subsection{Poisson problem with manufactured solutions}
The aim of these numerical experiments is to analyze convergence properties of the method for a high number of extension arrangements. Rotating the background mesh, while fixing the embedded geometry and manufactured solution, isolates effects imparted on the quality of the solution by the geometry of the element extensions $T^{ext}_{\ti{e}}$. Both a smooth, analytical shape (circle) and a concave, polygonal shape (star) were immersed into triangular background meshes. 
Nine increments of grid rotation (from zero to 45 degrees) and seven levels of grid refinement were applied.
Results from a boundary-fitted, primal formulation are also included for comparison.

The geometric setup of the first test involves a circular boundary of radius 0.25 and centered at [0.6,0.5].
 The computational grids are unstructured triangular meshes rotated around [0.5,0.5]: this offset ensures variability in the cuts for each rotation. 
Neumann boundary conditions are applied along the surrogate boundary associated with the circular shape.
The second test involves a star-shaped polygon (five-point star) centered at [0.5,0.5] and immersed into a structured background triangular mesh.
Neumann boundary conditions are applied on the surrogate boundary along the star shape for $x>0.5$ and Dirichlet conditions for $x \leq 0.5$.
As in the patch test, Dirichlet conditions are enforced strongly on the outer perimeter for both tests. 
The analytical solution and corresponding forcing function are
\begin{align}
  u(x,y) &=\; \sin(4 \pi x) \sin(4 \pi y) \; ,
\\
  f(x,y) &=\; 32\pi^2\sin(4 \pi x) \sin(4 \pi y) \; ,
\end{align}
which were deduced from the strong form of the Poisson equation, using the method of manufactured solutions. 
Boundary conditions are specified accordingly.

For visualization purposes, a sampling of computed solutions from various rotations are included in Figures~\ref{fig:visual} and~\ref{fig:visual2}.
The results displayed in Figures~\ref{fig:H0results} and ~\ref{fig:SPresults} show that the convergence rates of the $L^2$-norm and the $H^1$-seminorm are optimal.
In terms of the condition number $\kappa(A)$ associated with the algebraic problem, we see that the proposed method maintains the expected scaling of $\kappa(A) \sim h^{-2}$.
In Figure \ref{fig:SPresults}, both the symmetric ($\theta = 1$, $\gamma = 10$) and anti-symmetric ($\theta = -1$, $\gamma = 0$) Nitsche formulations are simulated and compared to the primal, boundary-fitted case.
The symmetric Nitsche formulation seems more accurate in the $L^2$-norm of the error, but at the expense of higher condition numbers.

\begin{figure}[!htb]
	\centering
\begin{tabular}{cccc}
	  \subfloat[$u^h$ for rotation of $0^o$]{\includegraphics[width = 0.25\textwidth]{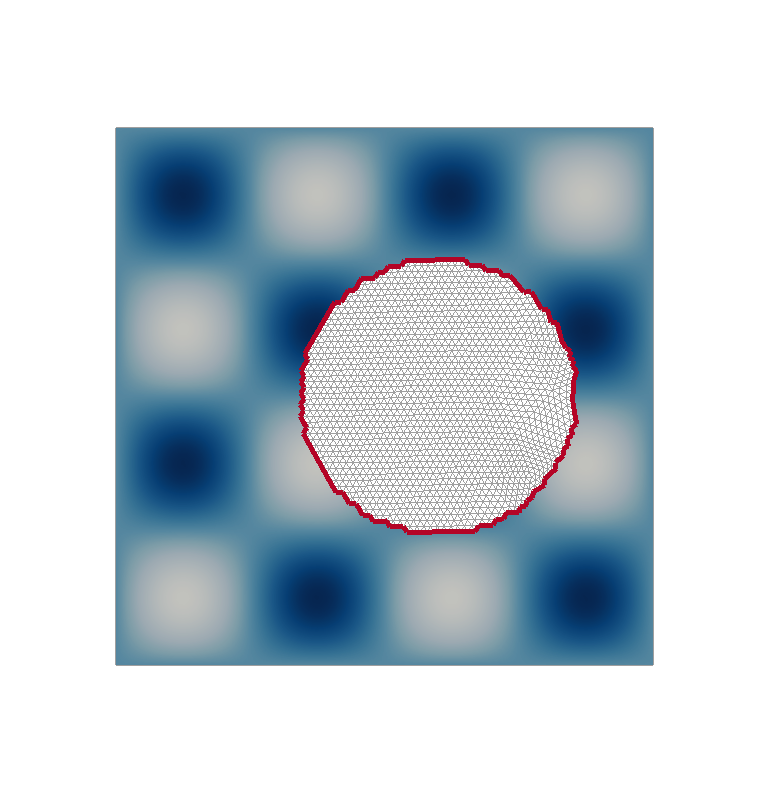}}\hspace{-0.8cm}
          &\subfloat[$u^h$ for rotation of $15^o$]{\includegraphics[width = 0.25\textwidth]{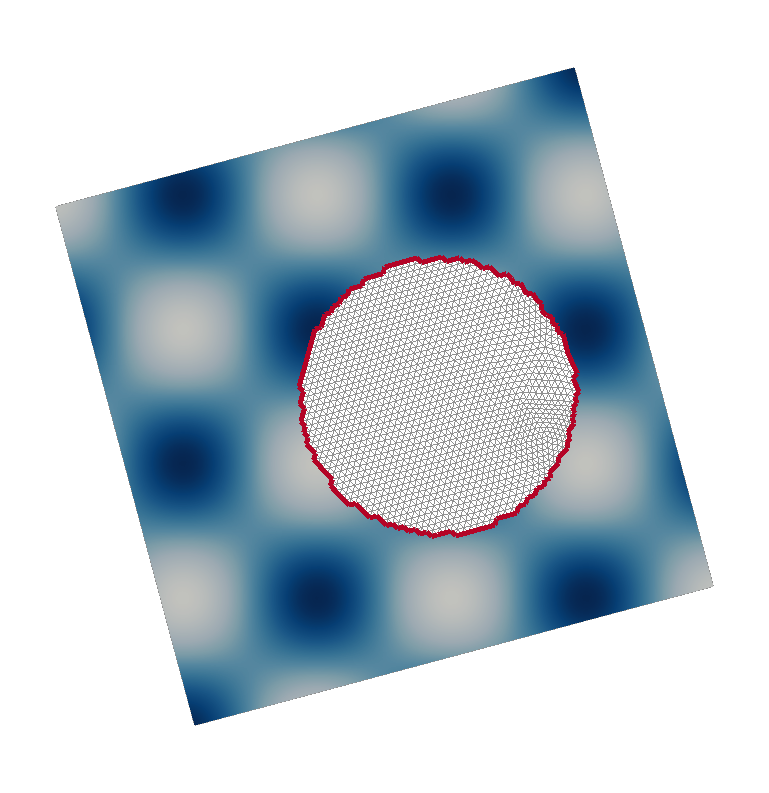}}\hspace{-0.5cm}
          &\subfloat[$u^h$ for rotation of $30^o$]{\includegraphics[width = 0.25\textwidth]{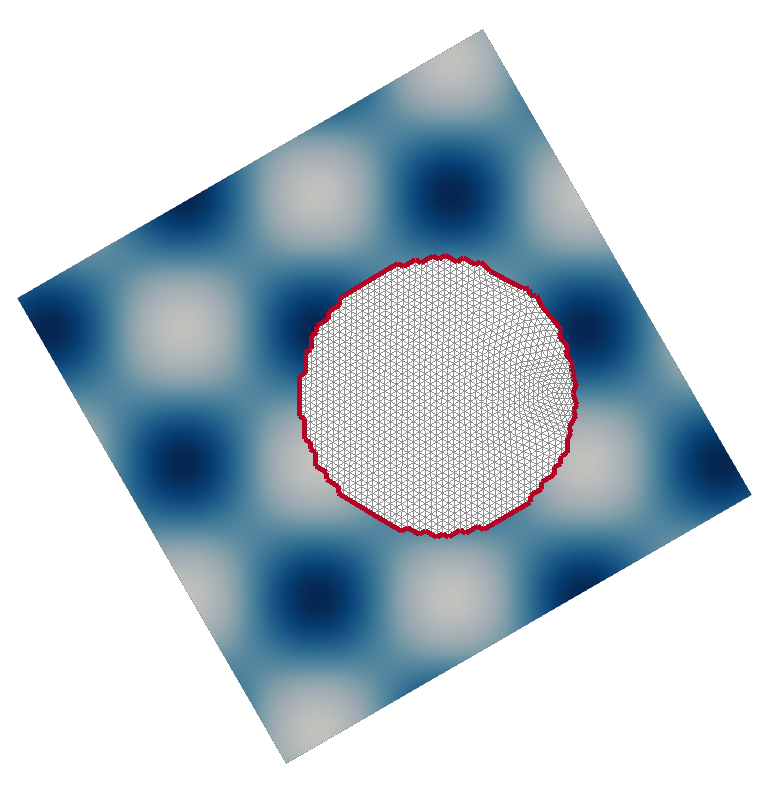}}\hspace{-0.2cm}
          &\subfloat[$u^h$ for rotation of $45^o$]{\includegraphics[width = 0.25\textwidth]{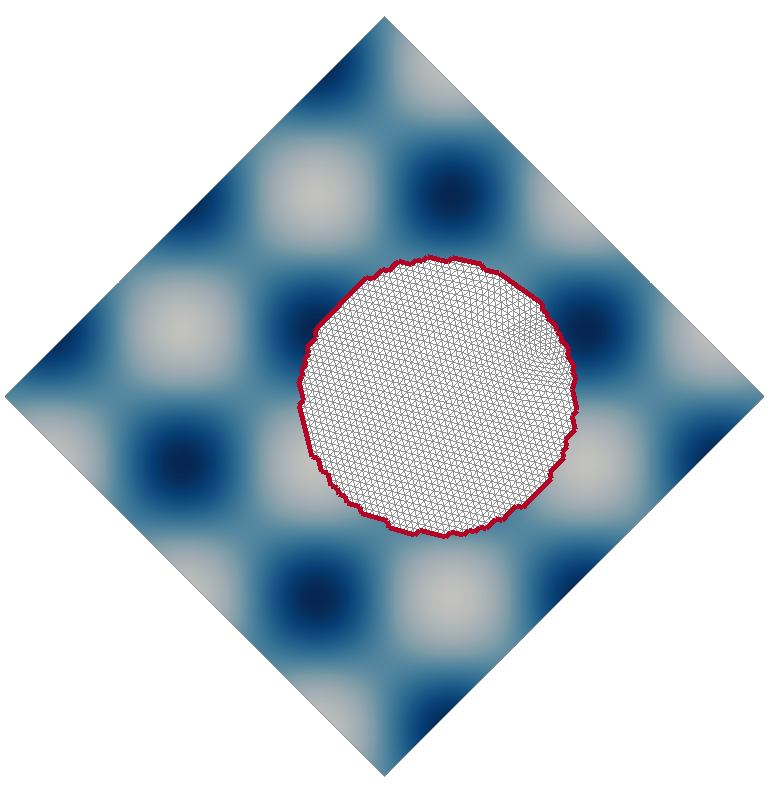}}
		
	\end{tabular}
        \caption{Visualizations of $u^h$ for various degrees of rotation, superimposed on background grid.}
        \label{fig:visual}

\end{figure}

\begin{figure}[tbh!]
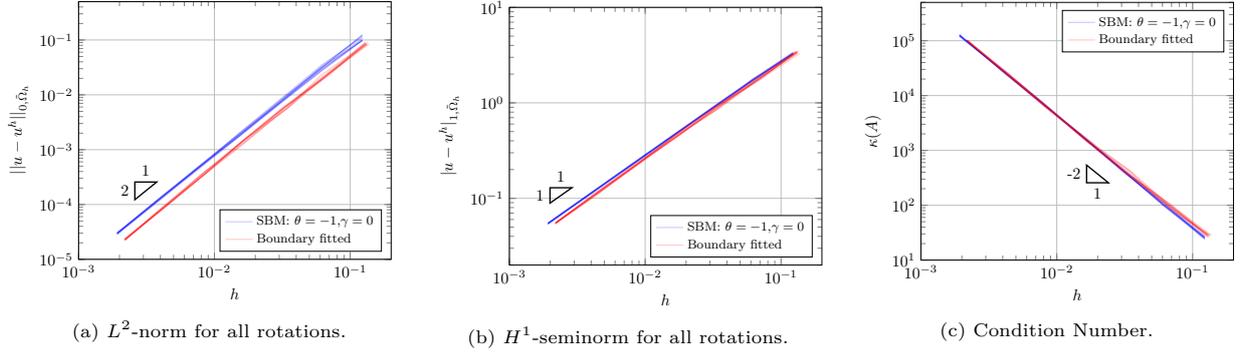

	\begin{subfigure}[tbh!]{0.33\textwidth}
		%
                \caption{$L^2$-norm for all rotations.}
	\end{subfigure}
        \begin{subfigure}[tbh!]{0.33\textwidth}\centering
		%
                \caption{$H^1$-seminorm for all rotations.}
	\end{subfigure}
        \begin{subfigure}[tbh!]{0.33\textwidth}\centering
		%
                \caption{Condition Number.}
	\end{subfigure}
        \caption{Convergence rates and condition numbers for the Poisson problem with a circular boundary.}
        \label{fig:H0results}
\end{figure}

% \begin{figure}[!htb]
% 	\centering
% 	\begin{tabular}{cccc}
% 	  \subfloat[$u^h$ for rotation of $0^o$]{\includegraphics[width = 0.25\textwidth]{Figures/star0.png}}\hspace{-0.8cm}
%           &\subfloat[$u^h$ for rotation of $15^o$]{\includegraphics[width = 0.25\textwidth]{Figures/star15.png}}\hspace{-0.5cm}
%           &\subfloat[$u^h$ for rotation of $30^o$]{\includegraphics[width = 0.25\textwidth]{Figures/star30.png}}\hspace{-0.2cm}
%           &\subfloat[$u^h$ for rotation of $45^o$]{\includegraphics[width = 0.25\textwidth]{Figures/star45.png}}
		
% 	\end{tabular}
%         \caption{Visualizations of $u^h$ for various degrees of rotation, superimposed on background grid.}
%         \label{fig:visual2}

% \end{figure}

\begin{figure}[!htb]
	\centering
	\begin{tabular}{cccc}
	  \subfloat[$u^h$ for rotation of $0^o$]{\includegraphics[width = 0.25\textwidth]{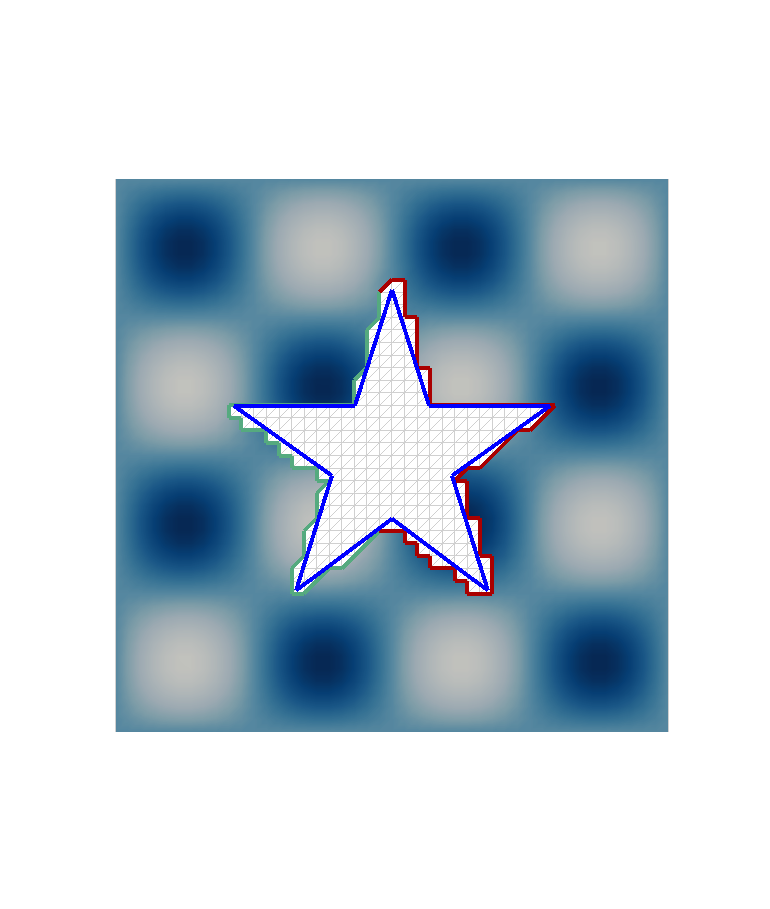}}\hspace{-0.8cm}
          &\subfloat[$u^h$ for rotation of $15^o$]{\includegraphics[width = 0.25\textwidth]{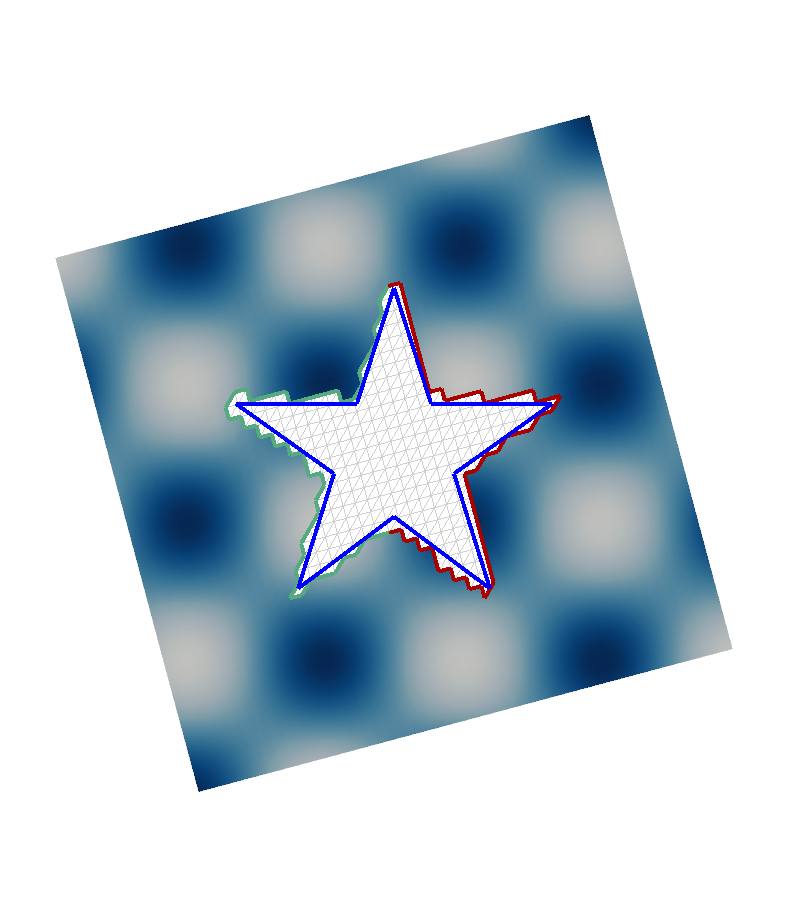}}\hspace{-0.6cm}
          &\subfloat[$u^h$ for rotation of $30^o$]{\includegraphics[width = 0.25\textwidth]{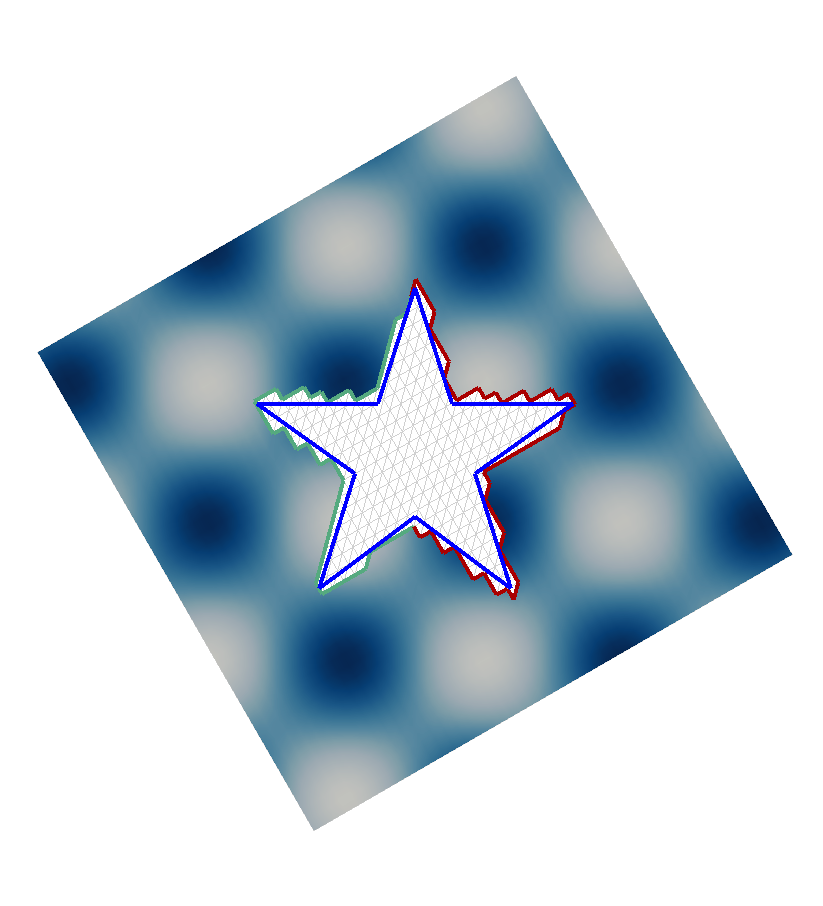}}\hspace{-0.2cm}
          &\subfloat[$u^h$ for rotation of $45^o$]{\includegraphics[width = 0.25\textwidth]{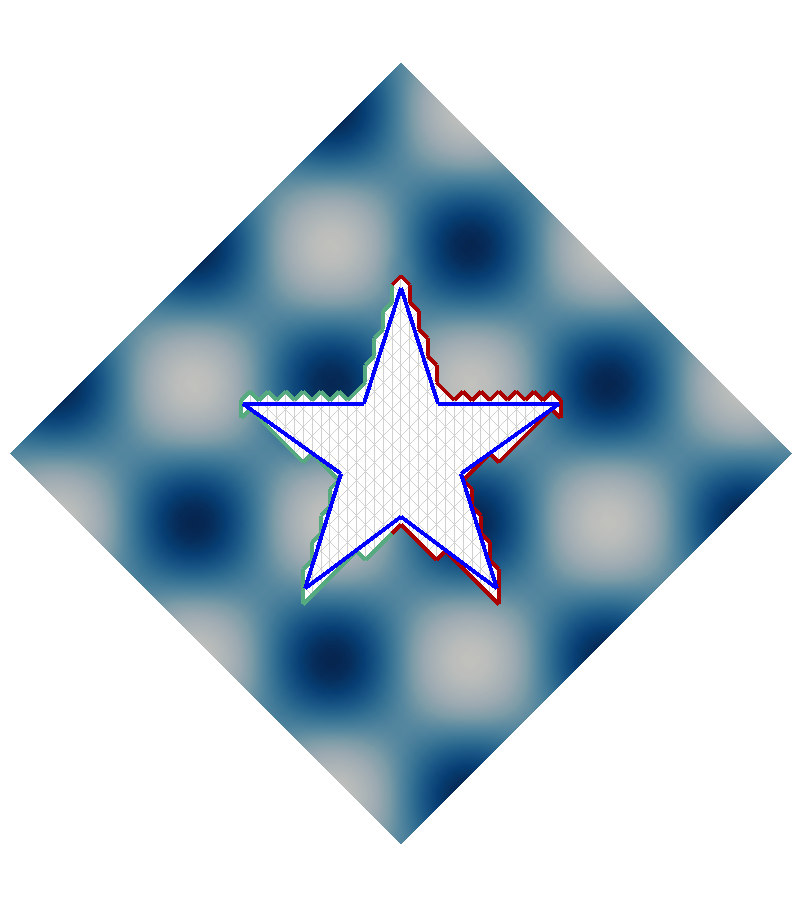}}
		
	\end{tabular}
        \caption{Visualizations of $u^h$ for various degrees of rotation, superimposed on background grid. $\tilde{\Gamma}_D$ (Green) and $\tilde{\Gamma}_N$  (Red)}
        \label{fig:visual2}

\end{figure}

 \begin{figure}[tbh!]
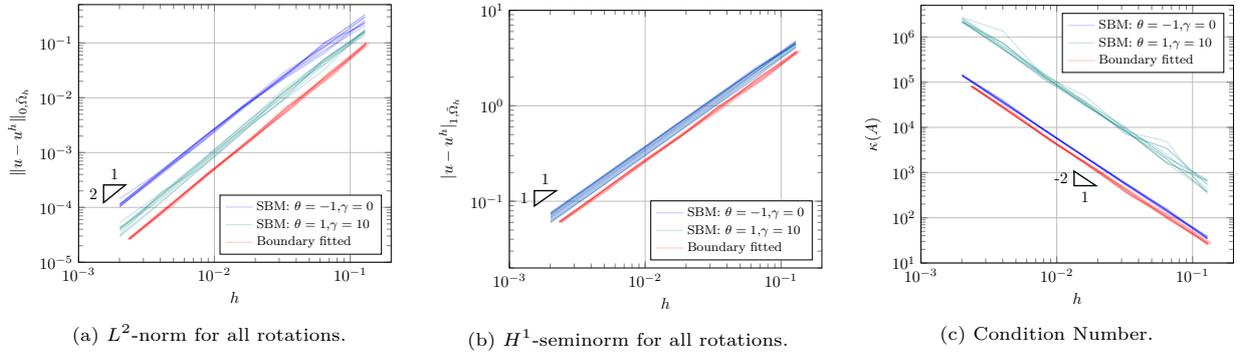

	\begin{subfigure}[tbh!]{0.33\textwidth}
		%
                \caption{$L^2$-norm for all rotations.}
	\end{subfigure}
        \begin{subfigure}[tbh!]{0.33\textwidth}\centering
		%
                \caption{$H^1$-seminorm for all rotations.}
	\end{subfigure}
        \begin{subfigure}[tbh!]{0.33\textwidth}\centering
		%
                \caption{Condition Number.}
	\end{subfigure}

        \caption{Convergence rates and condition numbers for the Poisson problem with a star-shaped boundary.}
        \label{fig:SPresults}
\end{figure}

Similar convergence tests with the same manufactured solution were performed using quadrilateral finite elements. 
For boundaries, a square with a side length of $l=0.48$ and centered at $[0.5,0.5]$ and a concave flower-like geometry are considered. The coordinates of the flower-like boundary are parametrized as functions of the angle $-\pi \leq \theta < \pi$:
\begin{equation}
\left\{ \begin{aligned} 
x(\theta) = 0.5 + (0.05+0.24\sin(7\theta)\cos(\theta)) \; , \\
y(\theta) = 0.5 + (0.05+0.24\sin(7\theta)\sin(\theta)) \; .
\end{aligned} \right.
\end{equation}
In both cases, the background grids are fixed in place while the immersed geometries are rotated by 0, 10, 20, 30 and 40 degrees. However, we did not perform rotations of the immersed geometries for quadrilateral body-fitted grids, since the results with the previous triangular body-fitted grids were tightly clustered.

Figure \ref{fig:visualsq} depicts the true and surrogate boundaries and numerical solutions on the active domains. In both cases, embedded Neumann boundary conditions are applied at the inner boundaries and strong Dirichlet boundary conditions are enforced at the outer boundaries. Both the symmetric and anti-symmetric Nitsche formulations are also considered herein. Figures \ref{fig:SqPoissonResults} and \ref{fig:FlPoissonResults} show the convergence rates of the $L^2$- and $H^1$-seminorms, as well as the condition numbers $\kappa(A)$.
It is evident that the results are very similar to the ones obtained with triangular grids.

\begin{figure}[!htb]
	\centering
\begin{tabular}{cc}
	  \subfloat[$\Gamma$ and $\tilde{\Gamma}$ for rotation of $40^o$]{\includegraphics[width = 0.22\textwidth]{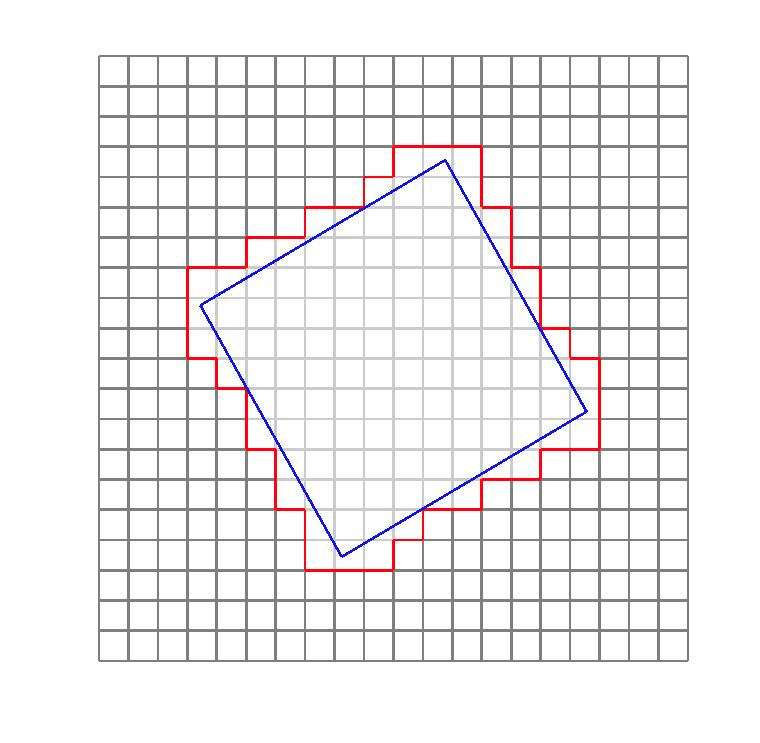}}
	  \subfloat[$u^h$ for rotation of $40^o$]{\includegraphics[width = 0.225\textwidth]{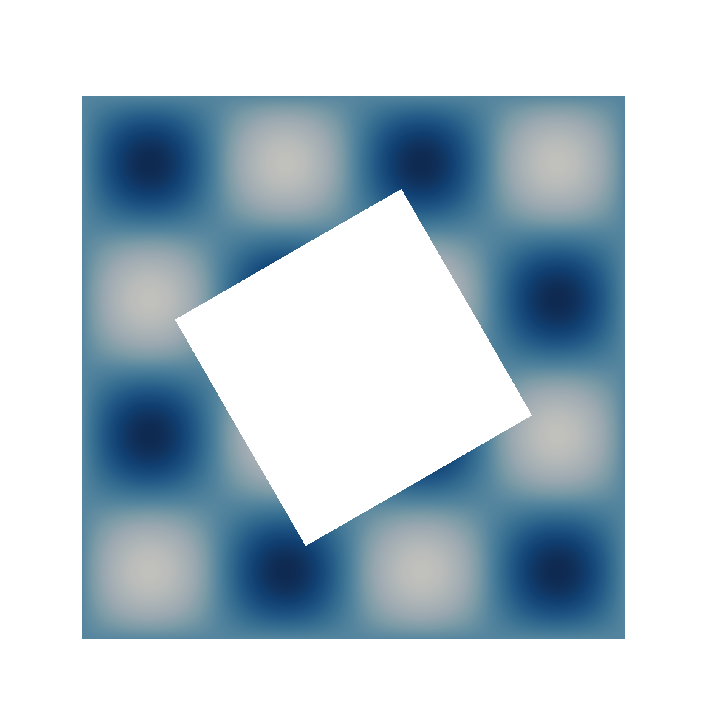}}
	  \subfloat[$\Gamma$ and $\tilde{\Gamma}$ for rotation of $0^o$]{\includegraphics[width = 0.22\textwidth]{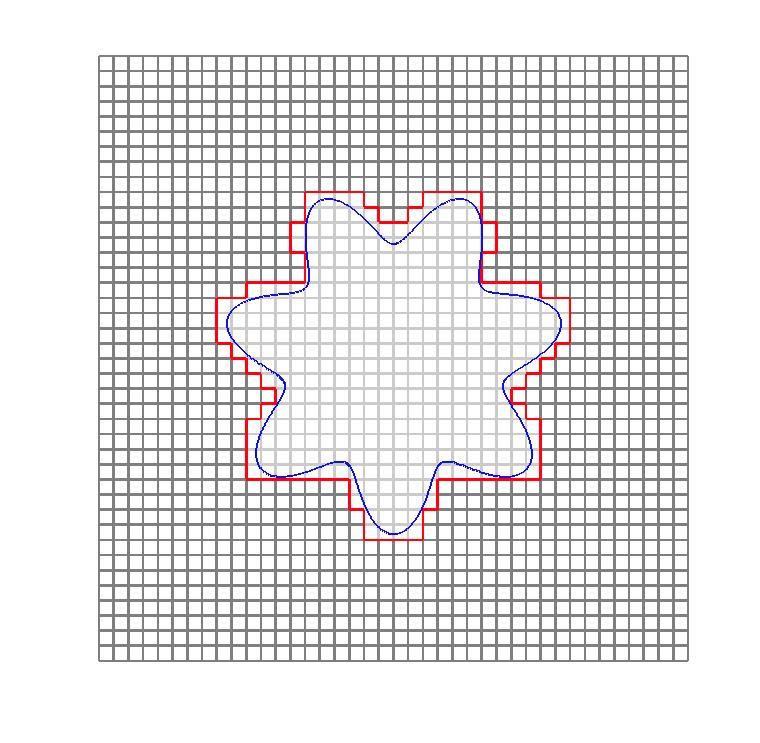}}
      \subfloat[$u^h$ for rotation of $0^o$]{\includegraphics[width = 0.21\textwidth]{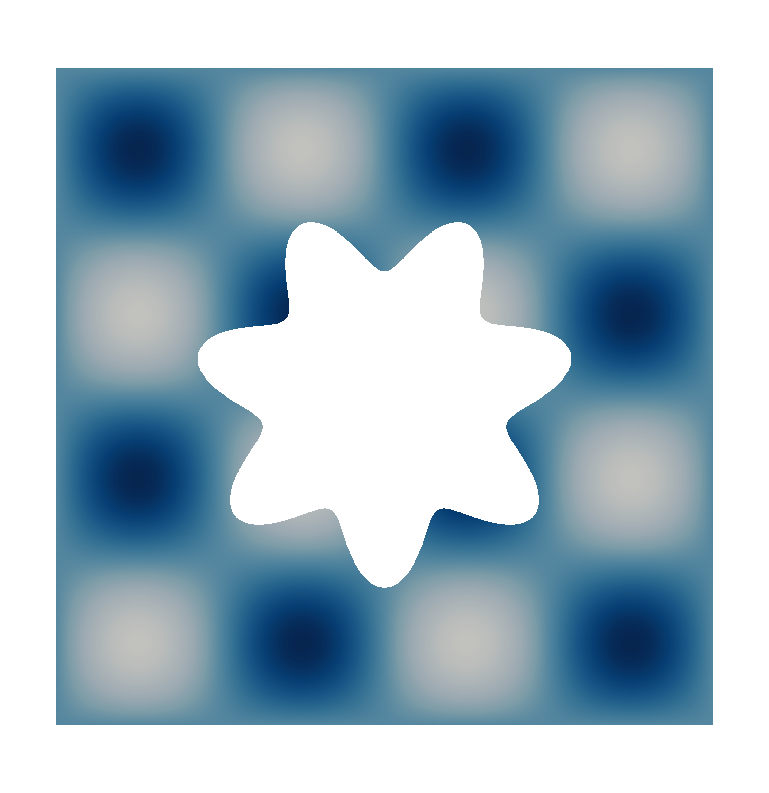}}
	\end{tabular}
        \caption{True boundary $\Gamma$(blue), surrogate boundary $\tilde{\Gamma}$(red) and $u^h$ on the surrogate domain $\tilde{\Omega}^h$.}
        \label{fig:visualsq}
\end{figure}

 \begin{figure}[tbh!]
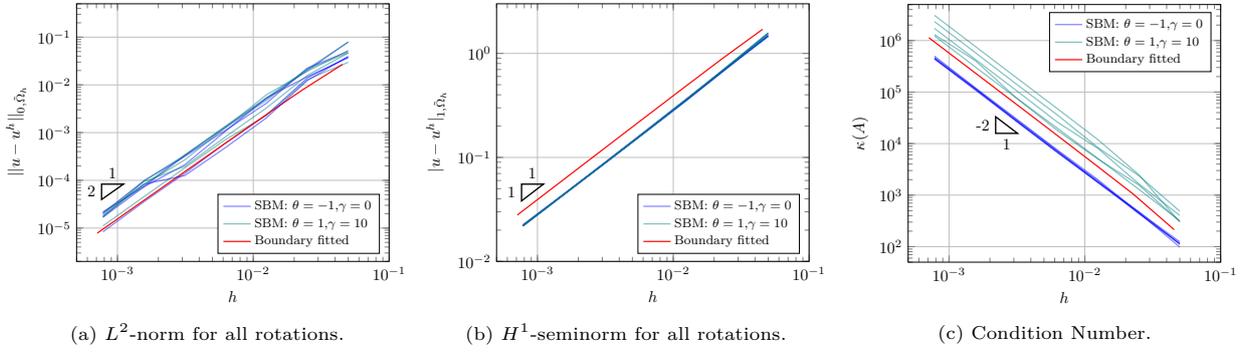

	\begin{subfigure}[tbh!]{0.33\textwidth}
		%
                \caption{$L^2$-norm for all rotations.}
	\end{subfigure}
        \begin{subfigure}[tbh!]{0.33\textwidth}\centering
		%
                \caption{$H^1$-seminorm for all rotations.}
	\end{subfigure}
        \begin{subfigure}[tbh!]{0.33\textwidth}\centering
		%
                \caption{Condition Number.}
	\end{subfigure}

        \caption{Convergence rates and condition numbers for the Poisson problem with a square boundary over Cartesian grids. Because the results for body-fitted grids are typically clustered more tightly, we do not perform grid rotations for this case.}
        \label{fig:SqPoissonResults}
\end{figure}
\begin{figure}[tbh!]
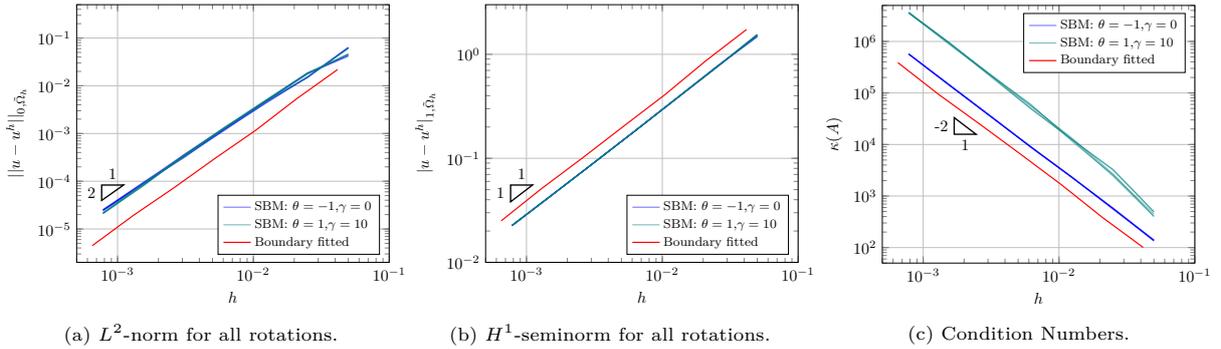

	\centering
	\begin{subfigure}[tbh!]{0.32\textwidth}
		%
                \caption{$L^2$-norm for all rotations.}
	\end{subfigure}
        \begin{subfigure}[tbh!]{0.32\textwidth}\centering
		%
                \caption{$H^1$-seminorm for all rotations.}
	\end{subfigure}
        \begin{subfigure}[tbh!]{0.32\textwidth}\centering
		%
                \caption{Condition Numbers.}
	\end{subfigure}

        \caption{Convergence rates and condition numbers for the Poisson problem with a flower-like boundary over Cartesian grids.}
        \label{fig:FlPoissonResults}
\end{figure}

\subsection{Linear Elasticity}

We consider a series of tests for the isotropic compressible linear elasticity equations that have similar setup as in the case of the Poisson equation. The circle and star geometry were immersed on triangular grids (unstructured and structured) with a prescribed manufactured solution. The elastic parameters were chosen to be a Young's Modulus $E = 10$ Gpa and a Poisson's ratio of $\nu = 0.3$. Neumann boundary conditions were applied on the embedded inner boundary and Dirichlet conditions were strongly enforced on the outer boundary. As before, the computational grids were incrementally rotated from 0 degrees to 45 degrees for seven levels of grid refinement.
The analytical solution was chosen to be
\begin{equation}
\left\{ \begin{aligned} 
  u_x(x,y) &= \sin(2 \pi x) \sin(2 \pi y) \; , \\
  u_y(x,y) &= \cos(2 \pi x) \cos(2 \pi y) \; .
\end{aligned} \right.
\end{equation}
\begin{figure}[tbh!]
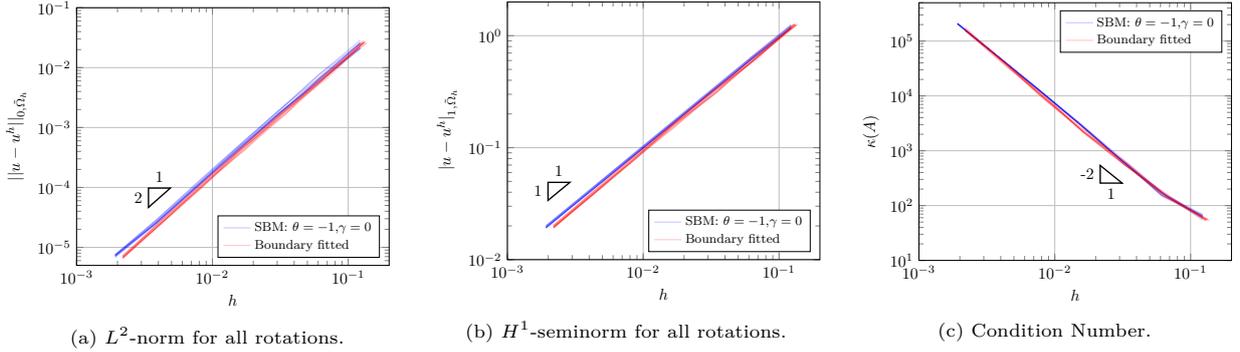

	\begin{subfigure}[tbh!]{0.33\textwidth}

                \caption{$L^2$-norm for all rotations.}
	\end{subfigure}
        \begin{subfigure}[tbh!]{0.33\textwidth}\centering
		%
                \caption{$H^1$-seminorm for all rotations.}
	\end{subfigure}
        \begin{subfigure}[tbh!]{0.33\textwidth}\centering
		%
                \caption{Condition Number.}
	\end{subfigure}
        \caption{Convergence rates for the linear elasticity problem with circular boundary and triangular grids.}
        \label{fig:H0results2}
\end{figure}
\begin{figure}[tbh!]
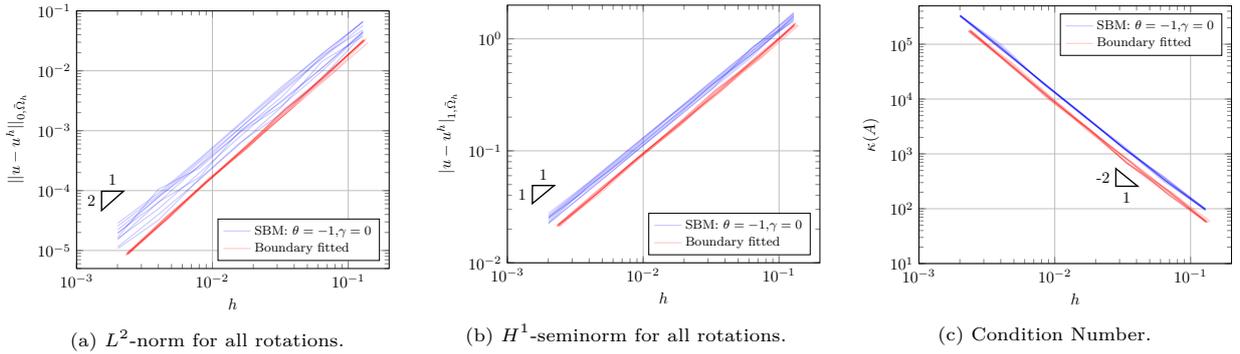

	\begin{subfigure}[tbh!]{0.33\textwidth}
		%
                \caption{$L^2$-norm for all rotations.}
	\end{subfigure}
        \begin{subfigure}[tbh!]{0.33\textwidth}\centering
		%
                \caption{$H^1$-seminorm for all rotations.}
	\end{subfigure}
        \begin{subfigure}[tbh!]{0.33\textwidth}\centering
		%
                \caption{Condition Number.}
	\end{subfigure}

        \caption{Convergence rates for the linear elasticity problem with star-shaped boundary and triangular grids.}
        \label{fig:H1results2}
\end{figure}
Figure~\ref{fig:H0results2} and~\ref{fig:H1results2} show that optimal convergence rates are obtained in both the $L^2$-norm and $H^1$-seminorm of the error. 
Also the condition number is well behaved and does not show any small-cut cell pathologies.

Analogously, convergence tests were performed based on quadrilateral finite elements. The same square and flower geometries were considered with a manufactured solution
\begin{equation}
\left\{ \begin{aligned} 
  u_x(x,y) &= \sin(3 \pi x) \sin(3 \pi y) \; , \\
  u_y(x,y) &= \cos(3 \pi x) \cos(3 \pi y) \; .
\end{aligned} \right.
\end{equation}
The material properties are: Young's modulus $E = 2.25$ GPa and Poisson's ratio $\nu = 0.125$.
The boundary conditions are kept the same as in the Poisson experiments, that is Neumann and Dirichlet conditions are applied with the same scheme, although this time they involve vector quantities like displacement and traction rather than scalar quantities like temperature and normal heat flux.
Figure~\ref{fig:SqElasticityResults} and~\ref{fig:FlElasticityResults} show that optimal convergence rates are also obtained with Cartesian grids of quadrilateral finite elements.

\begin{figure}[tbh!]
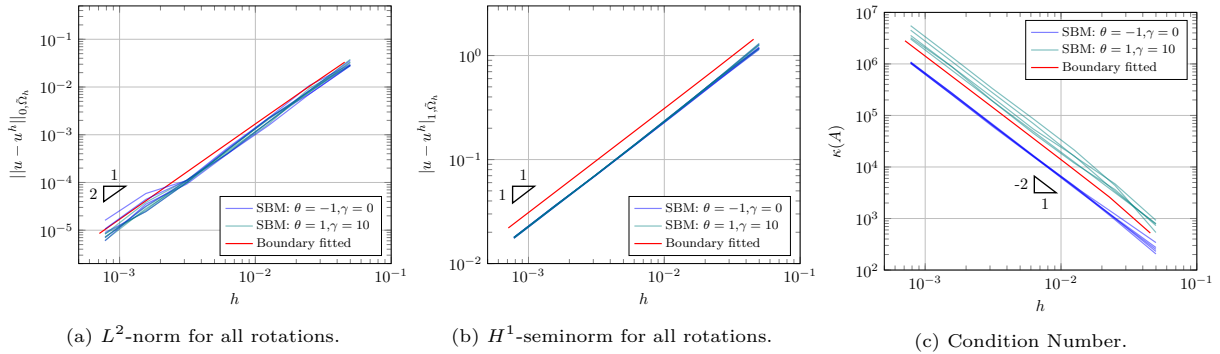

	\centering
	\begin{subfigure}[tbh!]{0.32\textwidth}

                \caption{$L^2$-norm for all rotations.}
	\end{subfigure}
        \begin{subfigure}[tbh!]{0.32\textwidth}\centering
		%
                \caption{$H^1$-seminorm for all rotations.}
	\end{subfigure}
        \begin{subfigure}[tbh!]{0.32\textwidth}\centering
		%
                \caption{Condition Number.}
	\end{subfigure}

        \caption{Convergence rates for the linear elasticity problem with square boundary and Cartesian grids.}
        \label{fig:SqElasticityResults}
\end{figure}

\begin{figure}[tbh!]
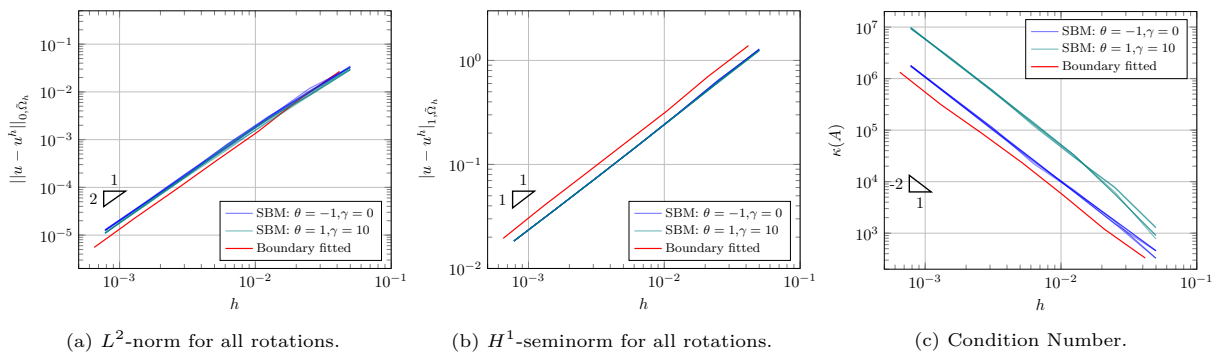

	\centering
	\begin{subfigure}[tbh!]{0.32\textwidth}
		%
\caption{$L^2$-norm for all rotations.}
	\end{subfigure}
        \begin{subfigure}[tbh!]{0.32\textwidth}\centering
		%
                \caption{$H^1$-seminorm for all rotations.}
	\end{subfigure}
        \begin{subfigure}[tbh!]{0.32\textwidth}\centering
		%
                \caption{Condition Number.}
	\end{subfigure}

        \caption{Convergence rates for the linear elasticity problem with flower-like boundary and Cartesian grids.}
        \label{fig:FlElasticityResults}
\end{figure}

\subsection{Cantilever beam}
The cantilever beam test is a classic elastostatics problem consisting of a loaded beam that is clamped on one end subject to a distributed load.
This test is a good candidate for assessing the performance of the proposed variant of SBM in the presence of mixed displacement/traction boundary conditions.
An analytical solution can be derived using the Euler-Bernoulli beam theory.
The beam has length $L=20$ and height $H=1$ and is subject to a uniformly distributed load $q = 1e{-3}$.
A zero displacement Dirichlet boundary condition is applied at $x=0$, with a stress-free Neumann boundary condition applied everywhere else.
The material Young's modulus is $E=1e5$ and the Poisson's ratio is $\nu=0.3$. The analytical solution for the vertical tip displacement is calculated as
\begin{equation}
u_{ymax} = \frac{qL^4}{8EI} \; .
\end{equation}

Simulations were performed with both the embedded (non-symmetric Nitsche) SBM and compared with a boundary-fitted standard primal formulation.

Visualizations of the embedded beam setup and displacement solution are shown in Figure \ref{fig:beam}. The convergence of the solution (largest vertical displacement) for both the SBM and boundary-fitted formulations are provided in Figure \ref{fig:beamsoln}. The Euler-Bernoulli reference solution is 0.24 indicated by the solid black line. The SBM (blue) shows proper convergence to the reference solution with sufficient refinement.
Actually, we observe that the SBM formulation converges faster than the primal body-fitted formulation to the reference solution.

\begin{figure}[tbh!]
  \centering
    \begin{subfigure}[t]{1\textwidth}\centering
    \includegraphics[width=\linewidth]{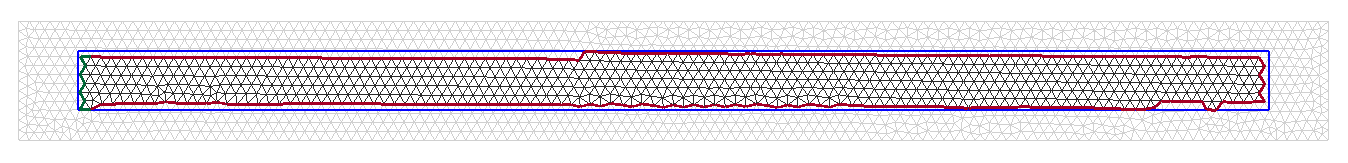}\vspace{3mm}
    \caption{Setup: the true boundary (blue), surrogate Neumann boudnary (red), and surrogate Dirichlet boundary (green).}

  \end{subfigure}\\
  \begin{subfigure}[t]{1\textwidth}\centering
    \includegraphics[width=\linewidth]{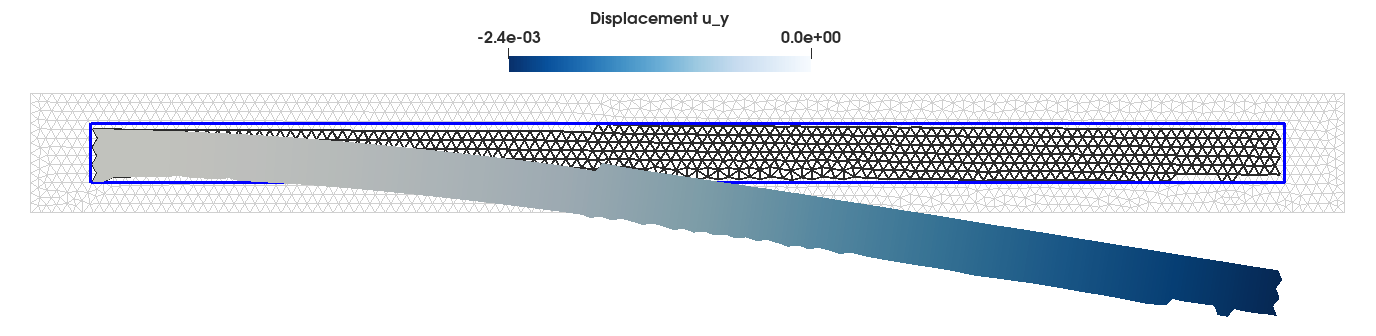}\vspace{3mm}
    \caption{Contour of vertical displacement and deformed configuration (displacement are amplified by $10^3$).}

  \end{subfigure}

\caption{Cantilever beam bending test: Surrogate domain $\tO$ with true $\G$ (blue) and surrogate boundaries $\tG$ (green, red), and deformed configuration.}
\label{fig:beam}

\end{figure}

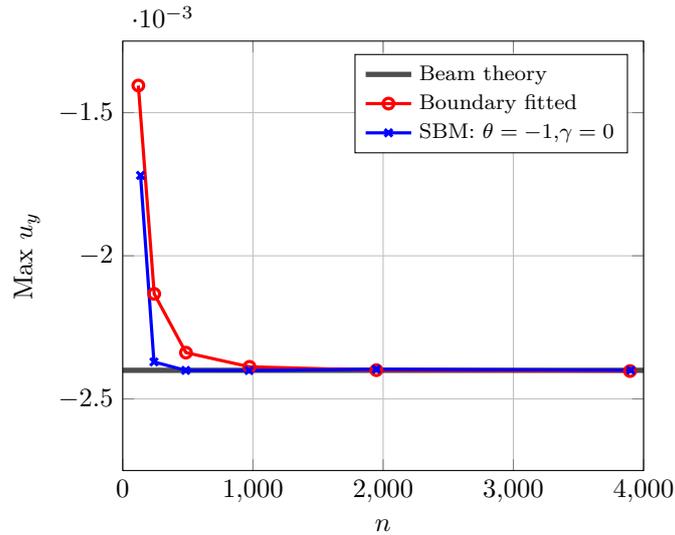
\begin{figure}[tbh!]
  \centering

		\begin{tikzpicture}[scale= 1]
			\begin{axis}[
				grid=major,x post scale=1.0,y post scale =1.0,
				xlabel={$n$},ylabel={Max $u_y$},xmax=4000, xmin=0, ymax=-0.00125, ymin=-0.00275,
				legend style={font=\footnotesize,legend cell align=left},
				legend pos= north east]
                            \addplot[ black, opacity=0.7,line width=2pt] coordinates {
                            (0 , -0.0024 )
                            (4000 , -0.0024 )
                          };
						  \addlegendentry{Beam theory}
                          \addplot[ red, opacity=1,line width=1.2pt,mark=o] coordinates {
                            (120 , -0.001405168522458252 )
                            (242 , -0.0021328711694979237 )
                            (486 , -0.002338261167716788 )
                            (974 , -0.002387720924460818 )
                            (1948 ,-0.002399547929445104 )
                            (3896 , -0.0024025627946877765 )
                          };
						  \addlegendentry{Boundary fitted}
                          \addplot[ blue, opacity=1,line width=1.2pt,mark=x] coordinates {
                            (137 , -0.0017190605631823934 )
                            (240 , -0.0023706685172972257)
                            (484 , -0.0024001597973039284 )
                            (970 , -0.002400702412107534 )
                            (1948 , -0.002396300937510495 )
                            (3902 , -0.0023998906847161904 )
                          };
						  \addlegendentry{SBM: $\theta = -1$,$\gamma = 0$}
			\end{axis}
		\end{tikzpicture}\vspace{3mm}
        \caption{Cantilever beam bending test: tip vertical deflection for the Boundary fitted method (red) and SBM (blue). The reference solution of the Euler Beam Theory is in black. $n$ is the number of boundary segments.}
\label{fig:beamsoln}

\end{figure}

\section{Summary \label{sec:summary} }
We proposed a new conceptualization of the SBM framework, which provides optimal error estimates
in both the $H^1$-norm and $L^2$-norm in the presence of Neumann or Dirichlet boundary conditions. The
proposed approach is based on approximate integration of the variational formulation in the gap between
the surrogate and true boundaries. The proposed approach is still classified as a SBM, since the construction
of the integration procedure on the gap relies on the concept of a distance, and because extensions of the
solution in the form of Taylor expansions are used to evaluate the solution in the gap. Hence, no cut-cell
integration procedure is performed. A series of numerical experiments proved the consistency, stability,
robustness, and optimal accuracy of the proposed approach.

\section*{Acknowledgments}
G. Scovazzi has been partially supported by the National Science Foundation (Division of Mathematical Sciences), with Grant DMS 2207164 and Grant DMS 2409919.

%\section*{References}
%\bibliographystyle{model1-num-names.bst}
\bibliographystyle{plain}
\bibliography{./SBM_bibliography} 

\end{document}